\documentclass[10pt,reqno]{amsart}
\usepackage{amsmath}
\usepackage{amssymb, latexsym, amsfonts, amscd, amsthm, mathrsfs, enumerate, esint}
\usepackage[usenames,dvipsnames]{color}
\usepackage{amssymb}
\usepackage{comment}

\numberwithin{equation}{section}

\newcommand{\red}[1]{#1}

\definecolor{purple}{rgb}{0.9,0,0.8}
\newcommand{\purple}[1]{#1}
\definecolor{gray}{rgb}{0.7,0.7,0.7}

\newcommand{\abbr}[1]{{\sc\lowercase{#1}}}

\global\long\def\Es{E_{\star}}%
\global\long\def\GS{\mathsf{GS}}%
\global\long\def\V{{\rm Vol}}%
\global\long\def\bsigma{\boldsymbol{\sigma}}%
\global\long\def\bx{\mathbf{x}}%
\global\long\def\be{\mathbf{e}}%
\global\long\def\bxs{\mathbf{x}_{\star}}%
\global\long\def\bxo{\mathbf{x}_0}
\global\long\def\hbxs{\hat{\mathbf{x}}_{\star}}
\global\long\def\qs{q_{\star}}%
\global\long\def\bz{\mathbf{z}}%
\global\long\def\by{\mathbf{y}}%
\global\long\def\sfv{{\mathsf{v}}}%
\global\long\def\bv{\mathbf{v}}%
\global\long\def\bu{\mathbf{u}}%
\global\long\def\bn{\mathbf{n}}%
\global\long\def\bJ{\mathbf{J}}%
\global\long\def\indic{\mathbb{I}}%
\global\long\def\SN{\mathbb{S}^{N-1}}%
\global\long\def\sp{{\rm sp}}%
\global\long\def\E{\mathbb{E}}%
\global\long\def\P{\mathbb{P}}%
\global\long\def\R{\mathbb{R}}%
\global\long\def\V{\mathbb{V}}%
\global\long\def\cpt{\mathsf{CP}}%
\global\long\def\cptq{\mathsf{CP}_{\qs}}%
\global\long\def\cpto{\cpt_{0}}
\global\long\def\cpteq{\cpt_{\qs}^{\rm e}}
\global\long\def\cpteqm{\cpt_{\qs}^{{\rm e},[m]}}
\global\long\def\cpteo{\cpt_{0}^{\rm e}}

\global\long\def\epsilon{\varepsilon} 
\global\long\def\b{\beta}%
\global\long\def\BJ{\mathbf{J}}%

\def\cpt{\mathsf{CP}}
\def\err{{\sf Err}}

\def\fdt{{\rm fdt}}

\def\b{{\beta}}
\def\SN{{\mathbb S}_N}
\def\qs{q_\star}
\def\qdyn{q_{\sf d}}
\def\SNqs{\qs \SN}
\def\bJ{{\bf J}}
\def\R{{\mathbb R}}
\def\P{{\mathbb P}}
\def\E{{\mathbb E}}

\def\BO{{\bf O}}
\def\BB{{\bf B}}
\def\BZ{{\bf Z}}

\def\BB{{\bf B}}
\def\BZ{{\bf Z}}

\def\BJ{{\bf J}}

\def\BJ{{\bf J}}

\def\bsigma{{\boldsymbol \sigma}}

\def\Bx{{\bf x}}
\def\bx{{\bf x}}

\def\by{{\bf y}}

\def\bx{\Bx}
\def\bn{\Bx_\star}

\def\Es{{E_\star}}

\def\Gs{{G_\star}}

\def\qinf{\alpha \qs}

\def\Aa{{\mathcal A}}

\def\Ca{{\mathcal C}}

\def\Ea{{\mathcal E}}
\def\Fa{{\mathcal F}}

\def\Ha{{\mathcal H}}

\def\La{{\mathcal L}}

\def\Pa{{\mathcal P}}

\def\Ua{{\mathcal U}}

\def\b{{\beta}}

\def\bJ{{\bf J}}
\def\R{{\Bbb R}}
\def\B{{\Bbb B}}
\def\P{{\Bbb P}}
\def\E{{\Bbb E}}

\def\BB{{\bf B}}
\def\BZ{{\bf Z}}

\def\BB{{\bf B}}
\def\BW{{\bf W}}
\def\BZ{{\bf Z}}

\def\BJ{{\bf J}}

\def\BG{{\bf G}}

\def\BJ{{\bf J}}

\def\Bx{{\bf x}}
\def\bx{{\bf x}}

\def\bv{{\mathsf v}}

\def\by{{\bf y}}

\def\bx{\Bx}

\def\Aa{{\mathcal A}}

\def\Ca{{\mathcal C}}

\def\Ea{{\mathcal E}}
\def\Fa{{\mathcal F}}

\def\Ha{{\mathcal H}}

\def\La{{\mathcal L}}

\def\Pa{{\mathcal P}}

\def\Ua{{\mathcal U}}

\def\C{{\Bbb C}}
\def\R{{\Bbb R}}

\def\E{{\Bbb E}}

\def\Var{{\rm Var}}
\def\Cov{{\rm Cov}}

\def\b{\beta}

\def\D{\Delta}

\def\part{\partial}

\def\ts{\times}

\def\tilde{\widetilde}

\cleardoublepage
\newtheorem{prop}{Proposition}[section]

\newtheorem{remark}[prop]{Remark}
\newtheorem{lem}[prop]{Lemma}

\newtheorem{defi}[prop]{Definition}

\newtheorem{cor}[prop]{Corollary}
\newtheorem{theo}[prop]{Theorem}

\def\HJs{H^{\star}_{\BJ}}
\def\HJsm{H^{\star,[m]}_{\BJ}}

\def\bU{{\sf U}}
\def\C{{\Bbb C}}
\def\R{{\Bbb R}}

\def\E{{\Bbb E}}

\def\Var{{\rm Var}}

\def\b{\beta}

\def\V{\vec V}
\def\Ve{{\vec V}_{\rm e}}
\def\vB{\vec B}
\def\Eg{E_\star}
\def\Egp{E'_\star}
\def\Gg{G_\star}
\def\Ggp{G'_\star}
\def\Ep{E}
\def\Egs{\Eg}

\def\Ggs{\Gg}
\def\Epp{E'}
\def\qo{q_o}
\def\qop{q'_o}
\def\bUqf{\bU^f_{\qs}}
\def\bUqtf{\bU^{\widetilde{f}}_{\qs}}
\def\bUof{\bU^f_{0}}
\def\bUsp{\bU^{\rm sp}_{\qs}}
\def\bUosp{\bU^{\rm sp}_{0}}

\def\D{\Delta}

\def\part{\partial}

\def\ts{\times}

\def\lak{\langle k \rangle}

\def\tilde{\widetilde}

\newcommand{\HN}{\widetilde{H}_N}
\newcommand{\phiN}{\widetilde{\varphi}_N}

\cleardoublepage
\begin{document}              

\title[Spherical spin glass dynamics]
{Dynamics for spherical spin glasses: \\
% Relaxion for 
Gibbs distributed initial conditions}
\author{Amir Dembo}
\address{Department of Statistics and Department of Mathematics\\
Stanford University\\ Stanford, CA 94305.}
\email{adembo@stanford.edu}

\author{Eliran Subag}
\address{Mathematics Department, Weizmann Institute of Science\\
Herzl St 234,  PO Box 26,  Rehovot 7610001,  Israel}
\email{eliran.subag@gmail.com}

\thanks{
\newline
{\bf AMS (2020) Subject Classification:}
Primary: 82C44 Secondary:  82C31, 60H10, 60K35.
}

\date{July 8,  2026}

\begin{abstract}
We derive the coupled non-linear integro-differential equations for 
the thermodynamic limit of the empirical correlation and response functions 
in the Langevin dynamics at temperature $T$,  for spherical mixed $p$-spin 
disordered mean-field models,  \emph{initialized according to a Gibbs measure
for temperature $T_0$},  in the replica-symmetric (\abbr{rs}) or $1$-replica-symmetry-breaking (\abbr{rsb}) 
phase.  For any $T_0=T$ above the phase transition point,  the resulting stationary relaxation dynamics 
coincide with the \abbr{FDT}-solution for these equations,  while for lower
$T_0=T$ in the $1$-\abbr{rsb} phase,  the relaxation dynamics coincides with the 
\abbr{FDT}-solution,  now concentrated on the single spherical band within the Gibbs measure's
support on which the initial point lies.  
\end{abstract}

\maketitle
%%%%%%%%%%%%%%%%%%%%%%%%%%%%%%%%%%%%%%%%%%%%%%%%%%%%
%%%%%%%%%%%%%%%%%%%%%%%%%%%%%%%%%%%%%%%%%%%%%%%%%%%%
\section{Introduction\label{s.introduction}}

The thermodynamic limits of a wide class of Markovian dynamics with random interactions,
exhibit complex long time behavior, which is of much interest in out of equilibrium statistical physics 
(c.f. the surveys 
\cite{BB,BKM,LesHouches} and the references therein). 
This work is about the thermodynamic ($N \to \infty$),  behavior 
at times $t$ which \emph{do not grow with $N$},  for 
%systems of  
Langevin particles $\bx_t=(x_t^i)_{1\le i\le N}\in\R^N$, 
interacting with each other through random \red{mixed $p$-spin} 
Hamiltonians $H_{\BJ}$.  Specifically,  \red{fix} $1 < r_\star < \bar r \le \infty$ \red{and the corresponding} open balls
\[
\B^N_\star := \B^N(r_\star \sqrt{N}) \quad \mbox{and} \quad  \B^N := \B^N (\bar r \sqrt{N}) \subseteq \R^N
\]
of radii \red{$r_\star \sqrt{N}$ and} $\bar r \sqrt{N}$ (so $\B^N=\R^N$ if $\bar r =\infty$).  \red{To any} $b_p \ge 0$
%not identically zero,  
decaying fast enough so that   
\begin{equation}\label{eq:r-star}
\limsup_{p \to \infty} \frac{1}{p} \log b_p \le  - \log \bar r \,,
\end{equation}
\red{correspond the mixed $p$-spin} centered Gaussian fields $H_\BJ:\B^N \to \R$ of positive definite covariance
\begin{equation}\label{eq:nudef}
\Cov\big(H_{\BJ}(\bx),H_{\BJ}(\by)\big) 
= N \nu \big( N^{-1} \langle \bx,\by \rangle \big)\,, \qquad
\nu(r):=\sum_{p \ge 2} b_p^2 r^p \,.
\end{equation}
We note in passing that the case of $\nu(t)=t^p$ is called \emph{pure $p$-spin} and more generally $\bar r=\infty$ 
for any \emph{finite mixture} (that is,  whenever $\nu(\cdot)$ is a polynomial function),  while 
a mixture $\nu(\cdot)$ is called \emph{generic} if 
	\begin{equation*}
		\sum_{p\,\text{odd}}p^{-1} \indic \{b_{p}>0\}=\sum_{p\,\text{even}}p^{-1} \indic \{b_p>0\}=\infty \,.
		% \label{eq:generic}
	\end{equation*}
\red{Having} $\nu(\cdot)$ real analytic on $(-\bar r^2,\bar r^2)$,  
\red{the field $H_\BJ$ of covariance \eqref{eq:nudef} 
is in} the collection $C^2_0(\B^N)$ of twice continuous differentiable
$\varphi_N:\B^N \to \R$,  such that $\varphi_N({\bf 0})=0$ and $\nabla \varphi_N({\bf 0})={\bf 0}$,
\red{with $H_\BJ$ further realizable as}
\begin{equation}\label{potential}
  H_\BJ(\bx)=\sum_{p\ge 2} b_p H^{(p)}_{\BJ} (\bx) \,, 
 \qquad
H^{(p)}_{\BJ} (\bx) := \sum_ {1\le i_1 \le  \cdots \le i_p\le N} J_{(i_1,\ldots, i_p)}x^{i_1}\cdots x^{i_p},
\end{equation}
%b_p \sqrt{p!} = a_p of \cite{DGM}
for independent centered Gaussian coupling constants\footnote{coinciding with \cite[(1.3)-(1.4)]{DS21}, 
where all coupling constants for each unordered $\{i_1,\ldots,i_p\}$ have been lumped together}
$\BJ=\{J_{(i_1,\ldots, i_p)}\}$,  such that 
%\begin{equation*}%\label{eq:vardef}
$\Var(J_{(i_1,\ldots, i_p)}) = N^{-(p-1)}$.
%\frac{p!}{\prod_k l_k!} \,.
%\end{equation*}
%where $(l_1,l_2,\ldots)$ are
%the multiplicities of the different elements of the set      
%$\{i_1,\ldots,i_p\}$ (so having $i_1 \neq i_2 \cdots \neq i_p$ 
%yields variance larger by a factor $p!$ from 
%the variance in case $i_1=i_2=\cdots = i_p$).
\red{Now,  fixing $\b>0$} and $f'(\cdot)$ locally Lipschitz on $[0,r_\star^2)$,  such that 
\begin{equation}
f(r) := f_\ell(r) := \ell(r-1)^2 + f_0(r)  \,,  \qquad \qquad
\label{eq:fdef}
%\label{eq:f-tail}
\lim_{r \uparrow r^2_\star} \big\{ (r^2_\star - r) \,  f_0'(r)  \big\} = \infty \,,
\end{equation}  
consider for $N$-dimensional Brownian motion $\BB_t$,  independent of $H_{\BJ}$
\red{and} 
$\bxo \in \B^N_\star$, 
the 
%unique strong 
solution $\{ \bx_t, t \ge 0 \}$ of 
\begin{equation}\label{diffusion}
 \bx_t=\bxo-\int_0^t f'(\vert\vert \bx_u\vert\vert^2/N)\bx_u du
  -\beta \int_0^t \nabla H_\BJ(\bx_u) du + \BB_t,
  \end{equation}
where $\vert \vert \cdot \vert \vert$ denotes the Euclidean norm on $\R^N$.  
\red{As} specified in Corollary \ref{cor:ex},  assuming \red{hereafter} \eqref{eq:r-star} and
\eqref{eq:fdef} guarantees the existence of unique strong solutions of \eqref{diffusion}
in $\Ca (\R^+,\R^N)$,  that for a.e.  path $t \mapsto \Bx_t$ are confined to $\B^N_\star$,  
with the dynamics 
%\eqref{diffusion} 
then reversible \abbr{wrt} $\mu^N_{2\b,\BJ}$.  Here $\mu^N_{\b,\BJ}$ denotes the 
random Gibbs measure of density 
\begin{equation}\label{Gibbs}
\frac{d \mu^N_{\b,\BJ}}{d\Bx} = Z_{\b,\BJ}^{-1}
e^{-\b H_{\BJ} (\Bx) - N f(N^{-1} \|\Bx\|^2)} 
\end{equation}
with respect to Lebesgue measure on $\B^N_\star$ and 
\begin{equation}\label{def:ZbeJ}
Z_{\b,\BJ} := \int_{\B^N_\star} 
e^{-\b H_{\BJ} (\Bx) -N f(N^{-1} \|\Bx\|^2)} d\Bx < \infty \,.
\end{equation}
Further,  $e^{-N f_\ell(r)}$ approximates at $\ell \gg 1$
%as $\ell \to \infty$,  
the indicator on $r=1$,  
effectively restricting $\{ \bx_t,  t >  0\}$ to the sphere 
$\SN := {\mathbb S}^{N-1}({\sqrt{N}})$ of radius  $\sqrt{N}$ with the 
corresponding  spherical mixed $p$-spin Gibbs measures 
% $\tilde{\mu}^N_{2\b,\BJ}$ 
%on $\SN$,  
of density 
\begin{equation}\label{eq:Gibbs-sph}
\frac{d\tilde{\mu}^N_{\b,\BJ}}{d\Bx} = \tilde{Z}_{\b,\BJ}^{-1} e^{-\b H_{\BJ}(\bx)} \,,
\end{equation}
with respect to the uniform measure on $\SN$.  The spherical mixed $p$-spin model \eqref{eq:Gibbs-sph} 
has been
extensively studied in mathematics and physics over
the last three decades.  See for example
\cite{Chen,Talagrand}, for the rigorous analysis of the asymptotic of the (non-random) limiting quenched free energy
\[
F(\b) := \lim_{N \to \infty} \frac{1}{N} \log \tilde{Z}_{2\b,\BJ} \,.
\]
% for the measure with a hard spherical constraint of having $\vert\vert \bx\vert\vert^2=N$).
In particular,  recall from \cite{Talagrand},  that $F(\b)=2 \b^2 \nu(1)$   
matching the annealed free-energy for $\tilde{\mu}^N_{2\b,\BJ}$ 
as long as $\b \le \b_c^{\rm stat}$, for the positive and finite
\begin{equation}\label{def:bc-stat}
\b_c^{\rm stat} := \inf \{ \b \ge 0 : \sup_{x \in [0,1]} \{ g_\b(x) \} > 0 \} \,,
\end{equation}
where (see also \cite[(1.13)]{SubagRS}),
% dividing here by two!
\begin{equation}\label{def:g-beta}
g_\b(x) := g_\b(x;\nu) := 2 \b^2 \nu(x) + \frac{x}{2} + \frac{1}{2} \log (1-x) \,.
\end{equation}
At $\b > \b_c^{\rm stat}$ one has the Parisi formula $F(\b) = \min_\zeta \Pa_{\nu,\b}(\zeta)$,
where the minimum of the strictly convex Chrisanti-Sommers \cite{CS92} functional $\Pa_{\nu,\b}(\cdot)$ 
is taken over all distribution functions $\zeta(\cdot)$ such that $\zeta(\hat q)=1$ for some $\hat q<1$.  The Parisi
measure $\zeta_P$ is the unique minimizer of this functional,  with the maximal point of its support denoted by 
$q_{\rm EA}(\b) \in [0,1)$ (c.f. \cite[Prop. 2.1]{Talagrand}).  For any $\b \le \b_c^{\rm stat}$ the model $\nu(\cdot)$
is at the \abbr{rs} phase with $q_{\rm EA}(\b)=0$,  while for any $\b > \b_c^{\rm stat}$ and $k \ge 1$
it is at the k-\abbr{rsb} phase whenever 
\begin{equation}\label{eq:kRSB}
	\mbox{supp}(\zeta_P)= \{ 0=q_0<q_1<\cdots<q_k=q_{\rm EA}(\b) \} \,.
\end{equation}
In particular,  having $r \mapsto 1/\sqrt{\nu''(r)}$ convex on $[0,1]$,  guarantees 
%being at
the $1$-\abbr{rsb} phase for all $\b>\b_c^{\rm stat}$ (see \cite[Prop.  2.2]{Talagrand} or 
\cite[Corollary 1.6]{JT2}).  Recall 
\purple{\cite[(1.19)]{DS24}} that such $\nu(\cdot)$ is {\em strictly} $k$-\abbr{rsb} if $\{q_j,  0 \le j \le k\}$
are the only maximizers in $[0,q_{\rm EA}(\b)]$ of 
%the function  
\begin{equation}\label{eq:phiP}
\phi_P(x) := 4 \b^2 \nu(x) - \int_0^x (x-u) \Big[\int_u^1 \zeta_P([0,r]) dr\Big]^{-2} du 
\end{equation}
(at $\b \le \b_c^{\rm stat}$ the model is replica symmetric,  with $\mbox{supp}(\zeta_P)=\{0\}$ 
and $\phi_P(x)=2 g_\b(x)$).

Large dimensional Langevin or Glauber dynamics often exhibit 
very different behavior at various time-scales (as functions 
of system size, c.f. \cite{ICM} and references therein). 
Following the physics literature (see
\cite{BKM,CHS,LesHouches,CK}), 
we study \eqref{diffusion}
% for the potential $H_{\BJ}(\bx)$ of \eqref{potential} 
at the shortest possible 
time-scale, where $N \to \infty$ first, holding $t \in [0,T]$.
% While it is too short to allow any escape from meta-stable states,
Considering the hard spherical constraint,  \cite{CHS,CK} 
predicted a rich picture for the 
limiting dynamics when initialized at infinite temperature ($\b_0=0$), namely
% out of equilibrium say
for $\bxo$ distributed uniformly over $\SN$.
Such limiting dynamics involve the coupled 
%(non-linear)
integro-differential equations relating the non-random 
limits 
$C(s,t)$ and 
\begin{equation}\label{eq:Rdef}
\chi(s,t):=\int_0^t R(s,u) du \,,
\end{equation}
as $N \to \infty$,  of the empirical covariance and integrated response functions,  that is  
\begin{align}\label{empiricalcovariance}
C_N(s,t)&:=\frac{1}{N} \langle \bx_s, \bx_t \rangle =
\frac{1}{N} \sum_{i=1}^N x_s^i x_t^i,
% \qquad  s\ge t
\\
\label{integrated}
%\label{eq:chidef}
 {\chi}_N(s,t)&:= \frac{1}{N} \langle \bx_s,\BB_t \rangle
 = \frac{1}{N}\sum_{i=1}^N x_s^i B_t^i\,.
\end{align}
Specifically,  it is predicted that for $\b_0=0$ and large $\b$ the 
asymptotic of $C(s,t)$ strongly depends on the way 
$t$ and $s$ tend to infinity, exhibiting 
{\it aging} behavior (where the older it gets, the longer the system takes to forget its current state,
see e.g. \cite{CK,Alice}).  A detailed analysis of such aging properties is given in \cite{2001}
for pure $2$-spins  
(noting that $\{J_{\{ij\}}\}$ form the \abbr{GOE} random matrix, 
whose semi-circle limiting spectral measure determines the 
asymptotic of $C(s,t)$).  More generally,  
\cite[Thm.~1.2]{BDG2} provides a rigorous derivation of closed equations for $C$ and $R$ 
in case of finite mixtures,  with $f'_0(r)=c r^{2k}$ for $k$ large enough and 
$\bxo$ independent of $\BJ$ (subject to mild moments condition on $N^{-1} \| \bxo\|^2$ 
%$N \mapsto \E [ C_N(0,0)^k ]$ is uniformly bounded for each fixed $k<\infty$,
%\begin{equation}\label{eq:x0cond}
%\lim_{N \to \infty} \E C_N(0,0) = C(0,0) \,,
%\end{equation}
%exists and the tail probabilities 
%$\P(|C_N(0,0)-C(0,0)|>x)$ decay exponentially fast in $N$.
and the concentration of measure \cite[Hypothesis 1.1]{BDG2} for the law of $\bxo$).
Further,   \cite[Prop. 1.1]{DGM} shows that for $\nu(r)=\frac{1}{8} r^p$ and 
such $f_0$,  in the limit $\ell \to \infty$,  
the equations of \cite{BDG2} for $(C,R)$ and $f_\ell(\cdot)$ of \eqref{eq:fdef}
%\begin{align}
%f_\ell(r) &:= \ell(r-1)^2 + \frac{\bphi}{4k} r^{2k} \,,
%\label{eq:fdef}
%\end{align}
%with $\bphi>0$ and integer $k>m/4$, 
coincide 
with the \abbr{CKCHS}-equations of Cugliandolo-Kurchan \cite{CK}  
%(who consider instead $C(2\cdot,2\cdot)$ and $R(2\cdot,2\cdot)$), 
and Crisanti-Horner-Sommers \cite{CHS}.  
Denoting by ${\sf P}_{\bx^\perp}$ the projection matrix onto the orthogonal complement of $\bx$ and by
$\nabla_{\rm sp} H_\BJ(\bx)={\sf P}_{\bx^\perp} \nabla H_\BJ(\bx)$ the Riemannian gradient on $\SN$, 
these \abbr{CKCHS}-equations are for $\b_0=0$ (uniform) 
initialization of the Langevin dynamics 
%of $\bx_t$ 
on $\SN$, 
\begin{equation}\label{sphere-diffusion}
 \bx_t=\bxo-\beta \int_0^t \nabla_{\rm sp} H_\BJ(\bx_u) du
 - \frac{N-1}{2N} \int_0^t \bx_u du + \int_0^t {\sf P}_{\bx_u^\perp} d \BB_u,
  \end{equation}
which are reversible for the corresponding 
%(pure) spherical $p$-spin 
Gibbs measures $\tilde{\mu}^N_{2\b,\BJ}$ of \eqref{eq:Gibbs-sph}.

In this paper we derive 
the corresponding limiting equations as $N \to \infty$ followed by $\ell \to \infty$
for mixtures satisfying \eqref{eq:r-star} and $f_\ell$ as in \eqref{eq:fdef},  now \emph{starting the diffusion \eqref{diffusion}
at $\bxo$ drawn from the Gibbs measure} $\widetilde{\mu}^N_{2\b_0,\BJ}$ of \eqref{eq:Gibbs-sph},  
possibly with $\b_0 \ne \b$.  
\red{The convergence in probability of functionals of such  $\bxo$ is related
in \cite{DS24} to their convergence for a fixed,  non-random $\bxo$,  under suitable conditioning of 
the Gaussian field $H_{\BJ}$.  Specifically,  in} view of \purple{\cite[Thm. 1.2]{DS24}, } this requires us to condition on the 
value of $H_{\BJ}(\bxo)$,
whereas by \cite[Thm. 1.4]{DS24},  if a generic $\nu(\cdot)$ is strictly $k$-\abbr{rsb}
at $\beta_0 > \beta_c^{\rm stat}$,  one needs to further condition on the values of the potential 
and its radial derivative,  at $k$ critical points of $H_{\bJ}(\cdot)$ of certain specified overlaps
(among themselves and with $\bxo$).  For simplicity we restrict the latter case to $k=1$,  where one
conditions for $\qs \in (0,1]$,  $|\qo| \le \qs$,  $|\qo|<1$,  $\bxo \in \SN$,  $\bxs \in \qs \SN$ and  $\V:=(\Ep,\Eg,\Gg,\qo)$,
on events of the form
\begin{align}\label{eq:cond-J}
\cptq (\V) :=\Big\{H_{\BJ}(\bxo) = - N \Ep \,, H_{\BJ}(\bxs)=-N \Eg \,, \nabla H_{\BJ}(\bxs)= - \Gg \bxs \Big\},  
\quad \hbox{where} \quad
\frac{1}{N} \langle \bxo,\bxs \rangle = \qo \,.
\end{align}
\red{In \cite{DS21} we generalized \cite{BDG2},  to get the limiting equations for $(C_N,\chi_N)$ and
$\bxo$ uniformly chosen on a band
\begin{equation}
\label{eq:subsphere}
{\Bbb S}_{\bxs} (q) :=
\Big\{\Bx \in \SN : \; \frac{1}{N}\langle \Bx,\bxs \rangle=q \Big\}  \qquad {\rm for} \;\; {\rm some} \quad q \in [-\qs, \qs]\,,
\end{equation}
which is centered at some critical point $\bxs$ specified as in \eqref{eq:cond-J}.  While given $\cptq(\V)$,  the point 
$\bxo$ is no longer uniformly distributed on ${\Bbb S}_{\bxs} (\qo)$,  nevertheless,  as in \cite{DS21},  it still} suffices 
to keep track of $(C_N,\chi_N)$ and the dynamics of the angle with the critical point
\begin{equation}\label{eq:bnq} 
q_N (s) :=\frac{1}{N}\langle \Bx_s,\bxs \rangle = \frac{1}{N} 
\sum_{i=1}^N x_s^i x_\star^i \,,
\end{equation} 
which under \eqref{eq:cond-J} starts at $q_N(0)=\qo$.  
\purple{Note that for} $\b_0 \le \b_c^{\rm stat}$
the model is replica-symmetric (\abbr{rs}),  \purple{which} amounts to setting $\qs=0$ (or alternatively $\bxs={\bf 0}$).
\purple{At $\qs=0$},  necessarily $\Es=\qo=0$ \purple{with} $\Gs$ irrelevant,  \purple{whereby}
% reduce to
$\V=(E,0,0,0)$ and  \purple{we condition on the event}
\begin{align}\label{eq:cond-HT}
\cpto (\V) :=\Big\{H_{\BJ}(\bxo) = - N \Ep  \Big\} \,.
\end{align}

Before stating our results,  we note that using the replica method,  for pure spherical $p$-spin models,
\cite{BBM} predicts the resulting limit equations for $(R_N,C_N)$ when starting the Langevin diffusion 
on $\SN$ at a sample $\bxo$ from $\tilde{\mu}^N_{2\b_0,\BJ}$,  for $\b_0 \in (\b_c^{\rm dyn},\b_c^{\rm stat})$
and possibly $\b \ne \b_0$ (see \eqref{def:bc-dyn} for the definition of $\b_c^{\rm dyn}$).
% (and their solution in the \abbr{fdt} regime). 
Building on it,  and using again the replica method,  \cite{BF} considers in this setting 
also the limit time-dynamics of the energy 
\begin{equation}\label{eq:Hn}
\HN (s) := - \frac{1}{N} H_{\BJ}(\Bx_s) \,,
\end{equation}
and of $q_N(s)$ of the form \eqref{eq:bnq},
taking as $\bxs$ an additional sample from $\tilde{\mu}^N_{2\b_0,\BJ}$ having a prescribed overlap with $\bxo$.
Further,  see \cite{Capone,FFRT2} for simulations and fine predictions of the out-of-equilibrium limit equations 
of Langevin diffusions initiated at $\b_0 \in (\b_c^{\rm dyn},\b_c^{\rm stat})$,  for certain family of 
non-pure models,  \cite{FFRT1,FZ} for such results in case of the 
$N$-dimensional Gradient flow (namely,  $\b=\infty$),  and \cite{Sun} for the corresponding evolution of states.
In contrast,  to the best of our knowledge \emph{there has not been even a physics prediction of
the low-temperature dynamics},  namely,  when starting the Langevin diffusion at $\bxo$ 
drawn from a Gibbs measure within the \abbr{rsb}-phase,  as established
rigorously in Theorem \ref{thm-uncond} (for generic strictly $1$-\abbr{rsb} mixture,  see
Definition \ref{def:limit-dyn-low-temp}).

\subsection{Gibbs distributed $\bxo$ and conditional disorder}\label{subsec:Gibbs}
Our \red{proofs often require general} Langevin diffusions 
\begin{equation}\label{diffusion-phi}
 \bx_t=\bxo -\int_0^t f'(\vert\vert \bx_u\vert\vert^2/N)\bx_u du
  -\b \int_0^t \nabla \varphi_N (\bx_u) du +\BB_t,
\end{equation}
\red{with $H_{\BJ}$ of \eqref{diffusion} replaced by a generic} potential $\varphi_N \in C^2_0(\B^N)$
\red{and a} Brownian motion $\BB_t$ independent of both $\bxo \in \B^N_\star$ and $\varphi_N$.
Denoting by $\|U_N\|_T$ the sup-norm of a generic function $U_N : [0,T]^d \to \R$ for $d \in \{1,2\}$,
\red{we replace \eqref{eq:Hn} by} 
\begin{equation}\label{eq:Hn-gen}
\phiN (s):=-\frac{1}{N} \varphi_N(\Bx_s)\,,
\end{equation}
and for any $N, T$,  $\bxs$,  $\varphi_N$ and path $t \mapsto (\Bx_t,\BB_t)$ define the error 
between $\bU_N:=(C_N,\chi_N,q_N,\phiN)$ and non-random $\bU_\infty=(C,\chi,q,H)$,  as
\begin{align}
\err_{N,T} (\Bx_\cdot,\BB_\cdot,\bxs,\varphi_N;\bU_\infty) &:= 
 \|C_N-C\|_T  \wedge 1
+\|\chi_N-\chi\|_T \wedge 1
+\|q_N-q\|_T \wedge 1
+\| \phiN -  H\|_T \wedge 1  \,.
\label{eq:err} 
\end{align}
Fixing $\varphi_N \in C_0^2(\B^N)$,  $\bxo \in \SN$ and $f$ as in \eqref{eq:fdef},  we then measure 
%the error 
% between $\bU_N$ 
% of \eqref{empiricalcovariance}--\eqref{eq:Hn}
for the strong solution 
%$(\Bx_s)$ 
of the corresponding Langevin diffusion \eqref{diffusion-phi},  the error between $\bU_N$ and a proposed non-random limit $\bU_\infty$,  as
\begin{align}\label{eq:err-hat}
\widehat\err_{N,T}(\bxo,\bxs,\varphi_N;\bU_\infty) := \E_{\bxo,\varphi_N} [\err_{N,T}(\Bx_\cdot,\BB_\cdot,\bxs,\varphi_N;\bU_\infty)] \,,
\end{align}
where $\E_{\bxo,\varphi_N}[\,\cdot\,]$ denotes the expectation over the 
%$N$-dimensional 
Brownian motion $\BB_\cdot$ and the induced solution $\Bx_\cdot$ of \eqref{diffusion-phi}.  We note in passing that for
any $\bU_\infty$,  $\bxo$, $\bxs$,  $\varphi_N$ and orthogonal $\BO \in \R^{N \times N}$, 
\begin{align}\label{eq:rot-inv}
\widehat\err_{N,T}(\bxo,\bxs,\varphi_N;\bU_\infty) = \widehat\err_{N,T}(\BO \bxo,\BO \bxs,\varphi_N(\BO^T \cdot);\bU_\infty) \,,
\end{align}
with $\widehat\err_{N,T} (\bxo,\varphi_N;\bU_\infty):=\widehat\err_{N,T}(\bxo,{\bf 0},\varphi_N;\bU_\infty)$ 
which is as in \eqref{eq:err-hat} but without the term $\|q_N-q\|_T$.  

\red{Hereafter,  we use the abbreviations $\bv_x$ and $\bv_y$ for the partial derivatives of $\bv(x,y)$
of \eqref{def:vt-new},  and set
\begin{equation}\label{dfn:psi}
\psi(r):=(r \nu'(r))' = r \nu''(r) + \nu'(r) \,.
\end{equation}
We} show in the sequel 
that starting \red{\eqref{diffusion}} at $\bxo$ drawn from $\widetilde \mu^N_{2\b_0,\BJ}$ yields for generic 
strictly 1-\abbr{rsb} mixtures at 
%low temperature (namely,  
$\b_0 > \b_c^{\rm stat}$ the following limiting 
dynamics,  for suitable,  $\b_0$-dependent \red{parameters} $\qs \in (0,1]$ and $\V$.
\begin{defi}[$1$-\abbr{rsb} dynamics\footnote{this refers to 
an initial state with $\b_0$ in the $1$-\abbr{rsb} phase,  regardless of the value of $\b$ one uses in the diffusion \eqref{diffusion}.}]\label{def:limit-dyn-low-temp}
Fixing $\qs \in (0,1]$ and $\V$,  the \emph{$1$-\abbr{rsb} spherical dynamics} 
$\bUsp(\V)$ 
are the functions $C(s,t)=C(t,s)$, $\chi(s,t):=\int_0^t R(s,u) du$, $q(s)$ and $H(s)$, such that:
\begin{equation}\label{eq:str-limit}
q(0)=\qo, \quad  C(s,s):= K(s) \equiv 1, \quad R(s,t)=0 \hbox{ for } t>s, \quad  R(s,s) \equiv 1,
\end{equation}
and for $s>t$,  the bounded,  continuous $C$, $R$, $q(s)$, $H(s)$ 
%satisfy 
are the unique,  continuously differentable in $s$,  solution 
%in the space of bounded,  continuous functions,  
of the integro-differential equations
\begin{align}
\partial_s R(s,t) =
& - \mu(s) R(s,t) + \b^2 \int_t^s
R(u,t) R(s,u) \nu''(C(s,u)) du ,\label{eqRs-new}\\
\partial_s C(s,t)= & - \mu(s) C(s,t) + \b \,  \widehat{{\sf A}}_C(s,t) \,,  \label{eqCs-new}\\ 
\partial_s q(s) = & -\mu(s) q(s) + \b \,  \widehat{{\sf A}}_q(s) \,, \label{eqqs} \\
H(s)= & \b \int_0^s R(s,u) \nu'(C(s,u)) \, du  -  \b \nu'(q(s)) L(s) + \bv (q(s),C_o(s)) \,,
 \label{eqH-new}\\
\mu (s) := & \frac{1}{2} + \b \, \widehat{{\sf A}}_C(s,s) 
 \,, \label{eqZs-new} 
\end{align}
where $C_o(s):=C(s,0)$ and 
\begin{align}
\widehat{{\sf A}}_C (s,t) :=& 
\b \int_0^s R(s,u) \nu''(C(s,u)) C(u,t) \, du + \b \int_0^t R(t,u) \nu'(C(s,u)) \, du  \nonumber \\
& -  \b q(t) \nu''(q(s)) L(s) - \b \nu'(q(s)) L(t)
+ q(t) \bv_x (q(s),C_o(s)) + C_o(t) \bv_y (q(s),C_o(s))  \,,\label{def:AC-new}\\
\widehat{{\sf A}}_q (s) := &
\b \int_0^s R(s,u) q(u) \nu''(C(s,u)) du  - \b \qs^2 \nu''(q(s)) L(s) 
+ \qs^2 \bv_x (q(s),C_o(s)) + \qo \bv_y(q(s),C_o(s)) \,,\label{def:Aq-new}\\
L(s) := & \frac{1}{\nu'(\qs^2)} \int_0^s R(s,u) \nu'(q(u)) \, du \,,
\label{eqL}
\\
\label{def:vt-new}
\bv (x,y) := &  \langle \underline{\hat{\bv}}_\nu,\underline{w} \rangle \;\;
 \hbox{for} \quad \underline{\hat{\bv}}_\nu(x,y,z) :=  [\nu(y),  \nu(x), \qs^{-2} x \,  \nu'(x), z \,  \nu'(x)] 
  \;\; \hbox{and} \;\; z=z(x,y):= \frac{y - \qo \qs^{-2} x}{\sqrt{\qs^2-\qo^2}} \,,
\end{align}
\begin{align}\label{def:w-s}
\begin{bmatrix} \Ep \\ \Eg \\ \Gg \\ 0 \end{bmatrix} &=  \Sigma_\nu(\qo)
% \underline{w}
\begin{bmatrix} w_1 \\ w_2 \\ w_3 \\ w_4 \end{bmatrix}
\quad
\Sigma_\nu (\qo) := 
\begin{bmatrix}
 \nu(1) &  \nu(\qo) & \qo \qs^{-2} \nu'(\qo) & \qs^{-2} \sqrt{\qs^2-\qo^2} \nu'(\qo) \\
\nu(\qo)       & \nu(\qs^2) &   \nu'(\qs^2) & 0 \\
	\qo \qs^{-2} \nu'(\qo)  &  \nu'(\qs^2) & \qs^{-2} \psi(\qs^2) & 0 \\
	\qs^{-2} \sqrt{\qs^2 - \qo^2} \nu'(\qo)& 0 & 0 & \qs^{-2} \nu'(\qs^2) 
	\end{bmatrix}
	\end{align}
In 	the \emph{$1$-\abbr{rsb} $f$-dynamics}  $\bUqf(\V)$
instead of taking $K(s) \equiv 1$ and $\mu(s)$ of \eqref{eqZs-new}
in \eqref{eq:str-limit}-\eqref{eqqs},  
we determine the pair $(K(\cdot),\mu(\cdot))$ via 
% the \abbr{ODE}
\begin{align}
\mu(s)=f'(K(s)) \,,  \qquad 
 \label{eqfZs} 
\frac{1}{2} K'(s) + f'(K(s)) K(s) =  \frac{1}{2} + \b \,  \widehat{{\sf A}}_C(s,s)  \,,  \qquad K(0) = 1 \,.
\end{align}
\end{defi}

\red{In view of \eqref{eq:cond-J} and \eqref{eq:cond-HT},  
we only consider \emph{allowed} parameters $\V=(\Ep,\Eg,\Gg,\qo)$,  as defined next.}
\begin{defi}\label{ft:cons}
\red{A vector $\V=(\Ep,\Eg,\Gg,\qo)$ is allowed,  for $\qs$,  if the following hold:\\
(a).  $|\qo| \le \qs$ and $|\qo|<1$.\\
(b).  If $\qs=0$ then also $\Eg=\Gg=0$. \\
(c).  For pure $p$-spins\footnote{where 
$\langle \Bx,  \nabla H_\BJ (\Bx) \rangle = p H_{\BJ}(\Bx)$
and $H_{\BJ}(\alpha \Bx)=\alpha^p H_{\BJ}(\Bx)$} 
if $\qs>0$ then $\Gg=p \Eg/ \qs^2$ and further $\Eg = \qo^{p} \Ep$ whenever $\qs=|\qo|$.}
\end{defi}

\begin{remark}[singularity]	\label{rem:w-unique}
If $\qs=|\qo|$ then $w_4=0$ and \abbr{wlog} we set $z=0$ in \eqref{def:vt-new}.  More generally,  
the linear equations of \eqref{def:w-s} correspond to $\Ep=\bv(\qo,1)$, $\Eg=\bv(\qs^2,\qo)$, 
$\Gg=\bv_x(\qs^2,\qo)
%+\frac{\qo}{\qs^2}\bv_y(\qs^2,\qo)
$ and $\bv_y(\qs^2,\qo)=0$. 
Except for pure $p$-spins,  the matrix $\Sigma_\nu(\qo)$ is positive definite 
for any $\qs \in (0,1]$ and $|\qo| \le \qs$,  $|\qo|<1$ (see Lemma \ref{lem:app}(a)), 
yielding a unique solution $\underline{w}=[w_1,w_2,w_3,w_4]$ of \eqref{def:w-s}.  For pure $p$-spins, 
necessarily $\Gg=p \Eg/\qs^2$, so one should set $w_3=0$ and proceed to 
solve the reduced system, now positive definite unless also $\qs=|\qo|$
(in which case $w_4=0$ and $\Eg = \Ep \qs^{2p}/\qo^p$ carries no 
additional information on $H_\BJ$,  resulting with 
%$\bv(x,y)= E \qs^{-2m} x^m$ or equivalently 
$\bv(x,y)=\Ep y^p$).
%That is, $\bar{H}(\bn)=-N E$, $\bar{H}(\bxo)=-N E'$ and
% $\bar{G}^1(\bn)=\sqrt{N} \qs G$.
\end{remark}
\begin{remark}[even symmetry]\label{rem:even}
Similarly to \cite[Remark 1.9]{DS21},  taking $\qo \mapsto -\qo$ in Definition \ref{def:limit-dyn-low-temp}
results,  in case $\nu(\cdot)$ is even,  with $(C,R,q,H) \mapsto (C,R,-q,H)$ in \eqref{eqRs-new}-\eqref{eqH-new} 
(for
%one easily checks that for even $\nu(\cdot)$
then the function 
$\bv(x,y)$ at $\qo$ matches $\bv(-x,y)$ at $-\qo$).  Indeed,  
taking $\bxs \mapsto -\bxs$ results with $q_N \mapsto - q_N$ while leaving 
$(C_N,\chi_N,\HN)$ unchanged and if $\nu(\cdot)$ is an even function, 
then also $H_{\BJ}(-\bx)=H_{\BJ}(\bx)$ per given $\BJ$ and hence 
$\cptq(\Ep,\Es,\Gs,\qo)=\cptq(\Ep,\Es,\Gs,-\qo)$.  
\end{remark}

For initial conditions drawn at the \abbr{rs}-phase,  one finds instead the following,  simpler limiting dynamics,  with a
single $\b_0$-dependent parameter $\Ep$,  which coincides with the \abbr{ckchs}-equations  
in case $\Ep=0$ (and thereby $\bv(\cdot)  \equiv 0$).
\begin{defi}[\abbr{rs} dynamics]\label{def:limit-dyn-high-temp}
Fixing $\V=(\Ep)$,  the \emph{\abbr{rs} spherical dynamics} 
$\bUosp(\V)=\bUosp(E)$ are the unique functions 
$C(s,t)=C(t,s)$, $\chi(s,t)=\int_0^t R(s,u) du$ and $H(s)$ which satisfy 
\eqref{eq:str-limit}-\eqref{def:Aq-new} when $\qs=\qo=0$.  In that case 
$q(s) \equiv 0$,  so we take \abbr{wlog} $L(s) \equiv 0$ and $\bv (y) = \bv(0,y) := \Ep \nu(y)/\nu(1)$,
resulting with equations \eqref{eqRs-new}, \eqref{eqCs-new} and \eqref{eqZs-new},  for 
$C_o(s):=C(s,0)$ and
\begin{align}
\widehat{{\sf A}}_C (s,t) =& 
 \b \int_0^s R(s,u) \nu''(C(s,u)) C(u,t) \, du + \b \int_0^t R(t,u) \nu'(C(s,u)) \, du +  C_o(t) \bv' (C_o(s))  \,, 
 \label{def:ACs} \\
H(s) = &  \b \int_0^s R(s,u) \nu'(C(s,u))  \, du  + \bv (C_o(s))\,,  \label{eqH}
\end{align}
with the 
\emph{\abbr{rs} $f$-dynamics}  $\bUof(\V)=\bUof(\Ep)$ again replacing $K(s) \equiv 1$ and $\mu(s)$ of \eqref{eqZs-new}, 
by $(K(\cdot),\mu(\cdot))$ of \eqref{eqfZs}.
\end{defi}
\begin{remark}[non-negative definiteness]\label{rem:non-neg} 
While not apparent from \eqref{eqCs-new},  our results imply that $C(s,t)$ 
of $\bUqf$,  $\bUsp$ are non-negative definite kernels,  and that when $\qs>0$ so are
$\bar{C}(s,t):=C(s,t)-q(s) q(t)/\qs^2$.  We utilize this when analyzing the limit dynamics.
For example,  similarly to \cite[Remark 1.7]{DS21},  the functions 
$(R,C,q)$ 
% of Definition \ref{def:limit-dyn-low-temp} 
may take negative values even if $\qo>0$,  but in $\bUsp(\cdot)$ we always have that 
$|q(\cdot)| \le \qs$ and $|C(\cdot,\cdot)| \le 1$
(since $C(s,s) \equiv 1$).
\end{remark}

%\begin{remark}\label{rem:rad-grad}
%While we do not show it explicitly,  
% it is easy to see that 
It is easy to check that our results also yield the convergence in probability 
%as $N \to \infty$, 
\begin{align}
\lim_{N \to \infty} \frac{-1}{N} \langle \nabla H_\BJ(\bx_s), \bx_t \rangle &=  \widehat{{\sf A}}_C (s,t)   \,, \qquad
\forall s \ge t, \\
%\end{align}
%and in case $\qs>0$ also
%\begin{align}
\lim_{N \to \infty} \frac{-1}{N} \langle \nabla H_{\BJ}(\bx_s),\bxs \rangle &= \widehat{{\sf A}}_q(s)  \,.
\end{align}
%\end{remark}

\red{As we detail next,} the values of $\qs$ and $\V$ one should set as the proxy for initializing 
from $\bxo$ drawn according to the spherical 
mixed $p$-spins Gibbs measure $\tilde{\mu}^N_{2\b_0,\BJ}$ \red{are determined by
the relevant Parisi measure $\zeta_P$}. 
%Setting the parameters as in our next definition, the high-temperature spherical limiting 
%dynamics appears for $\bxo$ drawn according to the spherical mixed $p$-spins Gibbs 
%measure $\tilde{\mu}^N_{2\b_0,\BJ}$ with $\b_0 \le \b_c^{\rm stat}$, whereas the 
%low-temperature spherical limiting dynamics shall appear in case $\b_0>\b_c^{\rm stat}$
%(under certain additional mild conditions on the model).
\begin{defi}[Gibbs measure initialization]$~$\label{def:Gibbs-ic}
\newline
For $\bxo$ drawn from the Gibbs measure at $\b_0 \le \b_c^{\rm stat}$, we consider $\bUof(\Ep)$ and $\bUosp(\Ep)$
% the high-temperature spherical dynamics of Definition \ref{def:limit-dyn-high-temp}
for $\Ep=\Ep(\b_0)$,  where
\begin{equation}\label{eq:Ep-high-temp}
\Ep(\b) := 2 \b \nu(1) \,.
\end{equation}
Next, recall $q_{\rm EA}(\b) \in (0,1)$ of \eqref{eq:kRSB} and
the Parisi variational formula derived in \cite{Chen-Sen,Jagannath-Tobasco} for the a.s.  limit of the ground 
state energy of the model $\nu(q^2 \cdot)$
\begin{equation}\label{def:GS-q}
\GS(q) = \lim_{N \to \infty} \frac{1}{N} \min_{\bsigma\in q\,\SN} \{ H_{\BJ}(\bsigma) \} 
\end{equation}
(i.e. for the Gibbs measures $\tilde{\mu}^N_{2\b,\BJ}(\cdot)$ restricted to $q\, \SN$).
%on the support of the Parisi measure for the model $\nu(\cdot)$ 
%(c.f. \cite[Prop. 2.1]{Talagrand}, which shows that $q_{\rm EA}(\b) \in (0,1)$). 
Setting
% consider the event $\cpt_+(E,G,E',\bsigma)$ for the parameters
\begin{align}\label{def:qs-beta}
\qs = &\qs(\b) := \sqrt{q_{\rm EA} (\b)} \,, \qquad \qquad \qo = \qo(\b) =q_{\rm EA}(\b) \,,\\
\label{def:Es}
\Eg=& \Egs(\b) := \GS (\qs) \,, \\
\label{G-stat}
\Gg = & \Ggs (\b) := \frac{1}{2 \b (1-\qs^2)} + 2 \b (1-\qs^2) \nu''(\qs^2) \,,\\
\label{Ep-E}
\Ep = & \Ep (\b) := \Es(\b) + 2\b \theta(\qs^2) \,, \qquad 
\theta(x):=\nu(1) - \nu(x) - \nu'(x) (1 - x) \,,
\end{align}
for $\bxo$ drawn from the Gibbs measure at $\b_0>\b_c^{\rm stat}$ which is strictly $1$-\abbr{rsb}, 
we consider $\bUqf(\V)$ and $\bUsp(\V)$
% the law-temperature spherical dynamics of Definition \ref{def:limit-dyn-low-temp}, 
for $\qs(\b)$ and 
$\V(\b)$ of \eqref{def:qs-beta}-\eqref{Ep-E},  evaluated at $\b=\b_0$.
\end{defi}

As promised,  our main result,  stated next,  relates the dynamics of the 
%unconditional 
model initialized at 
$\bxo$ drawn from $\tilde{\mu}^N_{2\b_0,\BJ}$, with the limiting dynamics 
of Definitions \ref{def:limit-dyn-low-temp} and \ref{def:limit-dyn-high-temp},
according to $\b_0>\b_c^{\rm stat}$ or $\b_0 \le \b_c^{\rm stat}$,  respectively.
\begin{theo}\label{thm-uncond}
Fix finite $\b,T>0$.  Choose $\bxo \in \SN$ according to $\tilde{\mu}^N_{2\b_0,\BJ}$ and 
consider
% the limit $N \to \infty$ followed by $\ell \to \infty$ for 
the dynamics \eqref{diffusion},  where the derivative of $f(\cdot)$ of the form \eqref{eq:fdef} 
is locally Lipschitz on $[0,r_\star^2)$ and the mixture $\nu(\cdot)$ satisfies \eqref{eq:r-star}.
\newline
(a). In case $\b_0 \le \b_c^{\rm stat}$,  for the dynamics 
$\bUof$ of Definition \ref{def:limit-dyn-high-temp}
at $\b$ and $\Ep=\Ep(\b_0)$ of \eqref{eq:Ep-high-temp},  as $N \to \infty$
\begin{equation}\label{eq:uncond} 
\widehat \err_{N,T}\big(\bxo,H_{\BJ};\bUof(\Ep)\, \big)  \to 0\,, \quad \hbox{in probability},
\end{equation}
at an exponential in $N$ rate.
\newline
(b). Suppose generic 
%mixture  
$\nu(\cdot)$ is strictly 1-\abbr{rsb} at  $\b_0 > \b_c^{\rm stat}$.
Set $\qs=\qs(\b_0)$ and $\V(\b_0)$ of \eqref{def:qs-beta}-\eqref{Ep-E}.  Then,  
there exists measurable mapping $\bxs = \bxs(\bxo,H_\BJ):\SN \mapsto \SNqs$ 
and $\delta_N \to 0$ as $N \to \infty$,   such that
\begin{align}\label{eq:band-dec}
\P\Big(  \big\{ \bigcup \cptq (\V') :  \|\V'-\V(\b_0) \| \le \delta_N \big\} \Big) & \to 1 \,,\\
\widehat\err_{N,T}\big(\bxo,\bxs,H_{\BJ};\bUqf(\V(\b_0))\big) & \to 0\,,  \quad \hbox{in probability}\, ,
\label{eq:uncond-low} 
\end{align}
for the $1$-\abbr{rsb} limiting dynamics $\bUqf(\cdot)$ of 
Definition \ref{def:limit-dyn-low-temp},  at the specified value of $\b>0$.
\end{theo}
\begin{remark} \red{In Theorem \ref{thm-uncond}(b) we learn from \eqref{eq:uncond-low} that choosing 
% the angle 
$q_N(s)$ of \eqref{eq:bnq} with respect to certain $\bxs(\bxo,H_{\BJ}) \in \qs \SN$ yields as $N \to \infty$ to 
a $\|\cdot\|_T$-convergence in probability
of $(C_N,\chi_N,q_N,H_N)$ to the specified limit dynamic $\bUqf(\V(\b_0))$ of Definition \ref{def:limit-dyn-low-temp}.
Further,  \eqref{eq:band-dec} shows that up to a negligible probability as $N \to \infty$,
the random $\bxs$ is a critical point of $H_{\BJ}$ restricted to $\qs \SN$,  with
$\frac{-1}{N} H_{\BJ}(\bxo)$,  $\frac{-1}{N} H_{\BJ}(\bxs)$,  the radial derivative of $\frac{-1}{\sqrt{N}}H_{\BJ}$ 
at $\bxs$ and $q_N(0)$,  all within some non-random $\delta_N \to 0$ of the values 
$\V(\b_0)$ we specify in \eqref{def:qs-beta}-\eqref{Ep-E}}.
\end{remark}
Analogously to \cite[Prop. 1.6]{DS21},  our next result relates 
%as $N \to \infty$ and then $\ell \to \infty$, 
the random functions $\bU_N=(C_N,\chi_N,q_N,\HN)$ 
% of Theorem \ref{thm-uncond} 
to the spherical limit dynamics 
$\bUsp(\V)$ and $\bUosp(\Ep)$ of Definitions  \ref{def:limit-dyn-low-temp}
and \ref{def:limit-dyn-high-temp},  respectively.
\begin{prop}\label{prop:ell-lim}
Fix $\nu(\cdot)$ satisfying \eqref{eq:r-star},  finite $\b,T>0$,  $\qs \in [0,1]$ and \red{allowed 
$\V$ as in Definition \ref{ft:cons}.}
%with $|\qo| \le \qs$.\footnote{\label{ft:rest}further restricting to 
%$\Ep=\Eg$ when $\qo=1$,  to $\Ep=\pm \Eg$ if $\qo=-1$ with $\nu(\cdot)$ even or odd,
%and to $\Eg=\Gg=0$ when $\qs=0$; for pure $p$-spins and $\qs>0$,  further restrict to 
%$\Gg=p \Eg/ \qs^2$ and in case $\qs=|\qo|$ also to $\Eg = \Ep \qo^{p}$,  see Remark \ref{rem:w-unique}}
Suppose $f_\ell$ as in \eqref{eq:fdef} has differentiable derivatives and 
\begin{equation}\label{eq:f0-cnd}
f'_0(1)=\frac{1}{2} + \b \qo \bv_x(\qo,1)+\b \bv_y(\qo,1) \,,
\end{equation}
in case $\qs>0$,  or  $f'_0(1)=\frac{1}{2} + \b \bv'(1)$ if $\qs=0$.  Then,   
\begin{align}
\lim_{\ell \to \infty} \| \,  \bU_{\qs}^{f_\ell}(\V) - \bUsp(\V) \,\|_T  &=0 \,.
\end{align}
%\begin{remark}\label{rem:sp}
%Fixing finite $T,\b>0$,  $\b_0 \ge 0$,  taking $f_\ell(\cdot)$ as in Proposition \ref{prop:ell-lim}, 
Consequently,  for $\V(\b_0)$ and mapping $\bxs(\bxo,H_{\BJ})$ as in Theorem \ref{thm-uncond}, 
we have as $N \to \infty$ then $\ell \to \infty$,  that 
% the convergence in probability 
%$\|\bU_N-\bUosp(\V)\|_\infty \stackrel{p}{\to} 0$ (when $\b_0 \le \b_c^{\rm stat}$),  or
$\|\bU_N-\bUsp(\V(\b_0))\|_T \stackrel{p}{\to} 0$,  with $\qs=0$ for $\b_0 \le \b_c^{\rm stat}$,
while $\qs>0$ for generic,  strictly $1$-\abbr{rsb} mixtures.
\end{prop}

\red{Recall that for} any $N \ge 1$,  the Langevin diffusion \eqref{sphere-diffusion} with $\bxo$ drawn 
from the Gibbs measure $\mu^N_{2\b,\BJ}$ of \eqref{Gibbs}-\eqref{def:ZbeJ}
is a stationary process.  In particular,  any limit point of $\E [C_N(s,t)]$,  as $N \to \infty$,
must then be of the form $c(|s-t|)$.  We naturally expect the same when taking such $\bxo$
in \eqref{diffusion} and considering the limit $N \to \infty$ followed by $\ell \to \infty$.  That is,  
in view of Theorem \ref{thm-uncond} and Proposition \ref{prop:ell-lim},
we also expect such stationary limit dynamic $\bUsp(\V)$ when considering $\b_0=\b$ and 
$\V$ as in Definition \ref{def:Gibbs-ic}.  More generally,  one may wonder:
\newline
(a).  Is such a stationary limit related to the \abbr{fdt}-regime of the non-stationary 
% solution for the 
\abbr{CKCHS}-equations?\\
(b).  Which $(\qs,\V)$ yield a stationary limit $\bUsp(\V)$ and is this 
possible for Gibbs initialization at some $\b_0 \ne \b$?\\
(c).  For the stationary limit in the $1$-\abbr{rsb} phase
(that is,  $\b_0=\b > \b_c^{\rm stat}$),  does the 
limit dynamic relax within $O_N(1)$-times to the corresponding band
in the pure state decomposition of the 
% spherical mixed $p$-spins 
Gibbs measure $\widetilde{\mu}^N_{2\b,\BJ}$ of \eqref{eq:Gibbs-sph}?

\red{In Section \ref{subsec:stat} we} provide definitive answers to all of these questions.

\begin{remark}[gradient flow]\label{rem:grad}
%\subsection{Gradient flow with Gibbs initial conditions}
Rescaling time $\tau=t/\b$ in \eqref{diffusion} and setting $\by_t=\bx_{\tau}$,  results with 
\begin{equation}\label{eq:pre-gradient}
 d\by_t=-\frac{1}{\b} f_\ell ' (\vert\vert \by_t \vert\vert^2/N)\by_t d t 
  -\nabla H_\BJ(\by_t) d t +\frac{1}{\sqrt{\b}} d\BW_{t},  
\end{equation}
for the standard $N$-dimensional Brownian motion $\BW_t:= \sqrt{\b} \,  \BB_{\tau}$ and initial condition $\by_0=\bxo$.
In the limit as $\ell \to \infty$ and then $\b \to \infty$,  the system \eqref{eq:pre-gradient} yields the $N$-dimensional 
gradient flow on $\SN$,
\[
\frac{d \by_t}{dt} = - \nabla_{\rm sp} H_{\BJ}(\by_t) \,.
\]
Note that scaling the time argument of $(C_N,q_N,\HN)$ by $1/\b$ results with the corresponding 
functions for $\by_t$,  while for a non-vanishing limiting response as $\b \to \infty$,  one must replace $\chi_N$ 
%of \eqref{integrated} 
by
% {\chi}_N(s,t)&:= 
$\frac{\sqrt{\b}}{N} \langle \by_s,\BW_t \rangle=\b \chi_N(\frac{s}{\b},\frac{t}{\b})$.  
By Theorem \ref{thm-uncond},  for a Gibbs distributed initial condition $\by_0$,  as $N \to \infty$ 
these functions converge in probability to the relevant solutions $\bUqf(\cdot)$ 
with all its time arguments rescaled by $1/\b$.  By Proposition  \ref{prop:ell-lim},  taking $\ell \to \infty$ yields  
the similarly time scaled solutions of $\bUsp(\cdot)$
(with $\mu(\cdot)$ and $L(\cdot)$ also scaled in size by the factors $1/\b$ and $\b$,  respectively).  This is equivalent 
to setting $\mu (s) = \frac{1}{2\b} +  \widehat{{\sf A}}_C(s,s)$ in \eqref{eqZs-new} and otherwise 
having $\b=1$ in Definitions \ref{def:limit-dyn-low-temp} and \ref{def:limit-dyn-high-temp}. 
An easy adaptation the proof of Proposition \ref{prop:ell-lim} thus shows that taking $\b \to \infty$ yields
the unique solution $\widetilde{\bU}_{\qs}^{\rm sp}(\cdot)$ that corresponds to
$\b=1$ and $\mu (s) = \widehat{{\sf A}}_C(s,s)$
in Definitions \ref{def:limit-dyn-low-temp} and \ref{def:limit-dyn-high-temp}.
\end{remark}

\begin{remark}[non-generic and $k$-\abbr{rsb},  $k \ge 2$] \label{rem:k-rsb} Our results apply also for non-generic,
strictly $1$-\abbr{rsb} models $\nu(\cdot)$ for which the pure-state decomposition of \cite[Prop. 3.1]{DS24} holds.  Moreover, 
while we do not do so here,  one can similarly derive for any $k \ge 2$
the limiting dynamics $\bU^{\rm sp}_{{\sf \qs}(\b_0)}(\V(\b_0))$ 
for strictly $k$-\abbr{rsb} mixtures
at $\b_0>\b_c^{\rm stat}$.  The parameters of Definition \ref{def:Gibbs-ic} 
then extend to ${\sf \qs}(\b)=(\sqrt{q_j-q_{j-1}},  1 \le j \le k)$ and $\V(\b)$ consisting of suitable $(2k+1)$ values 
$(\Ep(\b),E_\star^{(j)}(\b),G_\star^{(j)}(\b), 1 \le j \le k)$.  The function $q_N(s)$ of 
\eqref{eq:bnq} is likewise replaced by the overlaps ${\sf q}_N(s)=(q_N^{(j)}(s),1 \le j \le k)$ of $\bx_s$
with $k\ge 2$ orthogonal vectors $\bxs^{(j)} \in \qs^{(j)} \SN$,   starting at ${\sf q}_o(\b)=(q_j-q_{j-1}, 1 \le j \le k)$. 
In the equations for $\bU^{\rm sp}_{\sf \qs}$ the function $\bv(\cdot)$ of \eqref{def:vt-new} 
is replaced by $\bv(\bx,y)$ whose argument $(\bx,y) \in \R^{k+1}$ consists of  
the limit ${\sf q}(s)$ of ${\sf q}_N(s)$,  amended by $C_o(s)$,  with $\Sigma_\nu({\sf q}_o)$ 
of \eqref{def:w-s} now being a $2(k+1)$-dimensional matrix and the non-negative definite kernels are 
$\bar{C}(s,t)=C(s,t)-\sum_{j} q^{(j)}(s) q^{(j)}(t)/q_o^{(j)}$.  The events $\cpt_{\sf \qs}(\V)$ of \eqref{eq:cond-J}
now require that $H_{\BJ}$ have tangentially critical points at $\sum_{j \le m} \bxs^{(j)}$ for $1 \le m \le k$,  
as a result of which the formulas \eqref{eqH-new} and \eqref{def:AC-new}-\eqref{eqL} for 
$H(s)$,  $\widehat{{\sf A}}_C (s,t)$ and the $k$-dimensional 
$\widehat{{\sf A}}_{\sf q} (s)$,  ${\sf L}(s)$ be more involved than those for $k=1$.
\end{remark}

\subsection{Stationary dynamics and band relaxation}\label{subsec:stat}
The stationary limiting dynamics are defined as follows.
\begin{defi}\label{def:stat}
A solution $(C,R,q,H)$ of \eqref{eq:str-limit}-\eqref{def:w-s} is called stationary if 
$C(s,t)=c(s-t)$, $R(s,t)=r(s-t)$, $q(s)=q(0)$ and $H(s)=H(0)$ for all $s,t \ge 0$.
\end{defi}
Towards characterizing the stationary dynamics,  we 
set for $\b>0$ and $\gamma \in \R$,  as in \cite[(2.2)]{DS21}, the function 
\begin{equation}\label{def:phi-gamma}
\phi_\gamma(x):=\gamma + 2 \b^2 \nu'(x) \,,
\end{equation}
and set,  for $g_\b(x)$ of \eqref{def:g-beta},  as done in \cite[(2.7)]{DS21}, 
\begin{equation}\label{eq:c-inf}
c_\gamma(\infty):=\sup\{ x \in [0,1] : g_{\b}'(x) \ge \frac{1}{2} - \gamma \} \,,
% i.e. $\gamma \ge \frac{1}{2(1-x)} - 2 \b^2 \nu'(x)$
\end{equation}
provided $\gamma$ is large enough that such $x$ exists.  Now,  recall \cite[Prop. 1.4]{DGM},  that the \abbr{fdt} equation
\begin{equation}\label{FDTDb-new}
 c'(\tau)=-\int_0^\tau \phi_\gamma(c(v)) c'(\tau-v) dv - \frac{1}{2}\,, \qquad 
 c(0)=1 \,,
\end{equation}
admits for such $\gamma$ a unique $[0,1]$-valued,  strictly decreasing,  continuously differentiable solution,  which 
converges as $\tau \to \infty$ to $c_\gamma(\infty)$ of \eqref{eq:c-inf}.   Further,  \cite[Prop.  1.4]{DGM} 
shows that an \abbr{FDT} solution $C_{\fdt}(\tau)$ for the
\abbr{CKCHS}-equations,  namely for $\bUosp(0)$ of Definition
\ref{def:limit-dyn-high-temp},  must be of the form \eqref{FDTDb-new},  converging 
 as $\tau \to \infty$ to
\begin{equation}\label{def:qd-beta}
\qdyn (\b) := \sup \{ x \in [0,1) : g'_\b(x) \ge 0 \} \,.
\end{equation}
In particular,  for  \red{$g_\b(x)=g_\b(x;\nu)$ of \eqref{def:g-beta} and}
the positive and finite,  
%dynamical 
critical parameter of \cite[(1.23)]{DGM},  
\begin{equation}\label{def:bc-dyn}
\b_c^{\rm dyn} = \b_c^{\rm dyn} (\nu) := \inf \{ \b \ge 0 : \sup_{x \in [0,1)} \{ g'_\b(x) \} > 0 \} \,,
\end{equation}
one has by definition that $\qdyn(\b)=0$ for $\b<\b_c^{\rm dyn}$ 
while $\qdyn (\b)>0$ for $\b>\b_c^{\rm dyn}$.  Since $g_\b(0)=0$ and $g'_\b(0)=0$, 
necessarily $\b_c^{\rm dyn} \le \b_c^{\rm stat}$ of \eqref{def:bc-stat},  with 
a strict inequality provided $\qdyn (\b_c^{\rm dyn})>0$ (i.e.  a first-order dynamic phase transition),
or alternatively, whenever
\begin{equation}\label{eq:betac-gap}
\sup \{ x \in [0,1] : \sum_{p \ge 3} p b_p^2 x^{p-3} (1-x) \} > 2 b_2^2 \,.
\end{equation}

Our next result characterizes the parameters for which $\bUsp(\cdot)$ are stationary,  
%  Definitions \ref{def:limit-dyn-low-temp} and \ref{def:limit-dyn-high-temp},  
further relating such stationary dynamics to the solution of \eqref{FDTDb-new} and thereby to the large time asymptotic \abbr{fdt} 
regime for the corresponding equations (and at low temperature, also to the onset of aging).
\begin{prop}\label{prop-stat}
In case $\qs>0$,  set $\qo=\alpha \qs$ for $|\alpha| \le 1$.  The collection of all stationary dynamics of the form
$\bUsp(\cdot)$,  consists of $r(\tau)=-2 c'(\tau)$ and the twice continuously 
differentable, strictly decreasing, unique solution $c(\tau)=c_{\gamma_\star}(\tau)$ of 
 \eqref{FDTDb-new},  for $|\alpha|<1$ and
\begin{equation}\label{def:gamma-star}
\gamma_\star(\qs,\alpha) := \frac{1}{2 (1 - \alpha^2)} - 2 \b^2 \frac{\nu'(\alpha \qs)^2}{\nu'(\qs^2)} \,.
\end{equation} 
It is easy to verify that $c_{\gamma_\star}(\infty) \ge \alpha^2$,  with equality if and only if 
$g_\b'(x) < g_\b'(\alpha^2)$ whenever $x \in (\alpha^2,1]$ and either the mixture $\nu(\cdot)$ is pure 
$p$-spins or $\alpha=\qs$.\footnote{\label{ft:even-odd}
\noindent 
Except when $\nu(\cdot)$ is even or odd,  taking instead $\alpha=-\qs$ has no effect on 
$\gamma_\star$ or $c_{\gamma_\star}(\cdot)$.}

\red{With $\psi(\cdot)$ of \eqref{dfn:psi} and}
$b_\alpha:= \nu'(\alpha \qs)/\nu'(\qs^2)$,  such stationary dynamics appears if and only if $|\alpha|<1$ and 
\begin{align}\label{def:stat-EG}
	\begin{bmatrix}  
	 \Ep - 2\b ( \nu(1) - (1-\alpha^2) \nu'(\alpha \qs) b_\alpha) \\ 
	\Eg - 2 \b \nu(\alpha \qs) \\ 
	\qs \Gg - 2 \b \alpha  \nu'(\alpha \qs)  \\
	\frac{\alpha}{2\b  (1-\alpha^2)}
	- 2\b b_\alpha [\alpha \psi(\alpha \qs)  - \qs \nu''(\alpha \qs)]
	 \end{bmatrix}  &\in {\rm Image} \begin{bmatrix}
	 \nu(\alpha \qs) & \alpha \qs \nu'(\alpha \qs)  \\
	 \nu(\qs^2)       & \qs^2 \nu'(\qs^2)  \\
	 \qs \nu'(\qs^2)  & \qs \psi(\qs^2)  \\
	 \qs \nu'(\alpha \qs) & \qs \psi(\alpha \qs) 
	\end{bmatrix}
	\end{align}
\red{In particular},  for $\alpha=\qs$\footnote{\label{ft:even} 
If $\nu(\cdot)$ is even,  this applies also for $\alpha=-\qs$,  which does not change  
the functions $(C,\chi,H)$ in $\bUsp(\cdot)$,  see Remark \ref{rem:even}.}
\red{the constraint \eqref{def:stat-EG} amounts to}
\begin{align}\label{eq:stat-G-qs}
\Gg &= \frac{1}{2\b (1-\qs^2)} + 2 \b (1 - \qs^2) \nu''(\qs^2)  \,,  \\
\Ep - \Eg &= 2\b [ \nu(1) - (1-\qs^2) \nu'(\qs^2) -  \nu(\qs^2)] 
\label{eq:stat-EpE-qs}
\end{align}
\red{(as in \eqref{G-stat}-\eqref{Ep-E})}.  Similarly,  the only 
stationary dynamic of the form $\bUosp(\Ep)$ is $r(\tau)=-2 c'(\tau)$ 
and $c(\tau)=c_{1/2}(\tau)$ of \eqref{FDTDb-new},  which occurs if and only if $\Ep=2\beta \nu(1)$ as in \eqref{eq:Ep-high-temp}.
%  (as before, $H(s)=E'$ for all $s \ge  0$).
\end{prop}

\begin{remark} By the preceding,  for any $\qs>0$,  \red{having} a stationary dynamics $\bUsp(\cdot)$ 
determines via \eqref{def:stat-EG} only two degrees of freedom among the parameters $\V=(\Ep,\Eg,\Gg,\qo)$ 
(for pure $p$-spins it actually determines only one,  but that case fixes the value of $\Gg=\Gg(\Eg,\b,\qs)$). 
In particular,  one has a stationary dynamics whenever $\qo=\qs^2$ with $(\Gg,\Ep-\Eg)$ 
given in terms of $(\b,\qs)$ via 
\eqref{eq:stat-G-qs}-\eqref{eq:stat-EpE-qs}.  \red{But,  for $\bxo$ 
drawn from}  $\widetilde{\mu}^N_{2\b_0,\BJ}$ at some
$\b_0>\b_c^{\rm stat}$ where the model is strictly 1-\abbr{rsb},  
\red{all degrees of freedom disappear,  yielding}
a stationary limit dynamics $\bUsp(\cdot)$ if and only if $\b=\b_0$
(with the stationary solution then prescribed by Corollary \ref{cor:relax}).
Indeed,  by Definition \ref{def:Gibbs-ic} and Theorem \ref{thm-uncond}(b)
such Gibbs initial condition requires taking $\qo=\qs^2=q_{\rm EA}(\b_0)$.  
\red{In view of \eqref{eq:stat-EpE-qs}} 
%Having now $\alpha=\qs$ and $b_\alpha=1$ in \eqref{def:stat-EG},  it is easy to verify that 
a stationarity 
dynamics can \red{then} only appear if $\Ep-\Eg=2\b \theta(\qs^2)$ for such $\qs$ and 
$\theta(\cdot)$ of \eqref{Ep-E},  at $\b$ of the diffusion \eqref{diffusion}. 
Recalling from Definition \ref{def:Gibbs-ic} that one must have taken 
$\Ep-\Eg=\Ep(\b_0)-\Eg(\b_0)=2\b_0 \theta(\qs^2)$,  this is only possible when $\b=\b_0$,  as claimed.
\end{remark}

Considering the conclusion of Proposition \ref{prop:ell-lim} at $\b_0=\b$,  we provide in the sequel thermodynamic properties
of the shift-invariant Langevin dynamic \eqref{diffusion},  starting at  $\bxo$ 
sampled from the Gibbs measure $\widetilde{\mu}^N_{2\b,\BJ}$.  In particular,  we
determine precisely when such (random) dynamics has $O_N(1)$-band-relaxation time,  in the
following sense.
\begin{defi}\label{def:rel}
For $\hbxs
%=\qs^{-1} \bxs
\in \SN$ and $\alpha \in [-1,1]$,  we say that a 
Langevin dynamics \eqref{diffusion-phi} relaxed within $O_N(1)$-time 
onto $\alpha \hbxs$-band,  if as $N \to \infty$ followed by $\ell \to \infty$,  then $t \to \infty$ and finally $\tau \to \infty$,
\begin{equation}\label{def:fast-conf}
\frac{1}{N} \langle \bx_\tau,\hbxs \rangle  \to \alpha,   \qquad C_N(\tau+t,t) \to \alpha^2 , \qquad \hbox{in probability}  \,.
\end{equation}
We say that \eqref{diffusion-phi} has \emph{$O_N(1)$-band-relaxation time},
% to $\alpha \hbxs^\perp$-band,  
if \red{for some measurable map $(H_{\BJ},\bxo) \mapsto (\alpha,\hbxs)$ and taking the 
limits over $N$,  $\ell$,  $t$ and $\tau$ in the same order as above,   \eqref{def:fast-conf} holds} and 
\begin{equation}\label{def:fast-rel}
\frac{1}{N} \langle \bx_\tau - \alpha \hbxs, \bxo \rangle \to 0 \, , \qquad 
\hbox{in probability}  \,.
\end{equation}
Lastly,  dynamics have \emph{$O_N(1)$-relaxation time} if they satisfy \eqref{def:fast-conf}-\eqref{def:fast-rel} 
with $\alpha=0$.
\end{defi}
\begin{remark} As $N,\ell \to \infty$ the dynamics \purple{\eqref{diffusion}} evolves near $\SN$,  in which case 
\eqref{def:fast-conf} amounts to $\bx_\tau$ converging in $O_N(1)$-time to the band 
$
%{\Bbb S}_{\bxs} (\alpha \qs)=
{\Bbb S}_{\hbxs}(\alpha)$ of \eqref{eq:subsphere},  where it further un-correlates in $O_N(1)$-time.  Namely,
\[ 
\Big| \frac{1}{N} \langle \bx_{t+\tau}-\alpha \hbxs, \bx_t-\alpha \hbxs \rangle \Big| \to 0 \,.
\]
Having also \eqref{def:fast-rel} means that within the band the dynamics has un-correlated from
its initial state $\bxo$. 
% Namely, 
%\[
%\Big| \frac{1}{N} \langle \bx_{\tau}-\alpha \hbxs, \bxo- \frac{\qo}{\qs} \hbxs \rangle \Big| \to 0 \,.
%\]
\end{remark}
 
Recall \cite[Section 1.2]{SubagFlandscape},  that the Hamiltonian $H_{\BJ}(\cdot)-H_{\BJ}(\bxs)$ 
on the band ${\Bbb S}_{\bxs} (q)$ of \eqref{eq:subsphere},  conditional on the event 
$\nabla_{\rm sp} H_{\BJ}(\bxs) = {\bf 0}$ corresponds after the isometry ${\Bbb S}_{\bxs} (q) \mapsto \SN$
to the spherical mixed $p$-spin model for 
\begin{equation}\label{def:nu-q}
\nu_q(x):=\nu(q+(1-q)x)-\nu(q)-(1-q) \nu'(q) x  \,.
\end{equation} 
By definition of $q_{\rm EA}(\b)$,   \emph{at any $\b<\infty$},  the model $\nu_{q_{\rm EA}(\b)}(\cdot)$ is
\abbr{rs} at $\b$,  \red{namely 
$\b \le \b_c^{\rm stat}(\nu_{q_{\rm EA}(\b)})$ (see \cite[(1.43)]{SubagFlandscape})}.  As a corollary of 
Proposition \ref{prop-stat},  \red{we show that} the shift-invariant Langevin 
dynamic \eqref{diffusion} starting from $\bxo$ sampled from the Gibbs measure $\widetilde{\mu}^N_{2\b,\BJ}$ 
has $O_N(1)$-band-relaxation time \red{if $\beta<\b_c^{\rm dyn}(\nu_{q_{\rm EA}(\b)})$ but not 
when $\beta>\b_c^{\rm dyn}(\nu_{q_{\rm EA}(\b)})$.}
\begin{cor}\label{cor:relax}
Consider for any strictly 1-\abbr{rsb} generic model or \abbr{rs} model, 
the shift-invariant Langevin dynamic \eqref{diffusion},  starting at  
$\bxo$ sampled from the Gibbs measure $\widetilde{\mu}^N_{2\b,\BJ}$,
setting $q_\b:=q_{\rm EA}(\b)$ and 
\begin{equation}\label{def:gamma-star2}
\gamma_\b := \frac{1}{2 (1 - q_\b)} - 2 \b^2 \nu'(q_\b) \,.
\end{equation} 
The empirical covariance $C_N(t+\tau,t)$
% at $\b>\b_c^{\rm stat}$, 
is given in the limit $N \to \infty$ followed by $\ell \to \infty$,  uniformly over bounded $t$ and $\tau$,  
by the stationary \abbr{fdt} solution 
\begin{equation}\label{eq:stat-rel}
c_{\gamma_\b}(\tau) \to c_{\gamma_\b}(\infty) \ge q_\b \,.
\end{equation}
\red{The model $\nu(\cdot)$ admits} $O_N(1)$-band-relaxation time if and only if $c_{\gamma_\b}(\infty)=q_\b$.
\red{Equality in \eqref{eq:stat-rel} holds} for 
$\b < \b_c^{\rm dyn}(\nu_{q_\b})$,  with a strict inequality for any $\b > \b_c^{\rm dyn}(\nu_{q_\b})$.
\end{cor}
\begin{remark}\label{rem:res-model}
\red{In view of \eqref{def:bc-dyn} and \eqref{def:nu-q},  having $\beta<\beta_c^{\rm dyn}(\nu_q)$ 
is equivalent to $g'_\b(x) \le g'_\b(q)$ for all $x \in [q,1]$.  Corollary \ref{cor:relax} deals with 
both high (\abbr{rs}-phase),  and low temperatures (specifically,  strictly $1$-\abbr{rsb} phase).  
In the former case $q_\b=0$
% and $\nu_{q_\beta}(\cdot)=\nu(\cdot)$,  
with transition from fast relaxation 
at the original  $\beta^{\rm dyn}_c$,   while in the latter case it is dictated by $\beta^{\rm dyn}_c$ 
of a new mixture $\nu_{q_{\rm EA}(\b)}(\cdot)$.  Indeed,  the latter is the spherical model for 
$H_{\BJ}(\cdot)-H_{\BJ}(\bxs)$,  conditioned
%as in \cite[Thm.~1.1]{DS21},   
on $\cptq(\V)$ of \eqref{eq:cond-J},  at $\qs$ of \eqref{def:qs-beta} but without specifying that 
$H_{\BJ}(\bxo)=-N \Ep$ and restricted to the band ${\Bbb S}_{\bxs} (\qs)$ 
of \eqref{eq:subsphere} (which we further map to $\SN$).  Thus,  assuming that $\bx_t$ remains
within this band (which is related to what Corollary \ref{cor:relax} shows),  our dynamics
is effectively subjected to the model $\nu_{q_{\rm EA}(\b)}(\cdot)$.}
%[Of course, there is the question of how/why we don't escape the band, which is affected by the Hamiltonian in a small neighborhood of the band, not just the band itself.]  
\end{remark}

\subsection{Organization}
\red{In Section \ref{subsec:reduc} we reduce Theorem \ref{thm-uncond} (with $\bxo$ chosen according 
to a Gibbs measure),  to the proof of Theorem \ref{thm-macro} where $\bxo$ is non-random.  
Section \ref{subsec:pf-macro} details the challenge and key steps in proving the latter theorem,  
to which Sections \ref{sec:ex}-\ref{sec:cont-lim} are devoted (among 
other things,  
Proposition \ref{prop:ell-lim} is also proved in Section \ref{sec:cont-lim}).
It is followed by the outline in Section \ref{subsec:gloss} 
of our notational conventions.
Section \ref{sec-asymp} has the proof of
Proposition \ref{prop-stat} and Corollary \ref{cor:relax}}.
In addition,  under a certain \abbr{fdt}-ansatz (as in \cite[Prop.  2.1]{DS21}),  
Proposition \ref{prop:fast-relax-band} shows \purple{that} 
$O_N(1)$-fast relaxation to a spherical band,  for $\bxo$ from the Gibbs measure at $\b_0 > \b_c^{\rm stat}$ and 
a different Langevin inverse-temperature parameter $\b>\b_c^{\rm dyn}$,  is plausible only for pure $p$-spins, 
% but not for any other mixture,  
while
Section \ref{subsec:fdt} analyzes the large time asymptotic of 
the dynamic $\bUsp(\cdot)$ in the \abbr{fdt} regime
(namely,  for $s=t+\tau$,  with $\tau$ fixed and $t \to \infty$),
starting at the Gibbs measure for $\b_0 \ne \b$.

\bigskip
\noindent {\bf Acknowledgment}
We thank Federico Ricci-Tersenghi for pointing out the relevance of \cite{Capone, FFRT1,FFRT2,FZ,Sun}.
We also thank Silvio Franz,  Reza Gheissari and Pierfrancesco Urbani,  for scientific discussions about 
the relation between our work and other references,  \red{and to the referees 
for their comments, which significantly improved the presentation of this paper.}
This research was partially supported by 
\abbr{NSF} grant \abbr{DMS}-2348142 (A.D.),   \abbr{ISF} grant 2055/21 (E.S.) and \abbr{ERC} 
grant PolySpin 101165541 (E.S.).

\section{Proof of Theorem \ref{thm-uncond}: reduction,  challenges and key-steps}\label{subsec:org}

\subsection{\purple{Reduction to non-random $\bxo$}}\label{subsec:reduc}
The key to the proof of Theorem \ref{thm-uncond} is our analog of \cite[Thm. 1.1]{DS21},  
with exponential tail decay of the distance $\widehat\err_{N,T}\big(\bxo,\bxs,H_{\BJ};\cdot)$ 
between $\bU_N=(C_N,\chi_N,q_N,\HN)$ and the limiting dynamic $\bUqf = (C,\chi,q,H)$,
for non-random $\bxo \in \SN$ and $\bxs \in \SNqs$,  when conditioning on $\cptq(\V')$ 
of \eqref{eq:cond-J} for the following $\V'$ values.
\begin{defi}\label{def:vB-eps}
For $\qs \in [0,1]$,  $\bxs \in \qs \SN$,  \red{$\epsilon>0$,  allowed
% vector of parameters 
$\V=(\Ep,\Eg,\Gg ,\qo)$ as in Definition \ref{ft:cons} and the bands of  \eqref{eq:subsphere},  define the sets}
\begin{align}\label{eq:bar-B}
\vB_\epsilon( \V) &:= \big\{  \red{\V'=(\Epp,\Egp,\Ggp,\qop)} \; {\rm is} \; {\rm allowed} :  \|\V'-\V\|_\infty <\epsilon
\big\} \,,  \\
\vB_\epsilon(\bxs,\V) &:= 
\big\{(\bx, \V'):  \V' \in \vB_\epsilon(\V),  
\red{\bx \in {\Bbb S}_{\bxs} (\qop)} \big\} \,.
 \label{eq:bar-Bx}
\end{align}
\end{defi}

\begin{theo}\label{thm-macro}
Fix finite $\b,T > 0$,  non-random $\bxs \in \qs \SN$ for $\qs \in [0, 1]$,  some \red{allowed}
$\V$ and a mixture $\nu(\cdot)$ satisfying \eqref{eq:r-star}.  Consider the strong solution of 
\eqref{diffusion} with $f(\cdot)$ as in \eqref{eq:fdef} of locally Lipschitz derivative 
on $[0,r_\star^2)$ and with potential $H_{\BJ}$ conditional 
on the event $\cptq(\V)$ of \eqref{eq:cond-J}.\footnote{\label{ft:purecdn}
The conditional density of $\BJ$ is, 
up to normalization, the restriction of its original density to the appropriate affine subspace 
(with no change to the law of the independent $\BB$).}
%or \eqref{eq:cond-HT} (in case $\qs=0$),  
Then,  for $\vB_\epsilon(\cdot,\cdot)$ of  \eqref{eq:bar-Bx},  any $\delta>0$ and small enough $\epsilon=\epsilon(\delta,\qs,\V)>0$,
\begin{align}\label{eq:exp-err-LT}
% \lim_{\epsilon \to 0} 
\limsup_{N \to \infty}  \sup_{(\bxo,\V') \in \vB_\epsilon(\bxs,\V)}
\frac{1}{N} \log \P \Big( \widehat\err_{N,T}\big(\bxo,\bxs,H_{\BJ};\bUqf(\V)\big) > \delta  \,  \big| \,
\cptq(\V') \Big) &<  0 \,.
\end{align}
\end{theo}
\begin{remark} \red{Our proof of Theorem \ref{thm-macro} can handle also} starting at a critical point 
$\bxo=\bxs$ (i.e. taking $\qo=\qs=1$,  $\Ep=\Eg$),  for which some qualitative information 
about the limit dynamics has been gained in \cite{BGJ} from an approximate evolution for (only) the pair 
$(\HN(s),|\nabla_{\mathrm{sp}} H_{\BJ}(\bxs)|/\sqrt{N})$.  \red{But,  excluding here $|\qo|=1$,  
allows us to avoid dealing with the extra degeneracy}\footnote{\label{ft:degen}\red{For $\qo=1$,  necessarily 
$\Ep=\Eg$,  while $\Ep=\pm \Eg$ if $\qo=-1$ with $\nu(\cdot)$ even or odd.}}
in this case.
\end{remark}

\purple{As we show next},  thanks to the general principle of \cite{DS24},  we get Theorem \ref{thm-uncond} as
an immediate consequence of Theorem \ref{thm-macro} and the following continuity property,  whose 
proof is deferred to Section \ref{sec:ex}.
\begin{lem}\label{lem:cont-varphi}
\red{Fix $T,\b,N,d$, }$\varphi_N,\psi^{(i)}_N \in C_0^2(\B^N)$,
$\bxo, \bxs \in \B^N(\sqrt{N})$ and non-random $\bU_\infty$.  \red{For any $\eta_n^{(i)} \to 0$,
% as $n \to \infty$,
}
\begin{equation}\label{eq:varphi-cont}
\red{\lim_{n \to \infty}}
 \widehat\err_{N,T}(\bxo,\bxs,\varphi_N+ \sum_{i=1}^{d} \red{\eta^{(i)}_n}  \psi^{(i)}_N;\bU_\infty)
=  \widehat\err_{N,T}(\bxo,\bxs,\varphi_N;\bU_\infty) \,.
\end{equation}
\end{lem}

\begin{proof}[Proof of Theorem \ref{thm-uncond}]  \purple{We follow the outline in \cite[Sec. 2.1]{DS24}, } 
for the random Hamiltonian $H_{\BJ}$ on $\B^N_\star$ and 
$\bsigma$ drawn from the Gibbs measure $\widetilde{\mu}^N_{2\b_0,\BJ}$ on $\SN$.
\newline
(a).  For $\b_0 \le \b_c^{\rm stat}$,  the Gibbs measure is in the 
high-temperature regime of \cite[(1.6)]{DS24},  apart from the scaling factors 
$\frac{1}{N}\E[H_{\BJ}(\bsigma)^2]= \nu(1)$ throughout $\SN$ and having 
here $\b=2\b_0$.  Reflecting such scaling in our choice of $\Ep(\b_0)=2\b_0 \nu(1)$, 
we apply  \purple{\cite[Thm.  1.2]{DS24}} for $A=(\delta,\infty)$ and the a.s.  measurable functional 
$f_N(\bsigma) := \widehat \err_{N,T}\big(\bsigma,{\bf 0}, H_{\BJ};\bUof(\Ep(\b_0))\big)$ 
(which is defined in \eqref{eq:err-hat} in terms of the solution of \eqref{diffusion},  
with $f$ of \eqref{eq:fdef} and $\bxo=\bsigma$).  Indeed,  \purple{the measurability on 
$\SN \times \R$ of $\P(f_N(\bsigma) > \delta | H_{\BJ}(\bsigma) = -N E)$ is obvious, 
whereas} 
% the proof of \cite[Thm.  1.2]{DS24}
%applies verbatim also for a function of both $\bsigma$ and $H_{\BJ}(\cdot)$,  while 
the exponential convergence condition \cite[(1.8)]{DS24} is provided in \eqref{eq:exp-err-LT}.
\newline
(b).  For generic $\nu(\cdot)$ which is strictly 1-\abbr{rsb} at $\b_0>\b_c^{\rm stat}$ we set
$\qs=\qs(\b_0)$ and $\V=\V(\b_0)$ as in Definition \ref{def:Gibbs-ic}
with $\bU_\infty:=\bUqf(\V)$ the corresponding $1$-\abbr{rsb} dynamics 
of Definition \ref{def:limit-dyn-low-temp}.  Recall the pure state 
decomposition $(\bx^\star_i,B_i)_{i \le d_N}$ of \cite[Prop. 3.1]{DS24}.
Fix an arbitrary point $\bx^\star_0 \in \qs \SN$ and set $B_0:= \SN \setminus \cup_i B_i$, 
so that the partition $(B_i)_{i=0}^{d_N}$ induces the measurable mapping 
$\bxs(\bsigma,H_{\BJ})=\bx^\star_i \in \qs \SN$.  With 
$\widetilde{\mu}^N_{2\b_0,\BJ}(B_0) \le \epsilon_N \to 0$, 
by \cite[(3.4) and (2), (5) of Prop. 3.1]{DS24}, 
one has that \eqref{eq:band-dec} holds for some $\delta_N \to 0$,  while
%the convergence in probability of 
\eqref{eq:uncond-low} follows by applying \cite[Cor. A.1]{DS24} for $A=(\delta,\infty)$ and the functional
$\bar f_N(\bsigma,\bxs,\varphi_N) := \widehat\err_{N,T}(\bsigma,\bxs,\varphi_N;\bU_\infty)$
which is deterministic rotationally invariant,  see \eqref{eq:rot-inv}.  Indeed,   the required exponential
convergence of \cite[(A.2)]{DS24} is provided here by \eqref{eq:exp-err-LT},  the 
continuity in $\bxs$ of $\bar f_N(\bsigma,\bxs,\varphi_N)$ is obvious (see \eqref{eq:bnq},  \purple{\eqref{eq:err} and \eqref{eq:err-hat}),  while} Lemma 
\ref{lem:cont-varphi} provides the continuity in $\varphi_N$,  as required in \cite[(A.1)]{DS24}.
\end{proof}

\red{Towards the proof of Theorem \ref{thm-macro},  let $H^{\qo}_{\BJ}$ denote the centered field 
whose covariance been modified by conditioning 
on $\cptq(0,0,0,\qo)$.  Note that conditional on $\cptq(\V)$},  we have the representation
\begin{equation}\label{eq:H0-barH}
H_{\BJ}=H^{\qo}_{\BJ}+\bar{H}_{\V}\,,  \qquad \bar{H}_{\V} (\bx):=\E[H_{\BJ}(\bx) | \cptq(\V)]
\end{equation}
where,  see 
%$\V$-dependent 
Lemma \ref{lem:app}(b),  
% of the form 
\begin{equation}\label{eq:bH-V}
\bar{H}_{\V} (\bx)
% :=\E[H_{\BJ}(\bx) | \cptq(\V)]
= - N \bv(x,y) \quad \hbox{at} \quad
x =  \frac{1}{N} \langle \bx, \bxs \rangle \quad \hbox{and} \quad y =  \frac{1}{N} \langle \bx, \bxo \rangle.
\end{equation}
% As in \cite[Remark 1.4]{DS21},  
Further,  Lemma \ref{lem:G-conc} shows that \red{$H^{\qo}_{\BJ}(\cdot)$ are} 
well approximated \red{for $|\qo|<1$},
by the centered field \purple{$\HJs(\cdot)$} whose covariance \purple{corresponds 
to} conditioning on 
%the bigger event 
\begin{align}\label{eq:cond-star}
\cptq  :=\Big\{ \nabla_{\rm sp} H_{\BJ}(\bxs)= {\bf 0} \Big\} \supset \cptq(\red{0,0,0,\qo}) \,.
\end{align}

\begin{remark}\label{rem:DS21}
The conditions imposed on the critical point $\bxs$ of $H_{\BJ}$ by 
$\cptq(\V)$ of \eqref{eq:cond-J},  match the conditioning taken in \cite[Thm.~1.1]{DS21},  but 
here we further specify that $H_{\BJ}(\bxo)=-N \Ep$ (as needed to match sampling $\bxo$ from the 
Gibbs measure).  As reasoned in \cite[Remark 1.3]{DS21},  also here the conditional law of 
$\bU_N$ is invariant under the rotation $(\bxo,\bxs) \mapsto (\BO \bxo,\BO \bxs)$ for any 
non-random orthogonal matrix $\BO$.  Thus,  \abbr{wlog}  \red{we set hereafter}  
%fixing $q=q_N(0) \in (\qo-\epsilon,\qo+\epsilon)$,
$\bn=\qs \sqrt{N} \be_1$ and  \red{we can},  as in  \cite[Thm.  1.1]{DS21},
sample $\bxo$  in \eqref{eq:exp-err-LT} uniformly on the band
${\Bbb S}_{\bn} (\qop)$.  \red{Upon replacing $H^{\qo}_{\BJ}$ in \eqref{eq:H0-barH} by $\HJs$,  
we get},
as in \cite[(1.24)]{DS21},  that the covariance modification given $\cptq$ yields
the terms involving $L(\cdot)$ in Definition \ref{def:limit-dyn-low-temp},  which thus 
match those in \cite{DS21}.  In contrast,  \red{the extra conditioning on 
%the value of 
$H_{\BJ}(\bxo)$
affected} $\bar{H}_{\V}$,  with $\bv(x)$ of  \cite[(1.22)]{DS21} replaced here by $\bv(x,y)$.  Indeed,  
% given by \eqref{def:vt-new}.  
% Indeed,  in case $\qo=\qs^2$ it follows from \eqref{def:vt-new} and
% \eqref{def:w-s} that $\bv(x,x)$ coincides with $\bv(x)$ of \cite[(1.22)]{DS21}.  More generally,
the limiting equations of \cite[Thm. 1.1]{DS21} coincide with $\bUqf(\V)$ 
whenever the specified value of $\Ep$ results with $w_1=w_4=0$ in \eqref{def:w-s}.
% (with \eqref{def:w-s} then leading to $\bv(x,y)=\bv(x)$ of \cite{DS21}).  
In particular,  for $\qo=\qs^2$ this amounts to setting $\Ep=\Eg$,  whereas
for $\qo=0$ it amounts to setting $\Ep=0$.
\end{remark}

\subsection{\purple{Proof of Theorem \ref{thm-macro}: Challenges and key-steps}}\label{subsec:pf-macro}
Section \ref{sec:ex} 
utilizes stochastic calculus to establish in Proposition \ref{prop:existence}
the existence of a strong solution of \eqref{diffusion-phi},
which 
%is continuous under small perturbations of the potential $\varphi_N(\cdot)$ and 
in case of \eqref{diffusion} 
is further confined to $\B^N(r \sqrt{N})$ for some non-random $r \in (1,r_\star)$,   
up to an exponential in $N$ small probability.  In Proposition \ref{prop:cont} a
similar analysis yields the continuity in $\varphi_N$ of the solution of \eqref{diffusion-phi},  thereby establishing
Lemma \ref{lem:cont-varphi}.  Note that our approach to the \abbr{rsb}-Gibbs initial conditions 
of Theorem \ref{thm-uncond}(b) (via 
%the Gaussian conditioning of 
\cite[Thm. 1.4]{DS24}),  requires an a-priori precise pure state decomposition 
of $\widetilde{\mu}^N_{2\b_0,\BJ}$ in terms of \emph{thin spherical bands around 
critical points} (see \eqref{eq:band-dec}).  Apart from a few exceptions (e.g.  pure $p$-spins),  
this has only been established for generic mixtures and devising a proof for
Theorem \ref{thm-macro} for infinite mixtures,  is somewhat challenging.
Specifically,  though our proof of \cite[Thm. 1.1]{DS21} addresses a similar disorder dependent $\bxo$ 
(see  \red{Remark \ref{rem:DS21}}),  it can only handle 
finite mixtures.  Indeed,  in \cite{DS21} we adapt
the proof of \cite[Prop. 1.3]{BDG2},  where the key is to condition on $\Fa_t =\sigma(\bx_u, u \le t)$,  
and represent $\int_0^t k (\bx_u,\bx_s) \circ d \bz_s$ for $k (\bx,\by)=\E[\nabla H_{\BJ}(\bx) \nabla H_{\BJ}(\by)^\top]$ 
% and $u \in [0,T]$,  
as a weighted sum of $\Fa_t$-adapted It\^o stochastic integrals 
\abbr{wrt} the continuous semi-martingale $\bz_s := \E[\BB_s | \Fa_s]$.  While 
$k(\bx_u,\bx_s)$ is neither $\Fa_s$-adapted nor of finite variation,  
\emph{for finite mixtures} its entries are multivariate polynomials in the components of 
$\bx_u$ and $\bx_s$,  thereby allowing us to push the non-adapted part of each monomial 
out of the stochastic integral (see \cite[Appendix A,  \& after (3.10)]{BDG2}).  To 
prove our analog of \cite[Prop. 1.3]{BDG2} beyond finite mixtures thus involves
the considerable difficulty of doing so with non-adapted stochastic calculus,  or the 
delicate task of controlling well enough the tails of the infinite 
series of adapted integrals one encounters.  Instead,  in Subsection \ref{subsec:finite-m}
we reduce the proof of Theorem \ref{thm-macro} to that of our next proposition, 
about finite mixtures and everywhere locally Lipschitz $f'(\cdot)$
(also with only one nominal value $\V'=\V$ and a more convenient conditioning by $\cptq$).
 \begin{prop}\label{prop-macro-easy}
Fix finite $\b,T > 0$,  a \emph{finite mixture} $\nu(\cdot)$,  non-random $\bxs \in \qs \SN$ 
for $\qs \in [0, 1]$ and some \red{allowed $\V$ as in Definition \ref{ft:cons}.}
Consider the strong solution of 
\eqref{diffusion} with $f'(\cdot)$ locally Lipschitz \red{of sufficient growth rate 
(as in  \cite[(1.6) \& (1.10)]{DS21}),} and potential \emph{$H^{\star}_{\BJ}+\bar{H}_{\V}$} 
for $\bar{H}_{\V}$ of \eqref{eq:bH-V} and the conditional field $H_{\BJ}^\star$ 
(given the event $\cptq$ of \eqref{eq:cond-star}).  
Then,  for any $\delta>0$,
\begin{align}\label{eq:exp-err-no-sup}
% \lim_{\epsilon \to 0} 
\limsup_{N \to \infty}  
\frac{1}{N} \log \P \Big( \widehat\err_{N,T}\big(\bxo,\bxs,H_{\BJ}^\star+\bar{H}_{\V} ;\bUqf(\V)\big) > \delta \Big) &<  0 \,.
\end{align}
\end{prop}
Beyond bounding various Gaussian conditional expectations (see Appendix \ref{sec:Appendix}),  
our reduction of Theorem \ref{thm-macro} to merely showing Proposition \ref{prop-macro-easy}
relies on the following coupling,  in order to approximate well the Langevin diffusion 
for the potential associated with an infinite mixture $\nu(\cdot)$,  
by such diffusions that correspond to the truncated,  finite mixtures $\nu^{[m]}(\cdot)$. 
\begin{defi}\label{def:H-m-coupling}
Fixing $1<r_\star<\bar{r}$,  realizing via \eqref{potential}
the mixed $p$-spin Hamiltonian 
$H_{\BJ}^{[\infty]}:=H_{\BJ}(\cdot)$ for a model $\nu(\cdot)$ satisfying
\eqref{eq:r-star}-\eqref{eq:nudef},   we jointly realize
also the mixed $p$-spin models with truncated finite mixtures 
\begin{equation}\label{eq:nu-m-def}
\nu^{[m]}(r) := \sum_{p=2}^m  b_p^2 r^p 
\end{equation}
and the models $\nu^{\Delta}(r):= \nu(r)-\nu^{[m]}(r)$,  as
\begin{equation}\label{potential-m}
  H^{[m]}_\BJ(\bx) :=\sum_{p=2}^m b_p H^{(p)}_{\BJ} (\bx)  \quad \hbox{and} \quad 
  H^{\Delta}_\BJ(\bx):= H_\BJ(\bx)-H^{[m]}_\BJ(\bx)\, , \quad \hbox{respectively.}
\end{equation}
\end{defi}

Section \ref{sec:DS21} proves Proposition  \ref{prop-macro-easy} by an adaptation of \cite[\S 3 and 4]{DS21},   
to yield the required $e^{-\Omega(N)}$-decay of error probabilities.  This is supplemented by 
Section \ref{sec:cont-lim} where we re-run arguments from \cite[\S 6]{DS21} to establish 
various properties of the limit dynamics,  such as Proposition \ref{prop:ell-lim},  and the 
following continuity of the limit dynamic $\bUqf(\V)=\bUqf(\V;\nu)$ of Definitions \ref{def:limit-dyn-low-temp} 
and \ref{def:limit-dyn-high-temp} \abbr{wrt} the model mixture $\nu(\cdot)$ (which is part of 
our reduction of Theorem \ref{thm-macro} to Proposition \ref{prop-macro-easy}).
\begin{prop}\label{prop:nu-cont}
Fix finite $\b,T$,  $1<r_1<\bar{r}$,  $f'(\cdot)$ globally Lipschitz on $[0,r_1^2]$
and models $\nu$,  $\nu^{[m]}$ as in Definition \ref{def:H-m-coupling}.  For $\qs \in [0,1]$ and
\red{$\V$ as in Definition \ref{ft:cons}, } the corresponding limiting dynamics satisfy
\begin{align}
\label{eq:m-cont-Uf}
\limsup_{m \to \infty} \red{\| K^{[m]} \|_T} < r_1^2 \quad \Longrightarrow \quad 
\lim_{m \to \infty} \| \,  \bUqf(\V;\nu^{[m]}) - \bUqf (\V;\nu) \,\|_{T} &=0 \,.
\end{align}
\end{prop}
\begin{remark}\label{rmk:Uf-bd}
\red{Here $K^{[m]}(s)=C^{[m]}(s,s)$ is given by \eqref{eqfZs} for the model
$\nu^{[m]}$ and} at the end of  Section \ref{sec:ex} we show that the \abbr{lhs} of \eqref{eq:m-cont-Uf} holds
for any locally Lipschitz $f'(\cdot)$ satisfying \eqref{eq:fdef} and for some $r_1<r_\star$.
\end{remark}

Many arguments in \cite{BDG2,DGM,DS21} take advantage of having $f'(\cdot)$ Lipschitz 
on compacts.  However,  for the a.s.  global existence of solutions for \eqref{diffusion} with mixtures having
finite radius of convergence $\bar r^2$,  we must have $f'(r) \uparrow \infty$ at some $r_\star^2 \le \bar r^2$.  
Nevertheless,  with $r_\star<\bar r$,  it is relatively easy to confine such solutions
to $\B^N(r_1 \sqrt{N})$ for some $r_1<r_\star$,  whereupon $f'(\cdot)$ 
outside this set no longer matters,  and the relevant arguments 
of \cite{BDG2,DGM,DS21} thus apply in dealing with the truncated (finite) mixtures, 
or with the limiting dynamics.

\subsection{\red{Notational conventions}}\label{subsec:gloss}
Large parts of this paper involve adapting and citing results from \cite{DS21},  which in turn builds on 
\cite{BDG2}.  To make this as seamless as possible,  we adopt many of the notational conventions of
\cite{BDG2,DS21}.  Specifically,  we use throughout $\BB_t$,  $\BW_t$ and 
$W_t$ for the $N$-dimensional and $1$-dimensional Brownian motions,  
$\bx_t=(x_t^i)_{1 \le i \le N} \in \R^N$ for the Langevin path
at inverse-temperature parameter $\beta$,  subject to the Hamiltonian $H_{\BJ}(\bx)$ of mixture $\nu(\cdot)$,  
which we represent via the independent centered Gaussian vector $\BJ=\{J_{(i_1,i_2,\ldots,i_p)}\}$.
Similarly,  $\qs$ denotes the (rescaled) norm of the critical point $\bxs$ at the center of the 
relevant spherical bands ${\Bbb S}_{\bxs} (\cdot) \subset \SN$ with the corresponding parameters 
vector $\V$
%=(\Ep,\Eg,\Gg,\qo)
consisting of the specified energy $\Ep$ at $\bxo$ together with
$(\Eg,\Gg,\qo)$ of the conditioning event $\cpt(\cdot)$ from \cite[Thm.1.1]{DS21} (which 
amounts to our $\cptq(\V)$ except for not specifying the energy $\Ep$).  Indeed,  our gradient-covariance kernel 
is the same $k(\cdot,\cdot)$ of \cite{BDG2} when $\qs=0$,  and otherwise it is $\widetilde{k}(\cdot,\cdot)$ of \cite{DS21}.  We also deal here with the limit as $N \to \infty$ of same random functions $(C_N,\chi_N,q_N)$,  
amended in the sequel by  $(K_N,D_N,E_N,A_N,F_N,Q_N,V_N)$,  whose meaning is exactly as in \cite{DS21}.  
While doing so,  as in \cite{DS21} we denote the collection of all relevant functions by $\Ua_N$,  
using $U_N$ for a ``generic" function and $\bar{U}_N$ when the coordinate averaging in $U_N$ excludes
the one aligned with the normalized $\hbxs(=\be_1)$.  Going back to \cite{BDG2},  we use $U_N^a$ for 
the expectation of such a function,  $\widehat{U}_N$ for its expectation 
given $\Fa_t = \sigma(\bx_u,  u \le t)$,  omitting the subscript $N$ when referring to its limit (point) as $N \to \infty$.
Unlike \cite{BDG2,DS21} our forces and Hamiltonians blow-up,  respectively,  
for $\bx$ at the boundary of open balls $\B^N_{\star} \subset \B^N$ of normalized radii $1<r_\star<\bar r$,  
with $\{\bx_t: t \in [0,T]\}$ confined,  up to $e^{-\Omega(N)}$-probability,  to a smaller ball,  parametrized by $1<r_1<r_\star$.  Here $\bxo$ is sampled from the Gibbs measure at parameter $2 \beta_0$,  to which 
correspond
% partition functions $\widetilde{Z}_{2\beta_0,\BJ}$, 
free and ground state energies $F(\beta_0)$ and $\GS(\cdot)$, 
critical parameters $\beta_c^{\rm dyn} < \beta_c^{\rm stat}$,  
a Parisi measure $\zeta_P$ and its associated Edwards-Andreson parameter $q_{\rm EA}$.
Treating such Gibbsian $\bxo$
via the approach of \cite{DS24},  requires us to  
deal with Langevin system \eqref{diffusion-phi} of general potential $\varphi_N$ on $\B^N$,  possibly perturbed by
small linear combinations of some potentials $\{\psi_N^{(i)},  i \le d\}$.  We use $\varphi_N^{(1)}$
and $\varphi_N^{(2)}$ when comparing such potentials and 
$\widetilde{\varphi}_N(s)$ for the normalized potential along the Langevin path,  
specialized to $\widetilde{H}_N(s)$ in case $\varphi_N=H_{\BJ}$.
In addition,  $\|\cdot\|_\infty$ and $\|\cdot\|_L$ denote the supremum and Lipschitz 
norms of such $\varphi_N$,  its gradient ($\nabla \varphi_N$) or Hessian 
(${\rm Hess}(\varphi_N)$).  We likewise have 
$\|g\|_\infty^{(r)}$,  $\|g\|_L^{(r)}$ for functions $g:([-r^2,r^2])^d \to \R$,  with
$\|g\|^{(k)}_\star$ a norm for the uniform convergence of $g$ and its first $k$ partial derivatives
(so $\|g\|_\star:=\|g\|_\star^{(0)}=\|g\|_\infty^{(r)}$),  while $\|\cdot\|_T$ denotes the supremum norm on $[0,T]^d$.
Similarly,  $\|\cdot\|_\infty$ and $\|\cdot\|$,  denote 
the supremum and Euclidean norms on $\R^{N'}$,  respectively.  The approach of \cite{DS24} also relies on continuity
with respect to conditioning,  replacing criticality of $H_\BJ$ at $\bxs$ with specified values at 
$(\bxo,\bxs)$ and $\sp\{\bx_0,\bx_\star\}$-directional derivatives at $\bxs$,  by event 
$\cpteq(\Ve)=\cpt_{\qs}^{{\rm e},[\infty]}(\Ve)$,  
where the $(N+2)$-dimensional relevant Gaussian vector $\Ha=\Ha^{[\infty]}=(\hat{\Ha},\Ha_\perp)$ matches some
specified $\Ve=(\hat{V},{\bf u}) \in \R^4 \times \R^{N-2}$ (which in turn 
is in an $\epsilon$-neighborhood of $(\Ep,\Eg,\Gg,0) \times \{\bf 0\}$).
A generic limit of
$(C_N,\chi_N,q_N,\widetilde{H}_N)$ is denoted by $\bU_\infty$,  using $\err_{N,T}(\cdot;\bU_\infty)$ for 
a $\|\cdot\|_T$-type distance from such a limit,  per path $\bx_\cdot$,  critical point $\bxs$ and potential. 
We denote by $\widehat{\err}_{N,T}(\cdot;\bU_\infty)$ the expectation of such distance over the
Brownian motion and the induced path (which thus depends only on $\bxo$,  $\bxs$ and the potential),
with $\Delta \widehat{\err}^{[m]}_{N,T}(\BJ)$ the maximum over $\bU_\infty$ error in the latter quantity,  
when replacing $H_{\BJ}$ by its finite mixture approximation $H_{\BJ}^{[m]}$.
Reserving $\bUqf(\V;\nu)=\bU_\star$
for the limit vector corresponding to mixture $\nu$,  allowed parameters $\V$ of $\qs$ 
and a given choice of $f(\cdot)$,  leads to $\bUsp(\V;\nu)$ as $\ell \to \infty$, 
where our forces $f_\ell(\cdot)$ replace the $f_L(\cdot)$ of
% that been used for finite mixtures in 
\cite{DGM,DS21},  while adopting their convention of indicating by $U^{(\ell)}$ the dependence on $f_\ell$
of the limit of some generic $U_N$.  In detailing $\bUqf$ and $\bUsp$ we borrow the notations 
$\mu(s)$,  $\psi(r)$ and $\bv(\cdot)$ from \cite{DS21},  but seeking a more concise description of 
those limits,  we also set $C_o(s)=C(0,s)$,  introduce the $4$-dimensional matrices $\Sigma_\nu$ 
and lump certain terms into new functions $L(s)$,  
$\underline{{\sf A}} = ({\sf A}_R,{\sf A}_C,{\sf A}_q,{\sf A}_H,{\sf A}_K,{\sf A}_o)$
and the corresponding integrals $\widehat{{\sf A}}_q (s)$,  $\widehat{{\sf A}}_C (s,t)$.
In deriving properties of $\bUsp$ we re-use notations of \cite[Sec. 2]{DS21},  
such as $(R_{\fdt},C_{\fdt})$,  $\kappa_1$,  $\kappa_2$,  $\alpha$,  $\gamma$,  $\theta(x)$, 
or modified ones (such as $\mu_\star$ and $\phi_\gamma$
instead of $\mu$ and $\phi$,  or $c_\gamma(\infty)$ instead of $D_{\infty}$).
With some abuse of notation we use $\tau_k=\tau_{r_k}$ and $\tau_c$ for
first hitting times in terms of the norm of a stopped version
$\bx^{\langle k \rangle}_t$ of $\bx_t$,  likewise denoting by
$\hat{\tau}_k=\hat{\tau}_{r_k}$ and $\hat{\tau}_r^{(\pm)}$ such hitting times 
for certain bounding scalar diffusions.  Throughout the paper,  superscript $[m]$ indicates objects determined 
by the approximating $\nu^{[m]}$,  such as
% $H_{\BJ}^{[m]}$, 
$\Ha^{[m]}$, $\bx_t^{[m]}$, $\cptq^{[m]}$,
$\cpteqm$, $\bar H^{[m]}_{\Ve}$,
$H_\BJ^{{\qo},[m]}$,  $H^{\star,[m]}_{\BJ}$,
$\tau_r^{[m]}$, $\tau_r^{\star,[m]}$,
$(\mu^{[m]},\bv^{[m]},R^{[m]},C^{[m]},q^{[m]},K^{[m]},H^{[m]})$,  
with superscript $\Delta$ for the corresponding differences (e.g.  $\nu^{\Delta}$,  $H_{\BJ}^\Delta$,  
$\Ha^{\Delta}=(\hat{\Ha}^\Delta,\Ha^\Delta_{\perp})$).
We represent $H_{\BJ}$ given $\cpteq(\Ve)$ as the sum of non-random $\bar{H}_{\Ve}$ and centered field 
$H^{\qo}_{\BJ}$,  likewise splitting $H_{\BJ}^{[m]}$ given $\cpteqm$ and use 
$\widehat{H}_\BJ^{{\qo},[m]}$,  $\widehat{H}^{[m]}_{\Ve}$ when splitting $H_{\BJ}^{[m]}$ given $\cpteq(\Ve)$.
In analyzing finite mixtures we utilize the disorder sup-norm $\|\BJ\|^N_\infty$ of \cite{BDG2}.  Then,
approximating $H^{\qo}_{\BJ}$ by $H_{\BJ}^\star$ induces
on the space $\Ea_N$ of triplets $(\bxo,\BJ,\BB)$ the law $\P_\star$ of \cite{DS21},  
so as in \cite{BDG2},  the elements of $\Ua_N$ are $O(\frac{1}{\sqrt{N}})$-Lipschitz 
under the weighted Euclidean norm $\|\cdot\|$ on $\La_{N,M} \subset \Ea_N$ whose
complements have $e^{-\Omega(N)}$ small probability.

\section{Stationarity and relaxation,  \abbr{fdt} regime and localized states with no-aging}\label{sec-asymp}

\red{Section \ref{subsec:stat-pf} provides proofs for Section \ref{subsec:stat},  while Sections \ref{subsec:fdt} 
and \ref{subsec:no-aging} analyze the \abbr{fdt} regime and the possible localized states with no-aging,  
respectively}.

\subsection{Proofs of Proposition \ref{prop-stat} and of Corollary \ref{cor:relax}}\label{subsec:stat-pf} 

\begin{proof}[Proof of Proposition \ref{prop-stat}] Starting with $\qs>0$ recall 
the notations $\qo=\alpha \qs$, $b_\alpha=\nu'(\qo)/\nu'(\qs^2)$ and
consider stationary dynamics of the form $\bUsp(\cdot)$.
With $C(s,s)=1$,  plugging \eqref{eqZs-new} into \eqref{eqCs-new} at $s=t$, yields that
\begin{equation}\label{eq:cp-diag}
\partial_s C(s,t) \Big\vert_{s=t} = - \frac{1}{2} \,.
\end{equation}
Such stationary dynamics are of the form $C(s,t)=c(s-t)$,  $R(s,t)=r (s-t)$,  $q(s)=\qo$ and $H(s)=H(0)$,
\red{with $c(0)=r(0)=1=-2c'(0)$,  in view of \eqref{eq:str-limit} and  \eqref{eq:cp-diag}.  Further,  
from \eqref{eqRs-new} at $t=s$ we deduce that $r'(0) = - \mu(s) r(0)$,  so
necessarily $\mu(s)=\mu_\star$ is independent of $s$,  while} from \eqref{eqL} we have that 
\[
L(s)=b_\alpha \int_0^s r(u) du \,,   \qquad \forall s \ge 0 \,.
\]
From \eqref{eqH-new}, for such dynamics we have that for all $s \ge 0$,
\begin{align*}
H(s) = \b  \int_0^s r(u) [\nu'(c(u)) -b_\alpha \nu'(\qo) ] \, du +  \bv (\qo,c(s)) \,. 
\end{align*}
Our requirement of having $H'(s)=0$, thus amounts to 
\begin{equation}\label{eq:H-cnst}
 \b r(s) [\nu'(c(s))-b_\alpha \nu'(\qo)] = - c'(s) \bv_y(\qo,c(s)) \,.
\end{equation}
If $|\alpha|=1$ then \red{$\bxs=\qo \bxo$ so $q_N(s)=C_N(s,0)$ for all $N$ and $s \ge 0$.  It follows that
$q(s) = \qo C(s,0)$,  so any stationary dynamic must then have} 
$c'(0)=0$,  which we have already ruled out.  Proceeding hereafter with $|\alpha|<1$, 
recall \eqref{def:w-s} that $w_4 = - b_\alpha \sqrt{\qs^2 -\qo^2} \, w_1$,  hence from \eqref{def:vt-new} 
% and having $\bv(x,y)$ 
we get that 
\begin{align}\label{eq:bv-y}
\bv_y(\qo,c(s)) = w_1 [ \nu'(c(s)) - b_\alpha \nu'(\qo) ] \,.
\end{align}
Since $c'(0) \ne 0$ and $r(0)=-2 c'(0)$, upon plugging \eqref{eq:bv-y} in \eqref{eq:H-cnst} 
we deduce that $H(s)=H(0)$ if and only if 
\begin{equation}\label{eq:r-c-L}
r(s) = - 2 c'(s) \,, \qquad L(s) = -2 b_\alpha (c(s)-1) \,, \qquad  w_1 = 2 \b \,.
\end{equation}
Now \eqref{eqRs-new} reads as
\[
c''(\tau) = - \mu_\star c'(\tau) - 2 \b^2 \int_0^\tau c'(\tau-v) c'(v) \nu''(c(v)) dv \,,
\]
and upon integrating both sides, necessarily
\begin{equation}\label{eqCs-fdt}
c'(\tau) = - \mu_\star c(\tau) - 2 \b^2 \int_0^\tau c(\tau-v) c'(v) \nu''(c(v)) dv + 2 \b^2 c(0) \nu'(c(\tau)) + \gamma - \frac{1}{2} \,,
\end{equation}
where since $c(0)=1$ and $c'(0)=-1/2$, we also know that
\begin{equation}\label{eq:gamma-mu}
\gamma = \mu_\star - 2 \b^2 \nu'(1) \,.
\end{equation}
Integrating by parts in \eqref{eqCs-fdt}, it is easy to verify that $c(\cdot)$ 
must then be the unique $[0,1]$-valued, continuously differentiable solution 
$c_{\gamma}(\cdot)$ of \eqref{FDTDb-new}.  Further,  plugging \eqref{def:AC-new}
in \eqref{eqCs-new} at $s=t+\tau$,  and substituting
%(and $v=|t-u|$), 
the formulas for stationary dynamics $(R,C,q,H)$ as above,  results with 
\begin{align}\label{eqCs-fdt-alt}
c'(\tau)= & - \mu_\star c(\tau) - 2 \b^2 \int_0^\tau c(\tau-v) c'(v) \nu''(c(v)) \, dv 
- 2 \b^2 \int_0^t \frac{d}{dv} [c(v) \nu'(c(v+\tau)]
% [c'(v+\tau) \nu''(c(v+\tau)) c(v) + c'(v) \nu'(c(v+\tau))]
 \, dv \nonumber \\ 
 & + 2  \b^2 \qo \nu''(\qo) b_\alpha (c(s)-1) + 2 \b^2 \nu'(\qo) b_\alpha (c(t)-1)
+ \b \qo \bv_x (\qo,c(s)) + \b c(t) \bv_y (\qo,c(s)) \,.
\end{align}
Carrying out the second integration in \eqref{eqCs-fdt-alt} and plugging there \eqref{eq:bv-y},  
leads to 
\eqref{eqCs-fdt} if and only if for any $s \ge 0$, 
\begin{align}\label{eqCs-fdt-id}
\gamma-\frac{1}{2} =&  \b \qo [\bv_x (\qo,c(s)) +2 \b \nu''(\qo) b_\alpha c(s) ] - 2 \b^2 \psi(\qo) b_\alpha
% + \b c(t) [\bv_y (\qo,c(s)) - 2\b \nu'(c(s))+ 2 \b \nu'(\qo) b_\alpha ] 
\,.
\end{align}
Now, similarly to \eqref{eq:bv-y}, we get from \eqref{def:vt-new} and \eqref{def:w-s} that 
\begin{align}\label{eq:bv-x}
\bv_x(\qo,c(s)) &= w_2  \nu'(\qo) + (\qs^{-2} w_3 + b_\alpha \qo \qs^{-2} w_1) \psi(\qo) 
- w_1  \nu''(\qo) b_\alpha c(s) \,.
% \bv(\qo,c(s))-\bv(\qo,1) &= w_1 [\nu(c(s))-\nu(1) -b_\alpha \nu'(\qo)  (c(s)-1) ] 
\end{align}
Plugging \eqref{eq:bv-x} in \eqref{eqCs-fdt-id} and utilizing that $w_1 = 2 \b$ one confirms 
that the \abbr{rhs} of \eqref{eqCs-fdt-id} is independent of $s$. Hence, it suffices to consider
\eqref{eqCs-fdt-id} only at $s=0$. That is, to have
\begin{equation}\label{eq:gamma-fdt1}
\gamma-\frac{1}{2} = \b \qo \bv_x(\qo,1) +\b \bv_y(\qo,1)-2 \b^2 \nu'(1)
%= \b \qo [ w_1 \qs^2 \nu'(\qo) + w_2 \psi(\qo) ]  + \b w_3 (\qo^2 -\qs^2) \psi(\qo) b_\alpha
\,.
\end{equation}
Further, having $q(s) \equiv \qo$ implies in view of \eqref{eqqs} and \eqref{def:Aq-new}
that
\begin{align}\label{eq:qfdt}
0 = & -\mu_\star \qo - 2 \b^2 \qo \int_0^s c'(v) \nu''(c(v)) dv 
 +2 \b^2 \qs^2 \nu''(\qo) b_\alpha (c(s) -1)  
+ \b [ \qs^2 \bv_x (\qo,c(s)) + \qo \bv_y(\qo,c(s))] \,.
\end{align}
In particular, considering \eqref{eq:qfdt} at $s=0$ we find that
\begin{equation}\label{eq:gamma-fdt2}
\mu_\star \qo =  \b \qs^2 \bv_x(\qo,1) + \b \qo \bv_y(\qo,1) \,.
% = \b \qs^2 [w_1 \qs^2 \nu'(\qo) + w_2 \psi(\qo)] + 2 \b^2 (\qo^2 -\qs^2) \nu''(\qo) b_\alpha 
\end{equation}
Indeed, upon carrying out the integration in \eqref{eq:qfdt} and plugging there the values from
\eqref{eq:bv-y}
%\eqref{eq:gamma-mu} 
and \eqref{eq:bv-x}, we see that \eqref{eq:qfdt} is independent of $s$, hence 
% \eqref{eq:qfdt} is 
equivalent to \eqref{eq:gamma-fdt2}.
Comparing \eqref{eq:gamma-fdt1} and \eqref{eq:gamma-fdt2} yields that 
\begin{equation}\label{eq:v-x-stat}
\bv_x(\qo,1)=\frac{\qo}{2\b (\qs^2-\qo^2)} \,.
\end{equation}
In particular, plugging \eqref{eq:bv-y} and \eqref{eq:v-x-stat} in \eqref{eq:gamma-fdt1}, confirms 
that for any stationary dynamics,  the value of $\gamma=\gamma_\star$ must be given by  
the formula \eqref{def:gamma-star}.  To recap,  we have characterized the collection of 
stationary dynamics of the form $\bUsp(\cdot)$  as stated in 
Proposition \ref{prop-stat},  and have shown that such solutions appear if and only if 
\eqref{eq:v-x-stat} holds and $w_1 = 2 \b$.  It thus remains only to show 
that the latter two constraints are equivalent to \eqref{def:stat-EG}. Indeed, from 
\eqref{eq:bv-x} we deduce that 
%the \abbr{lhs} of \eqref{eq:v-x-stat} is then the linear function 
\[
\bv_x(\qo,1) = w_2 \nu'(\qo) + \qs^{-2} w_3 \psi(\qo) + 2\b b_\alpha [\qo \psi(\qo) \qs^{-2} - \nu''(\qo)] \,,
\]
whereby upon substituting this and $w_1 = 2 \b$ into \eqref{def:w-s}, we reduce the 
latter linear system to 
\begin{align}\label{def:w-ker}
	\begin{bmatrix} 
	\Ep - 2\b ( \nu(1) - (1-\alpha^2) \nu'(\qo) b_\alpha) \\
	 \Eg - 2 \b \nu(\qo) \\ \Gg - 2 \b \qo \nu'(\qo) \qs^{-2} 
	\\ \bv_x(\qo,1) - 2\b b_\alpha [\qo \psi(\qo) \qs^{-2} - \nu''(\qo)]
	 \end{bmatrix}  &=  \begin{bmatrix}
	 \nu(\qo) & \qo \nu'(\qo)  \\
	 \nu(\qs^2)       & \qs^2 \nu'(\qs^2)  \\
	 \nu'(\qs^2)  & \psi(\qs^2)  \\
	 \nu'(\qo) & \psi(\qo) 
	\end{bmatrix}
	\begin{bmatrix} w_2 \\ \qs^{-2} w_3 \end{bmatrix} \,.
	\end{align}
That is, in \eqref{def:w-ker} the \abbr{lhs} must be in the image of the columns of the matrix
on the \abbr{rhs}. To arrive at \eqref{def:stat-EG}, multiply the third and forth rows by 
$\qs$ and substitute $\qo=\alpha \qs$. 
	
Next,  by Cauchy-Schwarz,
\[
g'_{\b} (\alpha^2) = \frac{1}{2}-\gamma_\star + \frac{2 \b^2}{\nu'(\qs^2)} [ \nu'(\alpha^2) \nu'(\qs^2) - \nu'(\alpha \qs)^2] 
\ge \frac{1}{2} - \gamma_\star \,,
\]	
with equality only if the mixture $\nu(\cdot)$ is a pure $p$-spins,  or when $\alpha=\qs$ (or possibly
having $\alpha=-\qs$ in case $\nu(\cdot)$ is either even or odd function).  The latter inequality implies,
in view of \eqref{eq:c-inf},  that $c_{\gamma_\star}(\infty) \ge \alpha^2$ with equality if and only if 
$g'_{\b}(x) < g_{\b}'(\alpha^2)=\frac{1}{2}-\gamma_\star$ for all $x \in (\alpha^2,1]$,  as stated.
		
In case $\qo=\qs^2$,  the kernel of the matrix on the \abbr{rhs} \purple{of \eqref{def:w-ker}}
is the linear span of $[0, 0, 1, -1]$
and $[1, -1,0, 0]$,  to which the image of its columns is orthogonal.  In this case 
$b_\alpha=1$ and $\alpha=\qs$, so the two constraints are then \red{precisely
\eqref{eq:stat-G-qs} and \eqref{eq:stat-EpE-qs}}.
%\begin{align}\label{eq:stat-G-qs}
%\Gg &= \frac{1}{2\b (1-\qs^2)} + 2 \b (1 - \qs^2) \nu''(\qs^2)  \,,  \\
%\Ep - \Eg &= 2\b [ \nu(1) - (1-\qs^2) \nu'(\qs^2) -  \nu(\qs^2)] \,, 
%\label{eq:stat-EpE-qs}
%\end{align}
%as stated in \eqref{G-stat} and \eqref{Ep-E},  respectively.

We conclude by examining the simpler case of stationary dynamics of the form $\bUosp(\cdot)$.  
As $C(s,s)=1$,  plugging \eqref{eqZs-new} in \eqref{eqCs-new} at $s=t$,  leads to 
\eqref{eq:cp-diag} and hence $c(0)=1$,  $r(0)=-2 c'(0) =1$.  Now,  from \eqref{eqH}, 
\[
H'(s) = \bv' (c(s)) c'(s) + \b r(s) \nu'(c(s)) \,.
\]
Thus, with $\bv'(y)=\Ep \nu'(y)/\nu(1)$, the requirement $H'(s)=0$ amounts to 
$r(s)=-\Ep/(\b \nu(1)) c'(s)$. Considering $s=0$ we see that necessarily $\Ep=2\b \nu(1)$
and $r(s)=-2 c'(s)$, as stated.  For such stationary dynamics,  we deduce from \eqref{eqRs-new} by 
the same reasoning as before,  that $\mu(s)=\mu_\star$ 
%is independent of $s$ 
and \eqref{eqCs-fdt}-\eqref{eq:gamma-mu} hold here as well.  That is,
the stationary dynamics must again be $c(\cdot)=c_\gamma(\cdot)$.  The only difference is that 
utilizing here \eqref{def:ACs} in lieu of \eqref{def:AC-new},  results with the second line 
in \eqref{eqCs-fdt-alt} being now merely 
$\b c(t) \bv'(c(s)) = 2 \b^2 c(t) \nu'(c(s))$. Consequently, \eqref{eqCs-fdt-id} is replaced in this 
case by $\gamma - \frac{1}{2} = 0$, as stated. To complete the proof, one merely verifies from
\eqref{eqZs-new} and \eqref{def:ACs},  that for such a dynamics,  at any $s \ge 0$,
\[
\mu(s) = \frac{1}{2} - 2 \b^2 \int_0^s c'(v) \psi(c(v)) \, dv + \b c(s) \bv' (c(s)) = \frac{1}{2} + 2 \b^2 \nu'(1) \,.
\]
That is,  a dynamics of the form $\bUosp(\cdot)$ is stationary if and only if $\Ep=2\beta \nu(1)$.
\end{proof}

\begin{proof}[Proof of Corollary \ref{cor:relax}] From Propositions \ref{prop:ell-lim} and \ref{prop-stat},  
we deduce that in the setting considered here,  
as $N \to \infty$ and then $\ell \to \infty$,  the stationary stochastic process $C_N(t+\tau,t)$
follows the \abbr{fdt} solution $c_{1/2}(\cdot)$ in the \abbr{rs} case
(ie.  $\b \le \b_c^{\rm stat}$),  whereas for strictly $1$-\abbr{rsb} generic model at 
$\b > \b_c^{\rm stat}$ we have that $\qo=\qs^2=q_\b$ for $q_\b=q_{\rm EA}(\b)$,  and $C_N(t+\tau,t)$
follows the \abbr{fdt} solution $c_{\gamma}(\cdot)$ for $\gamma=\gamma_\star(\qs,\qs)=\frac{1}{2}-g_\b'(q_\b)$ 
matching $\gamma_\b$ of \eqref{def:gamma-star2}  (with \eqref{eq:stat-rel} a direct consequence of \eqref{eq:c-inf}).
Further,  taking $\bxs=\bxs(\bxo,H_\BJ)$ as in 
Theorem \ref{thm-uncond}(b) when $\b > \b_c^{\rm stat}$ and otherwise setting $\qo=0$,  we
have that in both cases $q_N(\tau) \to q_\b$ uniformly over $\tau \in [0,T]$.  
Thus,  as stated, 
here both \eqref{def:fast-conf} and \eqref{def:fast-rel} hold (at $\alpha=\qs$),  if and only if 
$c_{\gamma_\b}(\infty)=q_\b$.  In the \abbr{rs} case this amounts to 
$O_N(1)$-relaxation time (namely,  $\alpha=0$),  with $c_{1/2}(\infty)=\qdyn(\b)=0$ whenever $\b<\b_c^{\rm dyn}$
but not for $\b>\b_c^{\rm dyn}$ (see \eqref{eq:betac-gap} for the situation when $\b=\b_c^{\rm dyn}$).
In the strictly 1-\abbr{rsb} setting,  where $q_\b>0$,  by definition $c_{\gamma_\b}(\infty)=q_\b$
is equivalent to $g_{\b}'(q+(1-q) x) < g_{\b}'(q)$ for $q=q_\b$ and any $x \in (0,1]$.  It is easy to check that 
\[
(1-q) [g_{\b}' (q+(1-q)x) - g_{\b}'(q)] = 2 \b^2 (1-q)[\nu'(q+(1-q)x)- \nu'(q)] + \frac{1}{2} - \frac{1}{2(1-x)} = g'_\b(x;\nu_q)\,,
\]
for $g_\b(\cdot)$ of \eqref{def:g-beta} and the mixture $\nu_q$ of \eqref{def:nu-q}.  Thus,  
$c_{\gamma_\b}(\infty)=q_\b$ whenever $\b<\b_c^{\rm dyn}(\nu_{q_\b})$ of \eqref{def:bc-dyn}
but not for $\b > \b_c^{\rm dyn}(\nu_{q_\b})$,  as claimed.
\end{proof}

\subsection{The \abbr{fdt} regime}\label{subsec:fdt}
The large time asymptotic of $\bUosp(0)$,  namely the \abbr{CKCHS} equations for general mixture  $\nu(\cdot)$,
and in particular its \abbr{fdt} regime,  are considered in  \cite[Prop. 6.1]{DGM}.  More generally,  such analysis 
is carried out in \cite[Section 2]{DS21} for the limiting spherical dynamics of \cite[Prop. 1.6]{DS21}.  
Fixing $\qs>0$ and $\alpha_o:=\qo/\qs$,  we proceed here with a similar analysis for $\bUsp(\V)$.
That is,  assuming the \abbr{FDT}-ansatz of the existence of the limit
\begin{equation}\label{as:fdt}
\lim_{t \to \infty} (R(t+\tau,t),C(t+\tau,t),q(t),C(t,0)) = (R_{\rm fdt} (\tau),C_{\rm fdt} (\tau),\alpha \qs,\widehat{\alpha}\alpha_o)\,,
\end{equation}
for the solution $\bUsp(\V)$ of \eqref{eqRs-new}-\eqref{def:w-s} we proceed to characterize the 
possible 
%values of the 
\abbr{rhs} of \eqref{as:fdt}.  Indeed,  this is precisely the ansatz of \cite[(2.1)]{DS21},  except for 
requiring here also the convergence of $C(t,0)$ (which was irrelevant in \cite{DS21}).  
% the value of $\Ep$ led to $w_1=w_4=0$ in \eqref{def:w-s},  see Remark \ref{rem:DS21}).
%  assuming now in addition to $q(s) \to \alpha \qs$ of  \cite[(2.1)]{DS21},  that as $s \to \infty$,   
% \begin{equation}\label{as:Co-conv}
% C_o(s)=C(s,0) \to \widehat{\alpha} \alpha_o \,,
% \end{equation}
% for some $\widehat{\alpha}$.  
It was shown in \cite[Prop. 2.1]{DS21},  that even in the possible presence of aging,  one family
of such solutions be $(R_{\rm fdt},C_{\rm fdt})=(-2c_\gamma',c_\gamma)$,  for the 
unique solution $c_\gamma$ of \eqref{FDTDb-new} and some $\gamma \in \R$.  

Proceeding  hereafter under the physics prediction of having this \abbr{fdt} solution,  we examine the 
dependence of $(\alpha,\widehat{\alpha},\gamma)$ on the model parameters 
($\nu(\cdot)$,  $\b$,  $\qs$ and $\V$),  and what the various values of 
$(\alpha,\widehat{\alpha},\gamma)$ mean in terms of \eqref{diffusion} (or alternatively,  
in terms of \eqref{sphere-diffusion}).  First,  note that as in \cite[Remark 2.4]{DS21},
having $\alpha \ne 0$ corresponds to $\bx_t$ confined within $O_N(1)$-times to the band  
${\Bbb S}_{\hbxs} (\alpha)$ of \eqref{eq:subsphere},  which for $\alpha=\qs$ is 
centered at the critical point $\bxs$ of $H_\BJ(\cdot)$.  Likewise,  
having $C(t,0) \to \widehat{\alpha} \alpha_o$
amounts to the $O_N(1)$-relaxation $\bar C_N(\tau,0) \to b:=
(\widehat{\alpha}-\alpha) \alpha_o$,  for $N,\ell,\tau \to \infty$ as in Definition \ref{def:rel},  where
\begin{equation}
\bar C_N(s,t): = C_N(s,t)-\qs^{-2} q_N(s) q_N(t) \,,
\label{eq:barC}
\end{equation}
denotes the state correlation within the subspace perpendicular to $\bxs$.  Recall Remark \ref{rem:non-neg},
that the function 
%$(s,t) \mapsto \E[\bar{C}_N(s,t)]$ are non-negative definite functions.  This is retained in 
%the limit as $N,\ell \to \infty$,  hence the non-negative definiteness of
$\bar{C}(s,t)$ from $\bUsp (\cdot)$ is non-negative definite,  so
% In particular,  
for any $k<\infty$,  the limit ${\bf C}_{k}$ as $t \to \infty$,  then $\tau \to \infty$,  of the 
$(k+1)$-dimensional matrices $\{ \bar{C}(s_i,s_j) \}$ with $s_i \in \{0,t+i \tau,  1 \le i \le k\}$
is non-negative definite.  Under our assumption \eqref{as:fdt} with
the predicted \abbr{fdt} solution $c_\gamma(\cdot)$ of \eqref{FDTDb-new}, 
the first diagonal term of ${\bf C}_k$ equals $(1-\alpha_o^2)$,  with all other elements 
on its first row and column are $b$,  while upon deleting its first row and column 
the matrix ${\bf C}_k$ becomes $(1-c_{\gamma}(\infty)) {\bf I}_{k} + (c_{\gamma}(\infty)-\alpha^2) {\bf 1} {\bf 1}^T$.  
It is easy to check that such a matrix ${\bf C}_k$ remains non-negative definite as $k \to \infty$,  if and only if  
\begin{equation}\label{eq:wh-alpha-range}
% b^2 = 
(\widehat{\alpha}-\alpha)^2 \alpha_o^2 \le (c_{\gamma}(\infty)-\alpha^2) (1-\alpha_o^2) \,.
\end{equation}
% with $\widehat{\alpha}=\alpha$ in case of $O_N(1)$-band-relaxation time,  as in \eqref{def:fast-rel}.  
Further,  recall Remark \ref{rem:DS21} that the only change between the dynamics of \cite[Prop. 1.6]{DS21} 
and $\bUsp(\cdot)$ is in replacing $\bv(x)$ of \cite[(1.22)]{DS21} by $\bv(x,y)$ of \eqref{def:vt-new}.  
It is thus not hard to see that,  thanks to our assumption that $C(t,0) \to \widehat{\alpha}\alpha_o$, 
the characterization in \cite[Prop.  2.1]{DS21} 
of all plausible \abbr{fdt} solutions without aging,  also hold for $\bUsp(\cdot)$,  apart from 
modifying the first term on the \abbr{rhs} of \cite[(2.4) \& (2.5)]{DS21} as follows:
\begin{align}\label{eq:24-mod}
\b \qs^2 \bv'_\star(\alpha \qs) &\longleftarrow \b \qs^2 \bv_x(\alpha \qs,\widehat{\alpha} \alpha_o) 
+ \b \qo \bv_y (\alpha \qs, \widehat{\alpha} \alpha_o) \,,
 \\
\b \alpha \qs \bv'_\star(\alpha \qs) & \longleftarrow  \b \alpha \qs \bv_x (\alpha \qs, 
\widehat{\alpha} \alpha_o) 
+ \b \widehat{\alpha} \alpha_o  \bv_y(\alpha \qs, \widehat{\alpha} \alpha_o) 
\label{eq:25-mod}
\end{align}
(which reflect the effect of $\bv(x,y)$ on $\widehat{{\sf A}}_q(s)$ of \eqref{def:Aq-new}
and $\widehat{{\sf A}}_C(s,s)$ of \eqref{def:AC-new},  respectively,
as $s \to \infty$). Plugging \eqref{def:AC-new} in \eqref{eqCs-new} at $t=0$,  and taking
% the limit as 
$s \to \infty$,  one appends to the modified
\cite[(2.4) \& (2.5)]{DS21},  a new identity
\begin{align}\label{eq:new-24}
\mu \, \widehat{\alpha} \alpha_o &= \b \qo \bv_x (\alpha \qs, \widehat{\alpha} \alpha_o) 
+ \b \bv_y(\alpha \qs, \widehat{\alpha} \alpha_o) 
- \b^2 \qo \frac{\nu''(\qinf) \nu'(\qinf)}{\nu'(\qs^2)} \kappa_2 
+ \b^2 \widehat{\alpha} \alpha_o \kappa_1 \,,
\end{align}
which determines the value of $\widehat{\alpha}$.  Up to such modification, 
% as in \eqref{eq:24-mod}-\eqref{eq:new-24},  
everything in \cite[Prop. 2.1]{DS21} remains in place here,
with a predicted \abbr{fdt} solution of the form $c_\gamma(\cdot)$ of \eqref{FDTDb-new},  
where upon re-parametrizing  $\gamma = \frac{1}{2} - g_{\b}'(c_\gamma(\infty))$,  we
set $\mu=\phi_\gamma(1)$ of \eqref{def:phi-gamma} on the \abbr{LHS} of \eqref{eq:new-24} 
and \cite[(2.4)]{DS21},  substituting \cite[(2.6)]{DS21} on the \abbr{LHS} of \cite[(2.5)]{DS21},
and setting on the \abbr{rhs} of \eqref{eq:new-24} and \cite[(2.4) \& (2.5)]{DS21} 
the values $\kappa_1=2 [\nu'(1)-\nu'(c_\gamma(\infty))]$,  $\kappa_2=2(1-c_\gamma(\infty))$ 
to arrive at three non-linear algebraic equations for $(\alpha, \widehat{\alpha},c_\gamma(\infty))$
which must also satisfy the inequality \eqref{eq:wh-alpha-range}.

%\begin{remark}\label{rem:aging}  Subject to the preceding modification,  the discussion 
%of aging in \cite[Section 2]{DS21},  starting at \cite[(2.8)]{DS21},  also applies here. 
%\end{remark}

\subsection{Localized states with no-aging}\label{subsec:no-aging}
We proceed to examine the possibility of \emph{localized states with no-aging.}
That is,  having for $\bxo$ drawn from the Gibbs measure of a generic model, 
which is strictly $1$-\abbr{rsb} at $\b_0>\b_c^{\rm stat}$,  
the Langevin diffusion \eqref{diffusion},  potentially with $\b \ne \b_0$, 
relaxes within $O_N(1)$-time onto $\alpha \hbxs$-band for some $\alpha \ne 0$.
In view of Theorem \ref{thm-uncond}(b),  Proposition \ref{prop:ell-lim} and 
Definition \ref{def:rel},  this amounts to $q(\tau) \to \alpha \qs$ and
$C(\tau+t,t) \to \alpha^2$ for the dynamics $\bUsp(\cdot)$ 
with $\alpha_o=\qs>0$ and $(\qs,\V)$ of  \eqref{def:qs-beta}-\eqref{Ep-E} at $\b_0>\b_c^{\rm stat}$. 
In view of \eqref{eq:wh-alpha-range},  this requires $\widehat{\alpha}=\alpha$,  namely having $O_N(1)$-relaxation
time as in Definition \ref{def:rel}.  It is further not hard to see,  as in Remark \ref{rem:DS21}, 
that for $\alpha_o=\qs$ one has $\bv_y(x,x)=0$ and $\bv_x(x,x)=\bv'(x)$ of \cite[(1.22)]{DS21}.  
Having here also $\widehat{\alpha}=\alpha$,  it follows that 
both sides of \eqref{eq:24-mod}-\eqref{eq:25-mod} are equal,  and that \eqref{eq:new-24} 
must match \cite[(2.4)]{DS21}.  Thus,  one merely needs to consider which values of
$(\alpha,\gamma)$ satisfy \cite[(2.13)-(2.15)]{DS21}
in case of initial values given by \eqref{def:qs-beta}--\eqref{Ep-E} at $\b_o$.
%Plugging $\kappa_1 = 2(\nu'(1)-\nu'(c_\gamma(\infty)))$ and 
%$\kappa_2 = 2(1-c_\gamma(\infty))$ in the equations replacing \cite[(2.4) \& (2.5)]{DS21}, 
%using that $\mu_\star=\phi_\gamma(1)$ and the relation 
%\cite[(2.6)]{DS21} between $\IJ$, $\gamma$ and $c_\gamma(\infty)$, we thus arrive at
%\cite[(2.13) \& (2.14)]{DS21}, with $\bv'(x)$ replaced now by $(\bv_x + \bv_y)(x,x)$. 
%These in turn result with \cite[(2.15)]{DS21} holding here as well 
We next show that solutions of this type with $\b \ne \b_0$ exist only for 
pure $p$-spins and band sizes $\alpha^2=q_\b$,  which for $\b>\b^{\rm stat}_c$ coincide with $q_{\rm EA}(\b)$.
Specifically,  fixing the pure $p$-spin model $\nu(x)=x^p$ and $q_c \in (0,1)$ that solves
\cite[(1.9)]{SubagTAPpspin},  consider the solutions $q \in (0,1)$ of 
\begin{equation}\label{eq:TAP112}
y_p(\b,q) := 2 \b \sqrt{\nu''(q)} (1-q) = \sqrt{(p-1)(1-q_c)} \,.
\end{equation}
It is easy to check that $q_c > \frac{p-2}{p-1}$,  so the \abbr{rhs} of \eqref{eq:TAP112} is in $(0,1)$,  
whereas from \cite[(1.24)]{DGM} we know that for $\b > \b_c^{\rm dyn}$ the supremum 
%over $q \in (0,1)$
of its \abbr{lhs} exceeds one.  For pure $p$-spins $q \mapsto y_p(\b,q)$ is increasing on $[0,\frac{p-2}{p})$ and 
decreasing on $(\frac{p-2}{p},1]$,  so \eqref{eq:TAP112} has two solutions
\begin{equation}\label{dfn:qb-minus}
0<q^-_\b < \frac{p-2}{p} < q_\b < 1,
\end{equation}
where from \cite[(1.10)-(1.12)]{SubagTAPpspin} we 
deduce\footnote{For $\b>\b^{\rm stat}_c$ it is shown in \cite{SubagTAPpspin} that $q_\b$ is the maximal multi-samplable overlap,  
but instead of invoking \cite[Corollary 6]{SubagFlandscape} 
%(about the maximal multi-samplable overlap), 
one can appeal to \cite[Corollary 12]{SubagFlandscape} for the same fact about $q_{\rm EA}(\b)$,  thereby concluding that 
also $q_\b=q_{\rm EA}(\b)$. 
}  
that $q_\b=q_c$ at $\b_c^{\rm stat}$ and
$q_\b = q_{\rm EA}(\b)$ for $\b > \b^{\rm stat}_c$.
More generally,  $q_\b$ 
% for $\b \in (\b_c^{\rm dyn},\b_c^{\rm stat}]$
maximizes over $q \in (0,1]$
the free energy 
of \cite[(1.13)]{SubagTAPpspin},  which corresponds to a \abbr{rs} model on the band
${\Bbb S}_{\hbxs} (q)$, 
%of \eqref{eq:subsphere},  
conditional on having a critical point $\nabla_{\rm sp} H_{\BJ}(\hbxs) = {\bf 0}$
of ground-state energy $H_{\BJ}(\hbxs)=\GS(1)$.
%  (see also Remark \ref{rem:res-model}).
\begin{prop}\label{prop:fast-relax-band}
Fix $\b > \b_c^{\rm dyn}$ and $(\qs,\V)$ given by
\eqref{def:qs-beta}-\eqref{Ep-E} at some $\b_0>\b_c^{\rm stat}$. 
Consider the predicted \abbr{fdt} solution for $\bUsp(\cdot)$
of the form $c_\gamma(\cdot)$,  with $q(s) \to \alpha \qs$ and
$c_\gamma(\infty)=\alpha^2 \ne 0$, that satisfies 
\cite[(2.13)-(2.15)]{DS21}.  For pure $p$-spins there exist such a solution
(with $\gamma=\gamma_\b$ of \eqref{def:gamma-star2}),
%or alternatively
%\begin{equation}\label{eq:gamma-star}
%\gamma_\b =\frac{1}{2} + \frac{q_\b}{2(1-q_\b)^2} (q_c-q_\b) \,.
%\end{equation}
if and only if $\alpha^2 = q_\b$ and $\b < \b_c^{\rm dyn}(\nu_{q_\b})$
for $\nu_q$ of \eqref{def:nu-q}.\footnote{also when
$\b=\b_c^{\rm dyn}(\nu_{q_\b})$ but not having there a first-order dynamic phase transition (see
\eqref{eq:betac-gap}).} 
The corresponding limiting macroscopic energy is then 
\begin{equation}\label{eqH-lim}
H(\infty) := \lim_{s \to \infty} H(s)= \GS(\sqrt{q_\b}) + 2\b \theta(q_\b) \,,
\end{equation}
which for $\b > \b_c^{\rm stat}$ matches the macroscopic energy $\Ep(\b)$
of $\widetilde{\mu}_{2\b,\BJ}^N$
(see \eqref{Ep-E}). \\
In contrast,  for any other model such a solution requires 
that $\b=\b_0$,  $\alpha=\qs$ and $H(\infty)=\Ep(\b_0)$. 
\end{prop}
\begin{remark}\label{rem:loc-state} For $\b \ne \b_0$ a localized state without aging requires having
pure $p$-spins and limiting band size $q_\b$ (which for $\b>\b_c^{\rm stat}$ matches the 
macroscopic band sizes of our ``target'' Gibbs measures $\widetilde{\mu}_{2\b,\BJ}^N$).  Indeed,  
thanks to the homogeneity of $H_{\BJ}(\cdot)$ for pure $p$-spins,  such target Gibbs bands 
then perfectly align with the bands of $\widetilde{\mu}_{2\b_0,\BJ}^N$,  making it plausible for 
\eqref{diffusion} to relax,  within $O_N(1)$-time,  to a random position on the concentric target band, 
which within that band is un-correlated with $\bx_0$.  
But,  as in Corollary \ref{cor:relax},  such fast relaxation can occur only if $\b$ is below 
the dynamical critical parameter of the effective model on the target band.  While the value of $\b_0$
has no effect on the localized states,  namely on its parameters $q_\b$ and $\gamma_\b$,  
Proposition \ref{prop:fast-relax-band} only states necessary conditions for having such a state. 
Specifically,  for $\b_0$ which is too far from $\b$,  starting at $q(0)=q_{\b_0}$,  the function 
$q(s)$ may stop short of reaching our desired value of $\sqrt{q_{\b_0} q_\b}$.  In any case,  
there is no such alignment for other mixed models,  with the centering of the target band 
inherently carrying macroscopic information about the initial state $\bx_0$. 
\end{remark}

\begin{proof} As shown in the proof of Proposition \ref{prop-stat},  except for pure $p$-spins,  
the \abbr{rhs} of \cite[(2.15)]{DS21} requires having $\alpha=\pm \qs$,  with $\alpha=-\qs$ possible 
only when $\nu(\cdot)$ is either even or odd function.  Substituting first $\alpha=\qs$ and recalling Remark 
\ref{rem:w-unique} that $\bv_x(\qs^2,\qs^2)=\Gg$ and $\bv_y(\qs^2,\qs^2)=0$,  one concludes after some 
algebra that \cite[(2.13)-(2.15)]{DS21} hold if and only if 
$\gamma=\gamma_\star(\qs,\qs)$ of \eqref{def:gamma-star} and 
$\Gg=\Gg(\b)$ of \eqref{G-stat} for the specified $\qs^2=q_{\rm EA}(\b_0)$.  
The same applies for $\alpha=-\qs$ and $\nu(\cdot)$ even,  with $\Gg=-\Gg(\b)$ in case $\alpha=-\qs$ 
and $\nu(\cdot)$ is odd.  Recall \eqref{G-stat} that here $\Gg=\Gg(\b_0)$ and 
while $\Gg(x)=-\Gg(\b_0)$ has no positive solution, 
the quadratic equation $G_\star(\b)=G_\star(\b_0)$ has two positive solutions:  
\begin{equation}\label{eq:Plefka}
\b=\b_0\,,  \qquad \hbox{ or } \qquad 
\frac{1}{\sqrt{\b \b_0}} = 2 \sqrt{\nu''(q_{\rm EA}(\b_0))} (1-q_{\rm EA}(\b_0))  \,.
\end{equation}
In view of \eqref{eq:c-inf} and \eqref{def:gamma-star},  our assumption that $c_{\gamma_\star}(\infty)=\qs^2$
implies that $g''_{\b}(\qs^2) \le 0$,  which in turn results with $\b \le \b_0$ being the smaller of the two 
values in \eqref{eq:Plefka},  or equivalently,  that the \abbr{rhs} of \eqref{eq:Plefka} is at most $1/\b$.
Recall that Plefka's condition must hold at $q_{\rm EA}(\b_0)$,  namely that the \abbr{rhs} of \eqref{eq:Plefka}
must also be at most $1/\b_0$ (see \cite[Thm 12(3)]{SubagFlandscape}).  Necessarily 
$\b=\b_0$,  in which case by Proposition \ref{prop-stat},  for $\nu(\cdot)$ even and $\V$ given by 
\eqref{def:Es}-\eqref{Ep-E} at $\b=\b_0$,  the functions $C(s,t)=c_{\gamma_\star}(s-t)$ and $q(s) \equiv \pm  \qs^2 = \pm  q_{\rm EA}(\b_0)$,  are both
stationary dynamics of the form $\bUsp(\cdot)$.  Thus,  our assumption that $\qo=q_{\rm EA}(\b_0)$
rules out having $\alpha=-\qs$. 

Turning next to the pure $p$-spins, 
%recall that then $w_2=0$ and $G=m E/\qs^2$ (see Remark \ref{rem:w-unique}).  Having here $\qo=\qs^2$, one easily deduces that in this case $(\bv_x+\bv_y)(x,x)=\qs^2 (w_1+w_3) \nu'(x)
%=E \nu'(x)/\nu(\qs^2) 
%= G \nu'(x)/\nu'(\qs^2)$.
%In view of \cite[Footnote 4]{DS21}, the latter value matches the value which $\bv'_\star(x)$ of
%\cite{DS21}  takes for pure $m$-spins. Having thus recovered for pure $m$-spins exactly 
%the same equations as \cite[(2.13)-(2.15)]{DS21}, 
we conclude from \cite[(2.16)]{DS21} that $\alpha \in (0,1)$ must be such that 
\begin{equation*}
% \label{eq:pure-no-aging}
\Gg (\b_0) = \sqrt{\nu''(\qs^2)} ( y + y^{-1} )  
\qquad {\rm for } \qquad
y = y_p(\b,\alpha^2) 
% := 2 \b \sqrt{\nu''(\alpha^2)} (1-\alpha^2) 
\qquad {\rm of } \qquad \eqref{eq:TAP112} 
\end{equation*}
(except for $p$ even, where a-priori a negative solution $-\alpha$ is also possible).
As before, having $c_\gamma(\infty)=\alpha^2$ requires also that $g_\b''(\alpha^2) \le 0$,
or equivalently, to choose in the preceding the smaller solution $y \le 1$.
%  of \eqref{eq:pure-no-aging}. 
Recall \eqref{G-stat}, that $G_\star(\b_0)$ is of such a form 
for $y=y_p(\b_0,q_{\rm EA}(\b_0)) \in (0,1)$.
We thus need $|\alpha|<1$ such that $q=\alpha^2$ satisfies \eqref{eq:TAP112}.  \red{Consequently,  see
\eqref{dfn:qb-minus},}  either $\alpha^2=q^-_\b$ or $\alpha^2=q_\b$.
Recall that $g_\b'(0)=0$ with $g_\b'(\cdot)$ non-increasing on $[0,q^-_\b]$,  while
from \eqref{def:bc-dyn} we know that $g_\b'(x_\star)>0$ for some $x_\star(\b) \in (0,1)$.
In particular, $x_\star>q^-_\b$,  so in view of \eqref{eq:c-inf} 
only $q_\b$ can be the value of $c_{\gamma}(\infty)$.
This further requires having $\gamma=\frac{1}{2}-g_\b'(q_\b)=\gamma_\b$ of
\eqref{def:gamma-star2} and that $g'_\b(x)<g'_\b(q_\b)$ for any $x \in (q_\b,1]$.
While proving Corollary \ref{cor:relax} we saw that the latter condition amounts to 
$\b \le \b_c^{\rm dyn}(\nu_{q_\b})$ (and the case of equality is then resolved as in \eqref{eq:betac-gap}).

Next, having an \abbr{fdt} solution as in \cite[Prop. 2.1]{DS21} except for 
$\gamma_\b \ne \frac{1}{2}$, and with the no-aging condition 
$c_{\gamma_\b}(\infty)=\alpha^2$, yields the convergence of 
$H(s)$ as $s \to \infty$, to the value in \cite[(2.20)]{DS21}.
%apart from replacing $\bv_\star(x)$ there, by $\bv(x,x)$, for $\bv(\cdot,\cdot)$ of \eqref{def:vt-new}.
Moreover,  for $\qo=\qs^2$ and pure $p$-spins,  it follows from \eqref{def:vt-new}-\eqref{def:w-s} that 
\[
\bv (\alpha \qs,\alpha \qs) =\Es (\b_0) \frac{\alpha^p}{\qs^p} = \GS (\qs) \frac{\alpha^p}{\qs^p} = \GS (\alpha) 
\]
(where we also utilized \eqref{def:Es},  the homogeneity of $H_{\BJ}(\cdot)$ and \eqref{def:GS-q}).  The
value \eqref{eqH-lim} for $H(\infty)$ follows.
\end{proof}

\section{Existence,  confinement and continuity}\label{sec:ex} 

Hereafter,  for any finite $N,\bar{N}$ and $\psi : \B^N_\star \to \R^{\bar{N}}$ we denote by
\begin{align}\label{def:norm-N}
\| \psi \|_\infty := \sup \{ \| \psi (\bx)\| :   \bx \in \B^N_\star  \big\} \,,  \qquad
%\\ \label{def:norm-L}
\| \psi \|_{\rm L} := \sup \Big\{ \frac{\|\psi(\by) - \psi(\bx) \|}{\|\by-\bx\|} :  \by \ne \bx \in \B^N_\star \Big\} \,,
\end{align}
the supremum and Lipschitz norms \abbr{wrt} the Euclidean norm.  For $\Psi(\bx)$ from
$\B^N_\star$ to the set of $N'$-dimensional matrices,  we define $\|\Psi\|_\infty$
%and $\|\Psi\|_{\rm L}$ 
via the spectral norms
$\|\Psi(\bx)\|_2:=\sup\{ \|\Psi(\bx) \bu\| : \|\bu\|=1,  \bu \in \R^{N'} \}$ 
%and $\|\Psi(\by)-\Psi(\bx)\|_2$ 
at $\bx
%,\by 
\in \B^N_{\star}$,  as in \eqref{def:norm-N},  and recall for $\bar{N}=1$ and 
$\psi \in C_0^2(\B^N_\star)$ the control on various norms 
\begin{equation}\label{eq:L-op}
\frac{1}{r_\star^2 N} \|\psi\|_\infty \le \frac{1}{r_\star \sqrt{N}} \|\psi\|_{\rm L} \le \frac{1}{r_\star \sqrt{N}} \|\nabla \psi \|_\infty \le
\|\nabla \psi \|_{\rm L} \le \|{\rm Hess}(\psi)\|_\infty 
% by the triangle inequality and considering the $\|\cdot\|_2$ of
%\[
%\psi(\by)-\psi(\bx)=\int_0^1 \big\langle \nabla \psi\big( t \by+(1-t) \bx\big), \by-\bx \big\rangle dt \,.
%\]
\end{equation}
via the supremum norm of the $N$-dimensional Hessian matrix of $\psi$.  Equipped with these norms
we prove the existence of a strong solution of \eqref{diffusion-phi}.  We 
also show that when $\frac{1}{\sqrt{N}} \|\nabla \varphi_N \|_\infty$ are uniformly bounded,  the solution stays away 
from the blow-up of $f'(\cdot)$ at the boundary of $\B^N_\star$,  
up to an $e^{-\Omega(N)}$-small probability,  and for $f_\ell(\cdot)$ of \eqref{eq:fdef},
these solutions are further confined,  as $\ell \uparrow \infty$,  to shrinking annuli around $\SN$.
\begin{prop} \label{prop:existence} 
Fix $f=f_\ell$ as in \eqref{eq:fdef} with $f_0(\cdot)$ of  locally Lipschitz 
derivative on $[0,r_\star^2)$.
\newline
(a).  For any $N,  \b <\infty$,  $\bar r > r_\star$,  $\varphi_N \in C_0^2(\B^N)$ and $\bxo \in \B^N(\sqrt{N})$, 
the dynamics \eqref{diffusion-phi} admits a unique strong solution $\{\bx_t,  t \ge 0 \} \subset \B^N_\star$.
\newline
(b).  If $\| \nabla \varphi_N\|_\infty \le 3 \kappa r_\star \sqrt{N}$ for $\kappa$ finite and all $N$,  then 
for $\tau_r := \inf\{ t \ge 0 :N^{-1/2} \| \bx_t \| \ge r \}$ and some $r_1 \in (1,r_\star)$ that depends only on 
$\kappa,r_\star$, $f'(\cdot)$,
\begin{equation}\label{eq:exp-conf-phi}
\limsup_{N \to \infty} \frac{1}{N} \log \P\big( \tau_{r_1} \le T ) < 0 \,,  \qquad \forall T < \infty\,.
\end{equation}
(c).  For $\bxo \in \SN$,  we further have \eqref{eq:exp-conf-phi} holding for
$\tau_c := \inf\{ t \ge 0 :  N^{-1/2} \| \bx_t \| \notin I_\ell \}$,  some $c<\infty$ and all 
large enough $\ell$,  where $I_\ell := (1-\frac{c}{\ell},1+\frac{c}{\ell} )$
is adjusted to $f_\ell(\cdot)$ of \eqref{eq:fdef}.
\end{prop}

\begin{proof} (a).  As in \cite[proof of Prop. 2.1]{BDG2},  introducing bounded 
globally Lipschitz functions $\phi_k(\cdot)$ on $\R^N$ 
such that $\phi_k(\bx)=\bx$ on $\B^N(r_k \sqrt{N})$ for some $r_k \uparrow r_\star$, 
yields unique strong solutions \purple{$\bx^{\lak}_t$} of \eqref{diffusion-phi} up to the non-decreasing exit times
\begin{equation}\label{def:tau-k}
\tau_k := \inf\{ t \ge 0 : \|\bx^{\lak}_t \| \ge r_k \sqrt{N} \}\,. 
\end{equation}
Setting $\rho_N := 1 + \frac{\b}{\sqrt{N}} \| \nabla \varphi_N\|_\infty$,  \red{we next show that 
for some $\delta(r_k,b,T) \downarrow 0$ as $k \uparrow \infty$,
\begin{equation}\label{eq:tau-k-bd}
\P(\tau_k \le T) \le \delta(r_k,\rho_N,T),   \qquad \forall N,T < \infty 
\end{equation}
from which the} a.s.  unique strong solution of \eqref{diffusion-phi} \red{on $[0,\infty)$ follows.  Specifically, 
apply} Ito's formula for 
\begin{equation}\label{def:KN-uk}
\purple{Z_N (t)} := \frac{1}{\sqrt{N}} \|\bx^{\lak}_{t \wedge \tau_k}\|
\end{equation}
\red{and recall} that $\nabla \|\bx\| = \frac{1}{\|\bx\|} \bx$ and $\Delta \|\bx\| = \frac{N-1}{\|\bx\|}$,  \red{to get that}
\begin{equation}\label{eq:KN-diff}
Z_N(s) = Z_N(0) + \int_0^{s \wedge \tau_k} 
\Big[-\frac{\b}{\sqrt{N}} \Big\langle \nabla \varphi_N(\bx^{\lak}_t),\frac{\bx^{\lak}_t}{\|\bx^{\lak}_t\|} \Big\rangle 
-f'\big(Z_N(t)^2\big) Z_N(t) + \frac{(N-1)}{2 N Z_N(t)}  \Big] dt + \frac{1}{\sqrt{N}} W_{s \wedge \tau_k} \,,
\end{equation}
for the stopped standard Brownian motion 
\[
W_{s \wedge \tau_k} :=  \int_0^{s \wedge \tau_k} \frac{1}{\|\bx^{\lak}_t\|} \langle \bx_t^{\lak} d\BB_t \rangle \,.
\]
%(as in \eqref{def:norm-N}), 
\purple{Next,  consider for $N \ge 1$ and $b \in \R$, } the following \emph{reflected at $\frac{1}{2}$} diffusions on $(\frac{1}{2},r_\star)$,
\begin{equation}\label{eq:KN-bd}
Z_{N,b} (s) = 1 + \int_0^{s} g(Z_{N,b} (t);b) dt + \frac{1}{\sqrt{N}} W_{s} \,,  \qquad 
g(r;b) := b -f'(r^2) r  \,.
\end{equation}
\red{Comparing \eqref{eq:KN-diff} and \eqref{eq:KN-bd},  we see that} $Z_N(s) \le \purple{Z_{N,\rho_N} (s)}$ up to 
\red{the first hitting time}
\begin{equation}\label{def:tau-bar-k}
\hat{\tau}_k := \inf\{t \ge 0: Z_{N,\rho_N}(t) \ge r_k \} \,.
\end{equation}
In particular,  $\hat{\tau}_k \le \tau_k$,  \red{so \eqref{eq:tau-k-bd} follows upon showing} that a.s. 
$\hat{\tau}_k  \uparrow \infty$.  Namely,  that the boundary point $r_\star$ is inaccessible for 
the reflected diffusion $Z_{N,\rho_N}$.  Now,  by \eqref{eq:fdef} we have for some  
$r_o=r_o(b) \in (1, r_\star)$ and $\hat \kappa(r_o) <\infty$,   that for any finite $b$,
\begin{equation}\label{eq:fp-bd}
g(r;b) 
%= b - f'(r^2) r 
\le  - \frac{1}{r_\star - r} 
\quad \hbox{for} \quad r > r_o \quad 
\hbox{and otherwise} \quad
b - \hat \kappa \le g(r;b) \,.
\end{equation}
Thus,  for $t \ge r_o(\rho_N)$,  
\[
v(t) := -2 N \int_{r_o}^t g(x;\rho_N) dx \ge 2N \log[(r_\star-r_o)/(r_\star-t)]
\]
and so the scale function $p(r):=\int_{r_o}^r e^{v(t)} dt$ of $Z_{N,\rho_N}(\cdot)$ of \eqref{eq:KN-bd} is infinite at $r=r_\star^-$.
Further,  with $v(t)$ bounded above uniformly over $\frac{1}{2} \le t \le r_o(\rho_N)$,  clearly $p(\frac{1}{2}^+) > -\infty$.  
By the general 
theory of boundary behavior in \purple{one-dimensional} diffusions (see,  e.g. \cite[Prop.  5.5.22(b)]{KS}), 
\red{almost surely} the \abbr{iid} excursions of $Z_{N,\rho_N}(\cdot)$ away from $\frac{1}{2}$,  which are of positive
expected time,  stay away from $r_\star$,  with $r_\star$ thus inaccessible for the reflected diffusion. 

\medskip
\noindent
(b).  Here $\rho_N \le 1 + 3 \b \kappa r_\star =: b$,  so as reasoned in part (a),  it suffices for 
\eqref{eq:exp-conf-phi}
to show that $\P(\hat{\tau}_{r_1} \le T) \le C e^{-N \delta}$ for 
$\hat{\tau}_r := \inf\{t \ge 0: \purple{Z_{N,b}}(t) \ge r\}$ and some $\delta>0$,  $1<r_1<r_\star$.
In view of \eqref{eq:fdef},  there exists $1<r_o<r_\star$ such that 
$g(x;b) \le 0$ for 
all $x \ge r_o$.  Clearly $\P(\hat{\tau}_r \le T)$ is bounded above,  for any $r > r_o$,  
by moving to $r_o$ both \purple{$Z_{N,b}(0)$} and the reflection point $\frac{1}{2}$,  while also replacing the
drift $g(x;b) \le 0$ for $x \ge r_o$,  by a zero drift.  Upon doing so,  by the reflection principle,  we have for some
$C<\infty$ and any $r_1 > r_o$,  
\begin{equation}\label{eq:sup-bd}
\P(\hat{\tau}_{r_1} \le T) \le \P(\sup_{s \le T} \{ |W_s| \} \ge \sqrt{N} (r_1-r_o) ) \le C e^{-\frac{N(r_1-r_o)^2}{2T}} \,,
\end{equation}
as needed for completing the proof of \eqref{eq:exp-conf-phi}. 

\noindent
(c).  Fixing $b$ and $r_1$ as in part (b),  set $c = 2(b +b')$ where
$b':=\sup \{ x |f_0'(x^2)|  : x \in [0,r_1]\}<\infty$.  Note that 
$I_\ell \subseteq (\frac{1}{2},r_1)$ for all
$\ell \ge \ell_\star
%:= \frac{c}{(r_1-1) \wedge \frac{1}{2}}
$,  whereupon $\{\rho_N \le b \}$ and $\|\bxo\|=\sqrt{N}$ imply that 
%$N^{-1/2} \|\bx_t \|_2 \in I_\ell$ up to 
$\tau_c \ge \hat{\tau}^{(+)}_{1 + \frac{c}{\ell}}  \wedge \hat{\tau}^{(-)}_{1-\frac{c}{\ell}}$ for the stopping times 
$\purple{\hat{\tau}}^{(\pm)}_r := \inf\{t \ge 0: \purple{Z_{N,\pm b}}(t) = r \}$,  \purple{with  
$Z_{N,b}(\cdot)$ of \eqref{eq:KN-bd} no-longer reflected at $\frac{1}{2}$}.
Further,  for $\ell \ge \ell_\star$ and $f(\cdot)=f_\ell(\cdot)$ of \eqref{eq:fdef},
\begin{align*}
g(x;+b)= & +b - x f'_0(x^2) - 2 \ell (x-1) (x+1) x \le 0\,, \qquad \forall x \in [1+\frac{c}{2\ell},1+\frac{c}{\ell}] \,,\\
g(x;-b)= & -b - x f'_0 (x^2) + 2 \ell (1-x) (x+1) x \ge 0 \,, \qquad \forall x \in [1-\frac{c}{\ell},1-\frac{c}{2\ell}] \,. 
\end{align*}
We thus have as in \eqref{eq:sup-bd},  that for such $c$ and $\ell$, 
\[
\P(\hat{\tau}^{(\pm)}_{1\pm \frac{c}{\ell}} \le T) \le \P(\sup_{s \le T} \{W_s\} \ge \frac{c}{2\ell} \sqrt{N} ) \le C 
e^{-\frac{N c^2}{8 T \ell^2}} \,, 
\]
hence \eqref{eq:exp-conf-phi} holds for $\tau_{c}$
% = \inf\{ t \ge 0 : N^{-1/2} \|\bx_t \|_2 \notin I_\ell\} 
and any $\ell \ge \ell_\star$. 
\end{proof}

Utilizing the preceding proof,  we establish the continuity of 
$\varphi_N \mapsto \widehat\err_{N,T}(\bxo,\bxs,\varphi_N;\bU_\infty)$, 
with Lemma \ref{lem:cont-varphi} as an immediate consequence.
\begin{prop}\label{prop:cont}
For any $N,T,\b<\infty$,   $\bxo,\bxs \in \B^{N} (\sqrt{N})$,  $\varphi^{(i)}_N \in C^2_0(\B^N)$ and $f(\cdot)$ as in \eqref{eq:fdef},
\begin{equation}\label{eq:cont-phi}
\sup_{\bU_\infty} 
\big|\widehat\err_{N,T}(\bxo,\bxs,\varphi^{(2)}_N;\bU_\infty)
- \widehat\err_{N,T}(\bxo,\bxs,\varphi^{(1)}_N;\bU_\infty)\big|
\le \gamma_{T,\b}\big(\|\nabla \varphi^{(1)}_N \|_{\rm L},\frac{1}{\sqrt{N}} \|\nabla (\varphi^{(2)}_N - \varphi_N^{(1)}) 
\|_{\infty}\big)\,,
\end{equation}
for some $\gamma_{T,\b}(\kappa,\delta) \to 0$ as $\delta \to 0$,  for any fixed $\kappa<\infty$.
\end{prop}
\begin{proof}[Proof of Lemma \ref{lem:cont-varphi}] Fix $T,\b,N,\red{d}<\infty$,  $\bxo, \bxs \in \B^N(\sqrt{N})$,
$\varphi_N,\psi^{(i)}_N \in C_0^2(\B^N)$
%$i \le d_N$
and non-random $\bU_\infty$.  Let
$\varphi_N^{(1)}=\varphi_N$ and $\varphi_N^{(2)}=\varphi_N + \sum_i \red{\eta_n^{(i)}} \psi^{(i)}_N$.  Note that 
$\| \nabla \varphi_N^{(1)} \|_{\rm L} < \infty$ and \red{for sequences $\eta^{(i)}_n \to 0$,  at each $i \le d$,} 
\[
\|\nabla (\varphi^{(2)}_N - \varphi_N^{(1)})\|_{\infty} \le \Big( \red{\sum_{i=1}^d |\eta_n^{(i)}|} \Big) \red{\max_{i \le d}} \big\{ \| \nabla \psi_N^{(i)} \|_\infty \big\} 
\to 0 \,,
\]
with \eqref{eq:varphi-cont} thus an immediate consequence of \eqref{eq:cont-phi}.
\end{proof}
\begin{proof}[Proof of Proposition \ref{prop:cont}]
Fixing $\varphi^{(i)}_N \in C^2_0(\B^N)$ construct the strong solutions $\Bx^{(i)}_t: [0,T] \mapsto \B^N_\star$ 
of \eqref{diffusion-phi},  per Proposition~\ref{prop:existence}(a),
for $f$ as in \eqref{eq:fdef} and potentials $\varphi_N^{(i)}$ respectively, 
out of the same Brownian path $t \mapsto \BB_t$ \red{and starting point $\bxo$}.   Let
$\|\cdot\|_T$ denote the sup-norm on $[0,T]^2$ and fixing $r_k \uparrow r_\star$
consider the stopping times $\tau_k^{(i)}$ of \eqref{def:tau-k} \purple{for the $\Bx^{\lak}_t$ 
approximations} of $\Bx^{(i)}_t$.  It then 
follows from \eqref{eq:err},  \eqref{eq:err-hat} and the triangle inequality that 
\begin{align}
| \widehat\err_{N,T}(\bxo,\bxs,\varphi^{(2)}_N;\bU_\infty)-
 & \widehat\err_{N,T}(\bxo,\bxs,  \varphi^{(1)}_N;\bU_\infty)|  \le 4 \,  \P(\Aa_k^c)  
  + \E_k \big[\| \phiN^{(2)}-  \phiN^{(1)}\|_T \big] \nonumber \\
  &+\E_k \big[\|q^{(2)}_N-q_N^{(1)}\|_T  \big]
 + \E_k \big[ \|C^{(2)}_N-C^{(1)}_N\|_T \big] 
 +\E_k \big[ \|\chi^{(2)}_N-\chi_N^{(1)}\|_T \big] \,,
 \label{eq:triangle-ineq}
\end{align}
 where $(C_N^{(i)},\chi_N^{(i)},q_N^{(i)},  \phiN^{(i)})$ are 
evaluated at the induced solutions $\Bx_t^{(i)}$ and $\E_k[\cdot]$ denotes the expectation 
% over the Brownian motion,  
restricted to the event $\Aa_k := \{\tau_k^{(1)} \wedge \tau_k^{(2)} > T\}$.  \red{Next,  note that for any $N,T$,
\[
c_{N,T}^2 := \E \Big[\sup_{t \le T} \frac{1}{N} \|\BB_t\|^2\,\Big] \le c_{1,T}^2 < \infty 
\]
and} recall that $\sup_{t} \{ \| \Bx_t^{(i)}\| \}  \le r_\star \sqrt{N}$ \red{while
$\| \bxs\| \le \sqrt{N}$.  Hence,} setting 
\[
\purple{\Gamma_{N,k}}  :=  \E \big[ \| D_{N,k} \|_T^2 \big]^{1/2},  \qquad 
\purple{D_{N,k}}(t) := \frac{1}{\sqrt{N}} \indic_{\Aa_k} \,  \|\bx_t^{(2)}-\bx_t^{(1)}\|\,, 
\]
\red{we} deduce by Cauchy-Schwarz that 
\[
\E_k \big[\|q^{(2)}_N-q_N^{(1)}\|_T \big] \le \Gamma_{N,k} \,,\quad
 \E_k \big[ \|C^{(2)}_N-C^{(1)}_N\|_T \big] \le 2 r_\star \Gamma_{N,k}  \,,\quad
\E_k  \big[ \|\chi^{(2)}_N-\chi_N^{(1)}\|_T \big] \le c_{1,T} \Gamma_{N,k} \,.
\]
In view of \eqref{eq:Hn-gen},  by the triangle inequality we have that for any $s \ge 0$, 
\[
| \phiN^{(2)} (s) -  \phiN^{(1)}(s) | \le \frac{1}{N} \big[ \|\varphi_N^{(1)}\|_{\rm L}
\|\bx_s^{(2)} - \bx_s^{(1)}\| + \|\varphi_N^{(2)} - \varphi_N^{(1)} \|_\infty \big] \,.
\]
Thus,  by Cauchy-Schwarz and \eqref{eq:L-op},  
\begin{equation}\label{eq:varphi-diff}
\E_k \big[\| \phiN^{(2)} -  \phiN^{(1)}\|_T  
 \big] \le r_\star \big[ \|\nabla \varphi^{(1)}_N \|_{\rm L} \Gamma_{N,k} +
 \frac{1}{\sqrt{N}} \|\nabla ( \varphi^{(2)}_N - \varphi^{(1)}_N ) \|_{\infty} \big] \,.
\end{equation}
By the union bound and \red{ \eqref{eq:tau-k-bd}} we deduce that
% the first hitting times $\tau^{(i)}_k$ of $r_k \uparrow r_\star$ are such that 
$\P(\Aa^c_k) \to 0$ as $k \to \infty$,  
uniformly over $\bxo \in \B^N(\sqrt{N})$,  at a rate which in view of \eqref{eq:L-op} depends only on 
\[
\frac{\b}{\sqrt{N}} \max_{i=1,2} \{ \|\nabla \varphi^{(i)}_N\|_\infty \} \le 
\b \big[ r_\star \| \nabla \varphi^{(1)}_N\|_{\rm L} + \frac{1}{\sqrt{N}} \|\nabla (\varphi^{(2)}_N - \varphi_N^{(1)} )\|_\infty \big] \,.
\]
Plugging the preceding bounds in \eqref{eq:triangle-ineq},  it thus suffices bound \purple{$\Gamma_{N,k}$}, 
in the form of the \abbr{rhs} of \eqref{eq:cont-phi},  per fixed $k$.  To this end,  note that up to time 
$\tau_k^{(1)} \wedge \tau_k^{(2)}$ one has that
\purple{each $\Bx^{(i)}_t$ coincides with the corresponding} $\bx^{\lak}_{t \wedge \tau^{(i)}_k} \in \B^{N}(r_k \sqrt{N})$.  Thus,  setting 
$\purple{Z^{(i)}_N(t)}$,  \red{$i=1,2$,  analogously to \eqref{def:KN-uk},}
yields for $t \le \tau_k^{(1)} \wedge \tau_k^{(2)}$,  the identity
\[
\Bx_t^{(2)}-\Bx_t^{(1)} = \int_0^t \Big[ 
\b \nabla \varphi^{(1)}_N (\bx^{(1)}_s)  -\b \nabla \varphi^{(2)}_N (\bx^{(2)}_s)
+ f'(Z_N^{(1)}(s)^2) \Bx_s^{(1)} - f'(Z^{(2)}_N(s)^2) \Bx_s^{(2)}  \Big] ds \,.
\]
Considering the Euclidean norm of both sides 
multiplied by $\frac{1}{\sqrt{N}} \indic_{\Aa_{k}}$,
it follows by the triangle inequality that
\begin{align}
D_{N,k}(t) &\le 
\frac{\b}{\sqrt{N}} \indic_{\Aa_{k}} 
\int_0^t \|  \nabla \varphi^{(2)}_N (\bx^{(2)}_s)  - \nabla \varphi^{(1)}_N (\bx^{(1)}_s)\| ds \nonumber \\
&
+2 r_\star^2 \|f'\|^{(r_k)}_{\rm L}  \indic_{\Aa_{k}} 
 \int_0^t | Z_N^{(2)} (s) - Z^{(1)}_N(s) | ds + \|f'\|_\infty^{(r_k)} \int_0^t D_{N,k} (s) ds
 \,,
 \label{eq:Delta-bd}
\end{align}
where $\|f'\|^{(r)}_\infty$ and $\|f'\|^{(r)}_{\rm L}$ denote the finite supremum and Lipschitz norms
of $f'(\cdot)$ on $[0,r^2]$.  Similarly to \eqref{eq:varphi-diff},  we also have that
\[
\frac{1}{\sqrt{N}}\,
\indic_{\Aa_{k}} \|  \nabla \varphi^{(2)}_N (\bx^{(2)}_s)  - \nabla \varphi^{(1)}_N (\bx^{(1)}_s)\|
% \| \nabla \big (\phiN^{(2)}  - \phiN^{(1)} \big) \| (s) 
 \le \| \nabla \varphi^{(1)}_N \|_{\rm L} D_{N,k}(s) + \frac{1}{\sqrt{N}}  
\| \nabla (\varphi^{(2)}_N -\varphi_N^{(1)} ) \|_\infty \,.
\]
%(where $\|\nabla \psi_N\|_\infty$,  $\|\nabla \psi_N \|_{\rm L}$ denote the supremum and Lipschitz norms
%of $\nabla \psi_N : \B^N (r_\star \sqrt{N}) \mapsto \R^N$ \abbr{wrt} the Euclidean norm on $\R^N$,  as in \eqref{def:norm-N},
%with $\underline{\psi}_N$ to indicate taking also the supremum over $i \le d_N$).
Moreover,  by the triangle inequality,  for any $s \le T$,
\begin{equation}\label{eq:KN-DN-diff}
\indic_{\Aa_k}  | Z_N^{(2)}(s) - Z^{(1)}_N(s) | \le D_{N,k} (s) \,.
\end{equation}
Substituting these bounds in \eqref{eq:Delta-bd} yields the differential inequality
\begin{equation}\label{eq:Delta-pre-Gronwal}
D_{N,k}(t) \le C'_{N} t  + C^{(r_k)}_{N} \int_0^t D_{N,k}(s) ds,  \qquad D_{N,k}(0)=0 
\end{equation}
where $C'_{N}:= \frac{\beta}{\sqrt{N}} \| \nabla (\varphi^{(2)}_N -\varphi_N^{(1)} ) \|_\infty$ and 
$C_{N}^{(r)}:=\b \|\nabla \varphi^{(1)}_N \|_{\rm L}  + 2 r_\star^2 \|f'\|^{(r)}_{\rm L}+  \|f'\|_\infty^{(r)}$.  
Applying Gronwall's inequality,  we deduce from \eqref{eq:Delta-pre-Gronwal} that for any $k$ and $N, T,\b$
\begin{equation}\label{eq:DN-bd}
\| D_{N,k}\|_T \le C'_N T \exp(C^{(r_k)}_{N} T) \,.
\end{equation}
The non-random \abbr{rhs} of \eqref{eq:DN-bd} also bounds $\Gamma_{N,k}$ which thereby completes
the proof of \eqref{eq:cont-phi}.
\end{proof}

\red{\subsection{Reduction to finite mixtures}\label{subsec:finite-m} We shall utilize in the sequel 
the representation \eqref{eq:H0-barH} to control the effect 
of the discrepancy between $\V$ and $\V'$ in \eqref{eq:exp-err-LT}.
%of Theorem \ref{thm-macro} 
While changing $\qo$ to $\qo'$ requires to move $\bxo$,  the only
real constraint here is the specified angle with $\bxs = \qs \sqrt{N} \be_1$ (see Remark \ref{rem:DS21}).
Hence,  \abbr{wlog} we set hereafter $\bxo \in {\rm sp}\{\bxs,\be_2\}$ with 
$\langle \bxo, \be_2 \rangle \ge 0$.  Also,  when transforming from $H_{\bJ}$ to $H_{\BJ}^{[m]}$ of 
% under the coupling 
\eqref{potential-m},  the events $\cptq(\V)$ 
%of \eqref{eq:cond-J} 
relax somewhat,  
motivating the following generalization of $\cptq(\V)$}
 and the sets $\vB_\epsilon(\V)$ of 
Definition \ref{def:vB-eps}.
\begin{defi}\label{def:ext-cond}
\red{Fix $\qs \in [0,1]$ and an orthogonal basis of $\R^N$ consisting of unit vectors $\hbxs:=\be_1$,  $\bz:=\be_2$ 
and the rows of the matrix $M \in\R^{N-2\times N}$.  Set $\bxs := \qs \sqrt{N} \hbxs \in \qs \SN$ and 
$\bxo \in {\Bbb S}_{\bxs}(\qo) \cap \sp\{\bxs,\bz\}$ such that $\langle\bxo, \bz \rangle \ge 0$.}
% whenever $|\qo|<\qs$ and $\bxo: = \sqrt{N} \bz$ when $\qs=0$.}
For $\qs>0$ consider the event 
\[
\purple{\cpteq}(\Ve):=\{ \Ha=\Ve\} := \{\hat{\Ha}=\hat{V},   \Ha_\perp = \bu \}
\]
where
% $\hat{V}=(\Ep,\Eg,\Gg,G) \in \R^4$,  
$\Ha_{\perp} := - M \nabla H_{\BJ}(\bxs) \in \R^{N-2}$,  $\Ve=(\hat{V},\bu)$ and 
\begin{equation}\label{def:calH-c}
 \hat{\Ha} := 
		- \Big(\,
		 \frac{1}{N}H_{\BJ}(\bxo), \,\frac{1}{N}H_{\BJ}(\bxs), 
		\, \frac{1}{\|\bxs\|}\partial_{\hbxs} H_{\BJ}(\bxs),\,\frac{1}{\|\bxs\|}\partial_{\bz} H_{\BJ}(\bxs)\Big)\,.
\end{equation}
We similarly denote for $H_{\BJ}^{[m]}$ and $m \le \infty$ as in Definition \ref{def:H-m-coupling},
the vectors $\Ha^{[m]} = (\hat{\Ha}^{[m]},\Ha^{[m]}_{\perp})$ and 
%the corresponding 
events $\cpteqm (\Ve) := \{ \Ha^{[m]} = \Ve\}$.  In particular,  
$\cpteqm (\Ep,\Eg,\Gg,0,{\bf 0})=\cptq^{[m]}(\V)$ of \eqref{eq:cond-J}.  

For $\qs,\epsilon>0$ and $\Ve=(\Eg,\Ep,\Gg,G,{\bf u})$  we set 
\begin{align}\label{def:vBp}
		\vB'_\epsilon(\Ve) :=\Big\{(\hat{V}',{\bf u}') :  \|\hat{V}'-\hat{V}\|_\infty < \epsilon,
		 \|{\bf u}'-{\bf u}\| <\epsilon \sqrt{N} \Big\}\,,	
\end{align}
with $\vB'_\epsilon(\V):=\{ (\Ve',\qop) : \Ve' \in \vB'_\epsilon(\Ep,\Eg,\Gg,0,{\bf 0}),  |\qop-\qo| < \epsilon,  |\qop| \le \qs\}$.
When $\qs=0$ we only consider $\Ve=(E,0,0,0,{\bf 0})$,  so in that case $\cpteo(\Ve)=\cpto(E)$ and
$\vB'_\epsilon(\V)=\vB_\epsilon(E) := \{\Ep': |\Ep'-\Ep|<\epsilon\}$.
\end{defi}

\red{Recall the Hamiltonian $H^{[m]}_\BJ$ of \eqref{potential-m} for the truncated finite mixture
of \eqref{eq:nu-m-def},  and analogously to $H^{\qo}_{\BJ}$ of \eqref{eq:H0-barH},  consider
the centered fields $\widehat{H}_\BJ^{{\qo},[m]}$  
whose covariance been modified by conditioning on $\cptq(0,0,0,\qo)$.
We then} have conditionally on $\cpteq(\Ve)$,  \red{the representations}
\begin{align}\label{eq:H0-barHe}
H_{\BJ}&=H^{\qo}_{\BJ}+\bar{H}_{\Ve},  \qquad \quad \bar{H}_{\Ve} (\bx):=\E[H_{\BJ}(\bx) | \cpteq(\Ve)] \,,\\
H^{[m]}_{\BJ} &= \widehat{H}_\BJ^{{\qo},[m]}+\widehat{H}^{[m]}_{\Ve},  \quad 
\widehat{H}^{[m]}_{\Ve} (\bx) := \E[H^{[m]}_{\BJ}(\bx) | \cpteq(\Ve)] \,.
\label{eq:H0-barHe-m}
\end{align}
Appendix \ref{sec:Appendix} is devoted to the proof of our next lemma which 
provides the key control on the regularity of 
our mixed $p$-spin Gaussian fields $H_{\BJ}(\cdot)$, 
%of \eqref{eq:nudef}-\eqref{eq:r-star},  
conditional on $\cpteq(\Ve)$ (in particular,  on those of  \eqref{eq:cond-J} and \eqref{eq:cond-HT}).
\begin{lem}\label{lem:G-conc}
Fix $\nu$ satisfying \eqref{eq:r-star} which is not $p$-pure,  $\qs \in [0, 1]$ and allowed $\V$
\red{(see Definition \ref{ft:cons}).}\\
(a).  For $\delta=\delta(\qo)>0$ small enough,  there exist $\kappa=\kappa(\delta,\nu,\qs,\V)<\infty$ such that 
\begin{align}
\label{Grad-Lip-conc}
\limsup_{N \to \infty} \sup_{|\qo| \le \qs}  \frac{1}{N} \log \P\big( \|{\rm Hess}(H^{\qo}_{\BJ})\|_\infty \ge \kappa \, \big) & < 0\,,  \qquad 
\sup_N \sup_{\V' \in \vB_\delta(\V)} \{ \|{\rm Hess} (\bar H_{\V'}) \|_\infty \} \le \kappa \,.
\end{align}
This applies also for ${\rm Hess}(H^{\qo}_{\BJ}-\widehat{H}^{{\qo},[m]}_{\BJ})$ and ${\rm Hess}(\bar H_{\V'}-\widehat H^{[m]}_{\V'})$,  
% respectively,  
with $\kappa_m(\delta,\nu,\qs,\V) \to 0$ as $m \to \infty$.\\
(b).  For $\vB'_\epsilon(\V)$ of Definition \ref{def:ext-cond} 
\begin{align}\label{eq:Grad-CE-cont}
\sup \{ \| {\rm Hess}(\bar H_{\Ve'} - \bar H_{\V})\|_\infty : 
(\Ve',\qop) \in \vB'_\epsilon(\V),  N \ge 1 \} \to  0 \,\quad \hbox{as} \quad \epsilon \to 0\,.
\end{align}
(c).  \red{The centered fields $\HJs$ and $H^{\qo}_{\BJ}$,  whose covariances are 
modified by conditioning on  $\cptq$ of \eqref{eq:cond-star},  and respectively,  
on $\cptq(0,0,0,\qo)$,  can be coupled}
such that 
\begin{align}
\label{Grad-Lip-o}
\limsup_{N \to \infty} \sup_{\stackrel{|\qo| \le 1-\delta}{|\qo|\le \qs}} 
\frac{1}{N} \log \P\big( \|{\rm Hess}(H^{\qo}_{\BJ}-\HJs)\|_\infty \ge 2\delta\big) & < 0\,,    \qquad \forall \delta>0\,.
\end{align}
\end{lem} 
\begin{remark}\label{rem:trivial-norm-bd}  The Hessian is additive and its supremum-norm is convex.  Considering 
\eqref{Grad-Lip-conc} for the decomposition of $H_{\BJ}-H^{[m]}_{\BJ}$ given $\cptq(\V')$,  
shows that then $\|{\rm Hess}(H_{\BJ}-H^{[m]}_{\BJ})\|_\infty$ exceeds $\delta$  
with $e^{-\Omega_N(1)}$-probability,  uniformly over $\V' \in \vB_{\epsilon}(\V)$, 
for any $m \ge m_\star(\delta,\V)$ and $\epsilon>0$ small enough.  Likewise,  upon considering also 
\eqref{Grad-Lip-conc} for $H_{\BJ}$ given $\cptq(\V')$,  we see that 
$\|{\rm Hess}(H^{[m]}_{\BJ})\|_\infty$ exceeds $3\kappa$ with $e^{-\Omega_N(1)}$-probability,
uniformly over $\V' \in \vB_{\epsilon}(\V)$,   for any $m \ge m_\star(\kappa,\V)$ and small $\epsilon>0$.  
% uniformly over $\Ve' \in \vB'_\epsilon(\V)$.
By \eqref{eq:L-op},  it thus follows that
%the same applies also for $\|\nabla (H_{\BJ}-H^{[m]}_{\BJ})\|_{\rm L}$ and
%$\|\nabla H^{[m]}_{\BJ}\|_{\rm L}$.  Further,  by 
%\eqref{Grad-Lip-o}-\eqref{eq:Grad-CE-cont} we can replace $\nabla H_{\BJ}$ given 
%$\cpteq(\Ve')$ by $\nabla (\HJs+ \bar H_{\V})$,  uniformly over $\Ve' \in \vB_\epsilon'(\V)$.
%Combining 
%Definition \ref{def:H-m-coupling},  with \eqref{eq:L-op} and both results of Lemma \ref{lem:G-conc}, 
% and \eqref{eq:Grad-CE-cont}
%we thus deduce that
for any $\delta>0$ and $m \ge m_\star(\delta,\nu,\qs,\V)$,
\begin{align}
\label{Grad-lip-apx}
\lim_{\epsilon \to 0} \limsup_{N \to \infty} 
\sup_{\V' \in \vB_\epsilon(\V)}
\frac{1}{N}
 \log \P(\|\nabla (H_{\BJ}-H_{\BJ}^{[m]}) \|_{\rm L} \ge \delta 
\,|\,  \cptq (\V')
) & < 0\,,
\end{align}
whereas for any $m \ge m_\star(\kappa,\nu,\qs,\V)$, 
\begin{align}\label{Grad-sup-conc}
\lim_{\epsilon \to 0} \limsup_{N \to \infty} \sup_{\V' \in \vB_\epsilon(\V)}
\frac{1}{N}
 \log \P( \|\nabla H^{[m]}_{\BJ}\|_{\rm L} \ge 3 \kappa  \,|\,  \cptq (\V')
) & < 0 \,.
\end{align}
%Further,  it then follows from \eqref{eq:L-op} that \eqref{Grad-sup-conc} applies for $\|\nabla H^{[m]}_{\BJ}\|_{\infty}$ 
%and $\|H^{[m]}_{\BJ}\|_{\rm L}$ upon scaling by $r_\star \sqrt{N}$.  Similarly,  we can replace the variables 
%in \eqref{Grad-lip-apx} by $\|\nabla (H_{\BJ}-H_{\BJ}^{[m]}) \|_\infty/(r_\star \sqrt{N})$,  hence also by   
%$\|H_{\BJ}-H_{\BJ}^{[m]}\|_{\rm L}/(r_\star \sqrt{N})$ and $\|H_{\BJ}-H_{\BJ}^{[m]}\|_\infty/(r_\star^2 N)$.
\end{remark}

Combining Proposition \ref{prop:existence} and \eqref{Grad-sup-conc} we next deduce the a.s.  existence of a
unique strong solution of \eqref{diffusion} which is $\mu^N_{2\b,\BJ}$-reversible and up to $e^{-\Omega(N)}$-small probability
is confined away from the boundary of $\B^N_\star$.
\begin{cor}\label{cor:ex} 
Assuming \eqref{eq:nudef}-\eqref{eq:r-star},  for any $N,\b<\infty$,  $r_\star < \bar r$,  $\bxo \in \B^N(\sqrt{N})$ and 
almost every $\BJ$,  the dynamics \eqref{diffusion} admits a unique strong solution 
$\bx: \R_+ \mapsto \B^N_\star$ which is reversible \abbr{wrt} 
% the random Gibbs measure 
$\mu^N_{2\b,\BJ}$
%on $\B^N_\star$
of \eqref{Gibbs}-\eqref{def:ZbeJ}. \\
(a).  Fixing $\qs \in [0, 1]$ and an allowed $\V$,  there also exists $r_1 \in (1,r_\star)$ such that for all large enough $m \le \infty$,
\begin{equation}\label{eq:exp-conf}
\lim_{\epsilon \to 0} \limsup_{N \to \infty} \sup_{\V' \in \vB_\epsilon(\V)}
\frac{1}{N} \log \P\big( \tau^{[m]}_{r_1} \le T  \,|\,  \cptq (\V') \big) < 0 \,,   \qquad \forall T < \infty,
\end{equation}
%for the events $\cpt_+(\cdot)$ of \eqref{eq:cond-J} and
% parameters in 
%the sets $\vB(\epsilon)$ of \eqref{eq:bar-B} 
where $\tau^{[m]}_r := \inf\{ t \ge 0 : N^{-1/2} \| \bx^{[m]}_t \| \ge r \}$ for the a.s.  strong solution
$\bx^{[m]}_t$ of \eqref{diffusion} under the mixture $\nu^{[m]}(\cdot)$ and $\P$ denotes the 
product measure over $\BJ$ and $\{\BB_t,  t \ge 0\}$. 
In case $\bxo \in \SN$,  one further has \eqref{eq:exp-conf} holding for some $c<\infty$,  all large enough $m,\ell$
and $\tau_{c}^{[m]} := \inf\{ t \ge 0 :  N^{-1/2} \| \bx^{[m]}_t \| \notin (1-\frac{c}{\ell},1+\frac{c}{\ell} ) \}$.\\
\red{(b).  Denote by $\HJsm$ and $\bar{H}^{[m]}_{\V}$,  respectively,  the fields $\HJs$ and $\bar{H}_{\V}$ 
for the model $\nu^{[m]}$,  and by $\tau^{\star,[m]}_r$ the analog of $\tau^{[m]}_r$ for the solution of 
\eqref{diffusion-phi} with $\varphi_N=\HJsm+\bar{H}^{[m]}_{\V}$.  Then,  for some $r_1 \in (1,r_\star)$
and all large $m$,
\begin{equation}\label{eq:exp-conf-Ho}
\limsup_{N \to \infty} 
\frac{1}{N} \log \P\big( \tau^{\star,[m]}_{r_1} \le T ) < 0 \,,  \qquad \forall T < \infty \,.
\end{equation}}
\end{cor}
\begin{proof}
Recall that under \eqref{eq:r-star},  the Gaussian fields $H_\BJ$ of \eqref{potential} are in $C_0^2(\B^N)$,
so the a.s.  existence of a strong solution of \eqref{diffusion} is guaranteed by Proposition \ref{prop:existence}(a). 
Having locally Lipschitz $f'(\cdot)$ as in \eqref{eq:fdef} implies that $f(\cdot)$ is uniformly 
bounded below on $[0,r_\star^2]$,  so with $\|H_{\BJ}\|_\infty$ a.s.  finite,  the same applies for
$Z_{2 \b,\BJ}$ of \eqref{def:ZbeJ}.  The random Gibbs measure $\mu^N_{2\b,\BJ}$ of 
\eqref{Gibbs} is thus a.s.  well defined and the reversibility of the solution of \eqref{diffusion} 
follows by the general theory of Langevin diffusions.  \red{In addition,  \abbr{wlog} 
$\nu^{[m]}$ is not pure $p$-spins once $m \ge m_\star$.\\
(a).}  Turning to \eqref{eq:exp-conf},  fixing $\qs$ and 
an allowed $\V$ we deduce from \eqref{eq:L-op} and \eqref{Grad-sup-conc} 
that for any $m \in [m_\star,\infty]$,  there exist $\epsilon_o,\eta>0$,  $C<\infty$ such that 
\begin{equation}\label{eq:bd-kappa-N}
\sup_{\V' \in \vB_{\epsilon_o}(\V)} \P\big( \|\nabla H^{[m]}_{\BJ}\|_\infty \ge 3 \kappa r_\star \sqrt{N} \,|\,  \cptq(\V') \big) \le C e^{-N \eta} \,.
\end{equation}
From the proof of Proposition \ref{prop:existence}(b)-(c) we also see that
$\|\nabla H^{[m]}_{\BJ}\|_\infty \le 3 \kappa r_\star \sqrt{N}$
results with \eqref{eq:exp-conf} holding for some $r_1$,  respectively $c$,  up to an $e^{-\Omega(N)}$-small probability, 
with neither $r_1$ nor $c$ depending on $m$ (or $T$).\\
\red{(b).  Denoting by $H_{\BJ}^{\qo,[m]}$ the centered field 
whose covariance been modified by conditioning on $\cpteqm (\vec{\bf 0})$,  
we have analogously to \eqref{eq:H0-barHe-m},  that conditionally on the event $\cpteqm (\Ve)$, }
\begin{equation}\label{eq:H0-barH-m}
H^{[m]}_{\BJ} = H_{\BJ}^{\qo,[m]}+\bar H^{[m]}_{\Ve},  \qquad 
\bar H^{[m]}_{\Ve} (\bx) := \E[H^{[m]}_{\BJ}(\bx) | \cpteqm (\Ve)] \,.
\end{equation}
\red{In particular,  by part (a) for $\nu=\nu^{[m]}$ and $\V'=\V$,  we arrive at \eqref{eq:exp-conf-Ho} 
for the solution of \eqref{diffusion-phi} with $\varphi_N=H_\BJ^{{\qo},[m]}+\bar{H}^{[m]}_{\V}$ 
and any $m \ge m_\star$ large enough.  Coupling $\HJsm$ with $H_\BJ^{{\qo},[m]}$ as in 
Lemma \ref{lem:G-conc}(c),  it further follows from 
\eqref{eq:L-op} and \eqref{Grad-Lip-o} that 
$\|\nabla (H_\BJ^{{\qo},[m]} - \HJsm) \|_\infty \le 2 \delta r_\star \sqrt{N}$ up to an $e^{-\Omega(N)}$-small 
probability.  As in our proof of part (a),  under the latter event the bound \eqref{eq:exp-conf-Ho} extends 
to the solution of \eqref{diffusion-phi} with $\varphi_N=\HJsm+\bar{H}^{[m]}_{\V}$,  at slightly larger $r_1$ 
(and for all $m$ large enough).
}
\end{proof}

Hereafter we take $1<r_\star<\bar{r}$ and $f=f_\ell$ as in \eqref{eq:fdef} 
with $f_0(\cdot)$ of locally Lipschitz derivative on $[0,r_\star^2)$,
coupling $(H_{\BJ},H^{[m]}_{\BJ})$ as in Definition \ref{def:H-m-coupling} 
for $\nu(\cdot)$ satisfying  \eqref{eq:r-star}-\eqref{eq:nudef}.
Fixing $\qs \in [0,1]$ and an allowed vector $\V$,  
\red{
%Armed with the preceding results,  
we proceed to show that Theorem \ref{thm-macro} is a consequence of Proposition \ref{prop-macro-easy},  via the following four steps:\\
{\bf I.} Replacing $H_{\BJ}$ in \eqref{eq:exp-err-LT} by $H_{\BJ}^{[m]}$. \\
{\bf II.} Instead of conditioning there on $\cptq (\V')=\cptq^{[\infty]}(\V')$
% in \eqref{eq:exp-err-LT}, 
condition on $\cpteqm (\Ve')$ of Definition \ref{def:ext-cond}.\\
{\bf III.} Replacing such conditional statement by requiring \eqref{eq:exp-err-LT} for the
un-conditional field $\HJsm+\bar H^{[m]}_{\V}$.\\
{\bf IV.} Replacing $f_0'(\cdot)$ as in \eqref{eq:fdef} with
%which diverges at $r^2_\star$ 
$f_0'(\cdot)$ locally Lipschitz that satisfies \cite[(1.6) \& (1.10)]{DS21},  with 
the limiting dynamic}  $\bU_\star:=\bUqf(\V;\nu)$ 
of the model $\nu$ per Definitions \ref{def:limit-dyn-low-temp} and \ref{def:limit-dyn-high-temp},
\red{replaced by $\bUqf(\V;\nu^{[m]})$.}

\noindent
{\bf I.} Conditioning on the events $\cptq (\V')$ of \eqref{eq:cond-J} for $\V' \in \vB_\epsilon(\V)$ of Definition \ref{def:vB-eps},  consider for $N,T,\b<\infty$ and the corresponding strong solutions of \eqref{diffusion},  the variable
\[
\Delta \widehat{\err}_{N,T}^{[m]} (\BJ) :=\sup_{\bU_\infty}
|\widehat\err_{N,T}(\bxo,\bxs,H_{\BJ};\bU_\infty) -  \widehat\err_{N,T}(\bxo,\bxs,H^{[m]}_{\BJ};\bU_\infty)|
\]
\red{for $\bxo$ and $\bxs$ as in Definition \ref{def:ext-cond}.  Applying}
Proposition \ref{prop:cont} for $\varphi_N^{(1)}=H_{\BJ}$ 
and $\varphi_N^{(2)}=H^{[m]}_{\BJ}$ we deduce from \eqref{Grad-lip-apx}-\eqref{Grad-sup-conc} 
(using again \eqref{eq:L-op}),  by setting $\delta=\delta(\zeta)>0$ small enough so  
$\zeta \ge \gamma_{T,\b}(3\kappa,\delta r_\star)$ and the corresponding value of 
$m_\star$ in \eqref{Grad-lip-apx},  that for any $m \ge m_o(\zeta,f,\nu,\qs,\V)$,
\begin{equation}\label{eq:m-apx-err}
\lim_{\epsilon \to 0} \limsup_{N \to \infty} \sup_{(\bxo,\V') \in \vB_\epsilon(\bxs,\V)} \frac{1}{N} \log 
\P\Big(\Delta \widehat\err_{N,T}^{[m]} (\BJ) \ge \zeta \,\big|\,  \cptq (\V') \Big)  < 0 \,.
\end{equation}
% Denoting by $\cptq^{[m]}(\cdot)$ the events in
%\eqref{eq:cond-J} and \eqref{eq:cond-HT},  now for $H_{\BJ}^{[m]}(\cdot)$,  $m \le \infty$,  by \eqref{eq:m-apx-err},  
It thus suffices for \eqref{eq:exp-err-LT} to show that for any $\delta>0$ 
and $m \ge m_o(\delta,\V)$ large enough 
\begin{align}\label{eq:m-pre-apx}
\lim_{\epsilon \to 0} 
\limsup_{N \to \infty}  \sup_{(\bxo,\V') \in \vB_\epsilon(\bxs,\V)}
\frac{1}{N} \log \P \Big( \widehat\err_{N,T}\big(H^{[m]}_{\BJ};\bU_\star \big) > 5 \delta  \,  \big| \,
\cptq^{[\infty]}(\V') \Big) &<  0 \,,
\end{align}
\purple{where setting hereafter} $(\bxo,\bxs)$ as dictated by $(\qs,\qop)$,  \purple{we omit those arguments 
in $\widehat\err_{N,T}(\cdot;\cdot)$.  Further,} thanks to \eqref{eq:exp-conf} we
restrict the solution of \eqref{diffusion} for conditionally on $\cptq^{[\infty]}(\V')$ finite 
mixtures $H^{[m]}_{\BJ}$,  to $\B^N(r_1 \sqrt{N})$,  with $r_1<r_\star$ that does not vary in $m$
and with $f'$ globally Lipschitz on $[0,r_1^2]$.  

\noindent
\red{{\bf II.} Since} $|\qo| < 1$ implies that 
$1-|\qop| \ge C^{-1} >0$ for some $\epsilon_o>0$ and all $\V' \in \vB_{\epsilon_o}(\V)$,  
our next lemma allows for conversion of conditioning on 
$\cptq^{[\infty]}(\V')$ in \eqref{eq:m-pre-apx},  to a controlled enlargement of the parameters 
in $\cptq^{[m]}(\cdot)$.  Specifically,  in view of Lemma \ref{lem:Cinf-to-Cm} it
suffices for \eqref{eq:m-pre-apx} to show that for all $m \ge m_o(\delta,\V)$, 
\begin{align}\label{eq:m-apx}
\lim_{\epsilon \to 0} 
\limsup_{N \to \infty}  \sup_{(\Ve',\qop) \in \vB'_\epsilon(\V)}
\frac{1}{N} \log \P \Big( \widehat\err_{N,T}\big(H^{[m]}_{\BJ};\bU_\star \big) > 5 \delta  \,  \big| \,
\cpteqm (\Ve') \Big) &<  0 \,.
\end{align}
\begin{lem}\label{lem:Cinf-to-Cm}
For any $t,C< \infty$,  $\epsilon>0$,  $\qs \in [0,1]$ there exist finite $m_0$,  $N_0$, 
so for any $\V=(\Ep,\Eg,\Gg,\qo)$ and $\Ve=(\Ep,\Eg,\Gg,0,{\bf 0})$ with 
$1-|\qo| \ge C^{-1} \vee (1-\qs)$,  $\|\Ve\| \le C$ 
and all $m\geq m_0$,  $N\ge N_0$,
	\begin{equation}\label{eq:Cinf-to-Cm}
	\P\Big(H_{\BJ}^{[m]}\in \cdot \,|\, \cptq^{[\infty]}(\V) \Big) 
	\leq e^{-N t} +
	\sup_{\Ve' \in \vB'_\epsilon(\Ve)} \P\Big(H_{\BJ}^{[m]}\in \cdot \,|\, \cpteqm (\Ve') \Big) \,.
	\end{equation}	
%where if $\qs=0$ we further restrict the \abbr{rhs} of \eqref{eq:Cinf-to-Cm} to only $\Ve'=(\Ep',0,0,0,{\bf 0})$.
\end{lem}
\begin{proof} \abbr{Wlog} the mixture $\nu(\cdot)$ is not finite,  thereby excluding pure $p$-spins.
The random process $H_{\BJ}^{\Delta}:=H^{[\infty]}_{\BJ}-H_{\BJ}^{[m]}$, 
which corresponds to the mixture $\nu^\Delta := \nu-\nu^{[m]}$,  is independent of $H_{\BJ}^{[m]}$.  We thus 
have the following representation of the \abbr{lhs} of \eqref{eq:Cinf-to-Cm}, 
	\begin{align*}
		\P\Big(H_{\BJ}^{[m]}\in \cdot\,|\, \Ha^{[\infty]} =  \Ve \Big)= 
		\int \P\Big(H_{\BJ}^{[m]}\in \cdot\,|\, \Ha^{[m]} =  \Ve' \Big)\purple{p}(\Ve')d\Ve'\,,
	\end{align*}
	where $p(\Ve')$ denotes the conditional density of $\Ha^{[m]}$ 
	given that $\Ha^{[\infty]}=\Ve$.  With the conditional probability on the \abbr{RHS} bounded by
	one,  take $\Ha^{\Delta}:=\Ha^{[\infty]}-\Ha^{[m]}$ and 
	denote by $\hat{\Ha}^{[\infty]}$ and $\hat{\Ha}^{\Delta}$
the first four elements of $\Ha^{[\infty]}$ and $\Ha^{\Delta}$,  respectively,
	it suffices to show that 
	\begin{equation}\label{eq:HDelta}
		\P\Big(\Ha^{[m]} \notin \vB'_\epsilon(\Ve) \,\Big|\, \Ha^{[\infty]}=\Ve \Big) =
	\P\Big(
			\{ \|\hat{\Ha}^{\Delta}\|_\infty \ge \epsilon\}  \bigcup 
			\big\{ \|\Ha_\perp\| \ge \epsilon \sqrt{N} \big\}
			 \,\Big|\, \Ha^{[\infty]}=\Ve \Big)
			\leq e^{-Nt} \,.
	\end{equation}	
%	Conditionally on $\mathcal{H}_{\BJ}=(E,E',G,0,0)$, $\mathcal{H}_{\BJ}^{\Delta}\in\R^4\times\R^{N-2}$ is a collection of jointly Gaussian variables. To prove \eqref{eq:HDelta}, we need to show that their conditional expectation and (normalized) variance go to $0$ as $m\to\infty$.
Considering first $\qs>0$,  since $(\hat{\Ha}^\Delta,\hat{\Ha}^{[\infty]})$ and $(\Ha^\Delta_\perp,\Ha^{[\infty]}_\perp)$ are independent,  the conditional law of 
$\Ha^{\Delta}$ given $\{\Ha^{[\infty]}=\Ve\}$ is the product of the Gaussian laws of 
$\hat{\Ha}^{\Delta}$ given $\{\hat{\Ha}^{[\infty]}=(\Ep,\Eg,\Gg,0)\}$ and the 
Gaussian law of $\Ha^{\Delta}_\perp$ given $\{\Ha^{[\infty]}_\perp={\bf 0}\}$. 
Now,  by the well-known formula for conditional expectation of Gaussian variables 
\begin{align}\label{eq:M-ce}
\E \big[ \Ha^{\Delta}_\perp \,|\,  \Ha^{[\infty]}_\perp={\bf 0}] &= {\bf 0} \,, \\
\E \Big[ \hat{\Ha}^{\Delta}\,|\,  \hat{\Ha}^{[\infty]}=(\Ep,\Eg,\Gg,0) \,  \Big] &=
\Sigma_{\nu^\Delta} \Sigma_\nu^{-1} (\Ep,\Eg,\Gg,0)^{\top},
\label{eq:hat-ce}
\end{align}
for $\Sigma_\nu=\Sigma_{\nu}(\qo)$ of \eqref{def:w-s} and with $\Sigma_{\nu^\Delta}$ 
similarly defined for the mixture $\nu^\Delta$.
Avoiding pure $p$-spins and having $\qs>0$,  $|\qo|<1$ yields a positive definite
$\Sigma_\nu (\qo)$ (see Lemma \ref{lem:app}(a)),  with the conditional expectation in \eqref{eq:hat-ce} 
independent of $N$ and $\bxo$ given the value of $\qo$.  
Fixing $\V$,  this conditional expectation goes to zero as $m\to\infty$ 
since $\Sigma_{\nu^\Delta} \to 0$ entrywise.  This convergence is uniform 
under our assumed compact range of $(\Ep,\Eg,\Gg)$,  where keeping $1-|\qo|$ away from zero 
gives a uniform control on $\|\Sigma_\nu^{-1}(\qo)\|_2$.
% that the minimal eigenvalue of $\Sigma_\nu$ stays bounded away from zero.  
The variance of elements of $\sqrt{N} \hat{\Ha}^{\Delta}_{\BJ}$ 
are merely the diagonal elements of $\Sigma_{\nu^\Delta}$,  which go to zero as $m \to \infty$,
and the conditional variances 
%(given the value of $\mathcal{H}^{[\infty]}_{\BJ}$),  
can only be smaller.  Finally,  
the (N-2)-dimensional Gaussian vectors $(\Ha^{\Delta}_\perp,\Ha^{[\infty]}_\perp)$ 
have jointly \abbr{iid} coordinates and the conditional variance of each coordinate of 
$\Ha^{\Delta}_\perp$ is at most its unconditional variance $(\nu^\Delta)'(\qs^2)$ which 
goes to zero as $m \to \infty$.
% (independently of $N$,  $\bxo$ and $\Ep,\Eg,\Gg$).  
Combining this with the zero conditional mean of \eqref{eq:M-ce},
and our observations regarding the conditional mean and variance of $\hat{\Ha}^{\Delta}$,
results with the uniform bound of \eqref{eq:HDelta}.

For $\qs=0$ we have $\bxs={\bf 0}$ so we merely control the conditional Gaussian law 
of the first element of $\hat{\Ha}^{\Delta}$,  uniformly in $|\Ep| \le C$.  This is then 
a simple one-dimensional special case of the preceding argument.
\end{proof}

\noindent
\red{{\bf III.} 
%Next,  denoting by $H_{\BJ}^{\qo,[m]}$ the centered field 
%whose covariance been modified by conditioning on $\cpteqm (\vec{\bf 0})$,  
%we have analogously to \eqref{eq:H0-barHe-m},  that conditionally on the event $\cpteqm (\Ve)$, }
%\begin{equation}\label{eq:H0-barH-m}
%H^{[m]}_{\BJ} = H_{\BJ}^{\qo,[m]}+\bar H^{[m]}_{\Ve},  \qquad 
%\bar H^{[m]}_{\Ve} (\bx) := \E[H^{[m]}_{\BJ}(\bx) | \cpteqm (\Ve)] \,.
%\end{equation}
%\red{In particular, 
Fixing $m$ and allowed $\V'=(\Ep,\Eg,\Gg,\qop)$,  recall
% the representation 
\eqref{eq:H0-barH-m} and} 
note that \eqref{Grad-Lip-conc} holds for the model $\nu^{[m]}$.  Consequently,  by the reasoning of Remark \ref{rem:trivial-norm-bd},  we get analogously to \eqref{Grad-sup-conc},  that for some finite
$\kappa=\kappa(\nu^{[m]},\V')$,
\begin{align}\label{Grad-m-conc}
\limsup_{N \to \infty} 
\frac{1}{N}
 \log \P( \|\nabla H^{[m]}_{\BJ}\|_{\rm L} \ge 3 \kappa  \,|\,  \cptq^{[m]} (\V')) & < 0 \,.
\end{align}
Conditioning in \eqref{eq:m-apx} on $\cptq^{[m]}(\V')$ \red{instead of 
$\cpteqm (\Ve')$},
merely changes the potential $\varphi_N^{(2)}:=H_{\BJ}^{\qop,[m]}+\bar{H}^{[m]}_{\Ve'}$ to 
$\varphi_N^{(1)}:=H_{\BJ}^{\qop,[m]}+\bar{H}^{[m]}_{\Ve}$ with $\Ve=(\Ep,\Eg,\Gg,0,{\bf 0})$.
For $m \ge m_\star$ finite,  the mixture $\nu^{[m]}$ is not pure $p$-spins.  In that case,
\eqref{eq:Grad-CE-cont} applies for $\nu^{[m]}$ and by \eqref{eq:L-op} we conclude that
$\sup_{\vB'_{\epsilon} (\Ve)} \{ \|\nabla (\varphi^{(2)}_N-\varphi^{(1)}_N)\|_\infty \} \le \delta \sqrt{N}$ for any 
$\delta>0$,  $N \ge 1$,  provided $\epsilon \le \epsilon_o(\delta,m,\V')$.  It thus follows by 
Proposition \ref{prop:cont} that up to the negligible event of \eqref{Grad-m-conc},  as $\epsilon \to 0$, 
\[
\sup_{N,\vB'_\epsilon(\Ve)} 
\big|\widehat\err_{N,T}\big(\varphi_N^{(2)};\bU_\star \big)-\widehat\err_{N,T}\big(\varphi_N^{(1)};\bU_\star \big) \big| \to 0 \,,
\]
thereby allowing us to replace \eqref{eq:m-apx} by the easier task of showing that for all $m \ge m_o(\delta,\V)$,
\begin{align}\label{eq:m-apx2}
\lim_{\epsilon \to 0} 
\limsup_{N \to \infty}  \sup_{\stackrel{|\qop-\qo| <\epsilon}{|\qop| \le \qs}}
\frac{1}{N} \log \P \Big( \widehat\err_{N,T}\big(H^{\qop,[m]}_{\BJ}+\bar H^{[m]}_{\V'};\bU_\star \big) > 4 \delta  \,  \Big) &<  0 
\end{align}
\red{(where we rely also on \eqref{eq:H0-barH-m}} \purple{and set $\frac{1}{N} \langle \bxo,\bxs \rangle=\qop$)}.
% and use}  $H^{\qo,[m]}_{\BJ,\qop}$ to stress that the value of $\qop$ affects the covariance of this field). 
Taking $2 \delta < 1-|\qo|$,  we deduce by considering \eqref{Grad-Lip-o} for the mixture $\nu^{[m]}$,  that  
$\|{\rm Hess} (H^{\qop,[m]}_{\BJ} - \HJsm)\|_\infty \le \delta$ up to an $e^{-\Omega(N)}$-small 
probability,  uniformly over $|\qop-\qo| < \delta$.  In case $\qs>0$ we also recall \eqref{eq:Grad-CE-cont},
that for all $m$ large enough,
$ \| {\rm Hess}(\bar H^{[m]}_{\V'} - \bar H^{[m]}_{\V})\|_\infty \to 0$ as $\epsilon \to 0$,  uniformly in $N$ 
and $|\qop-\qo|<\epsilon$.  Utilizing once more Proposition \ref{prop:cont},  we thus deduce that up to the exponentially 
negligible probabilities in our applications of \eqref{Grad-Lip-o} and \eqref{Grad-m-conc},  the bound of 
\eqref{eq:m-apx2} is a consequence of having 
\begin{equation}\label{eq:to-H-o}
\limsup_{N \to \infty}
\frac{1}{N} \log \P \Big( \widehat\err_{N,T}\big(\HJsm+\bar H^{[m]}_{\V};\bU_\star \big) > 2 \delta  \,  \Big) <  0 
\end{equation}
\purple{(where by continuity and rotational invariance of 
$\bxs \mapsto \widehat\err_{N,T}(\bxo,\bxs,\varphi_N;\bU_\star)$ we re-adjust to 
$\frac{1}{N} \langle \bxo,\bxs \rangle=\qo$).}

\noindent
\red{{\bf IV.} Fix $r_1 \in (1,r_\star)$ and increase $m_o$ to have \eqref{eq:exp-conf-Ho} holding
for all $m \ge m_o$.  Excluding the events $\{ \tau^{\star,[m]}_{r_1} \le T \}$ of $e^{-\Omega(N)}$-small probability,
the solutions of \eqref{diffusion-phi} with $\varphi_N=\HJsm+\bar H^{[m]}_{\V}$ and $(\bxs,\bxo)$ 
of Definition \ref{def:ext-cond} (for the prescribed $\qo$),  are unaffected by modifying 
$f(\cdot)$ to $\widetilde{f}(\cdot)$ outside $[0,r_1^2]$ whose derivative is 
locally Lipschitz throughout $[0,\infty)$ with the growth rate of Proposition \ref{prop-macro-easy}.
In particular,  the uniformly bounded $\err_{N,T}(\HJsm+\bar H^{[m]}_{\V};\bU_\star)$ are unchanged by 
such a modification in \eqref{diffusion-phi}.  Further,  setting $K_N^{[m]}(s):=Z_N(s)^2$ for $Z_N$ as in \eqref{def:KN-uk} at $k=1$,   
\begin{equation}\label{eq:tau-star-KNm}
\{ \tau^{\star,[m]}_{r_1} >  T \} \subseteq \{ \| K_N^{[m]} \|_T < r_1^2 \}\,,
\end{equation}
so the \abbr{rhs} holds up to $e^{-\Omega(N)}$-small probability,  both for the original $f(\cdot)$ 
and the modified $\widetilde{f}(\cdot)$.   Now,
Proposition \ref{prop-macro-easy} establishes \eqref{eq:to-H-o} 
with $\bUqtf(\V;\nu^{[m]})$ instead of $\bUqf(\V;\nu)$ and with the \abbr{rhs} of 
\eqref{eq:tau-star-KNm} holding with high probability,  it further follows that 
$\bUqtf(\V;\nu^{[m]})$ depends only on the values of $\widetilde{f}(r)$ for $r \le r_1$.
Hence, $\bUqtf(\V;\nu^{[m]})=\bUqf(\V;\nu^{[m]})$.  Moreover,}
the \abbr{lhs} of \eqref{eq:m-cont-Uf} holds for such $r_1$ 
(confirming Remark \ref{rmk:Uf-bd}) allowing us to deduce 
from Proposition \ref{prop:nu-cont} that $\|\bUqf(\V;\nu^{[m]})-\bU_\star\|_T \to 0$ as
$m \to \infty$,  thereby obtaining \eqref{eq:to-H-o}.

%%%%%%%%%%%%%%%%%%%%%%%%%%%%%%%%%%%%%%%%%%%%%%%%%%%%

\section{\label{sec:DS21} Proof of Proposition  \ref{prop-macro-easy}}

Hereafter we fix,  as was done in \cite{BDG2,DS21} a finite mixture \red{of maximum degree $m$ (as in 
\eqref{eq:nu-m-def}),} and a locally Lipschitz $f'(\cdot)$ which satisfies \cite[(1.6) \& (1.10)]{DS21}.
Further,  as explained in Remark \ref{rem:DS21},  
\abbr{wlog} we set $\bxs=\qs \sqrt{N} \be_1$ and have $\bxo$ sampled uniformly on ${\Bbb S}_{\bxs} (\qo)$ 
of \eqref{eq:subsphere}.  Recall the a.s.  existence of unique strong 
solutions for \eqref{diffusion-phi} with $\varphi_N=\bar H_{\V} + H^\star_\BJ$,  that consists of
the non-random potential 
%$\bar H_{\V}$ 
of  \eqref{eq:bH-V} plus $H_\BJ$ conditioned 
on $\cptq$ of \eqref{eq:cond-star}.
With $K_N(t)=C_N(t,t)$,  it further follows from \red{\eqref{eq:exp-conf-Ho}} that 
\begin{equation}\label{342n-DS21}
\limsup_{N \to \infty} \frac{1}{N} \log \P (\|K_N\|_T \ge z) < 0\, ,  \qquad \forall z \ge z_0 
\end{equation}
(which replaces here \cite[(3.42)]{DS21} and \cite[(2.3)]{BDG2}).  
As in \cite[Remark 1.2]{DS21},  our solution $\bx_t$ is unchanged by 
mapping $\nu(\cdot) \mapsto \b^2 \nu(\cdot)$,  $\V \mapsto \b \V$ and setting $\b=1$ in \eqref{diffusion-phi}, 
so \abbr{wlog} we set hereafter $\b=1$.  Having done so,  we arrive at the stochastic differential system
%(\abbr{sds})
\begin{align}
\bx_s=&\bxo+\BB_s - \int_0^s f'(K_N(u)) \bx_u du + \int_0^s \BG (\bx_u) du \nonumber \\
&+ \sqrt{N} \qs \Big[ \int_0^s \bv_x(q_N(u),C_N(u,0)) du \Big] \be_1 + \Big[ \int_0^s \bv_y(q_N(u),C_N(u,0))du \Big] \bxo \,,
\label{313n-DS21}
\end{align}
for $\BG=-\nabla H_{\BJ}^\star$ \red{and} $\bv(\cdot,\cdot)$ of \eqref{def:vt-new}.  

The convergence of $(C_N,\chi_N,q_N)$ to the non-random explicit limits of
 \cite[Thm.  1.2]{BDG2} and \cite[Thm.  1.1]{DS21}, 
 are established in the following four steps involving 
 a certain collection of random functions \red{$\Ua_N \supset (C_N,\chi_N,q_N)$:}\\
 (a).  Show that \red{$\Ua_N$} concentrate at $N \gg 1$ around their means \red{$\Ua^a_N$ and that
 $\{\Ua^a_N,   N \ge 1\}$ is pre-compact}.\\
 (b).  Derive \red{relations which any sub-sequential limit of $\{\Ua^a_N, N \geq 1\}$ must satisfy}. \\
 (c).  Show that \red{those relations amount to} a solution of the equations in \cite[Thm. 1.2]{BDG2} 
 %\red{(for $\qs=0$), }
  or \cite[Thm. 1.1]{DS21}.
  %\red{(for $\qs>0$).}
  \\
 (d).  By auxiliary,  Gronwall type arguments,  \red{show uniqueness of the latter solutions}.

\medskip 
Following here the same strategy,  we only list the necessary modifications in the definitions of the objects considered,
and when relevant,  in adapting certain specific arguments.  \red{We start by stating our modification of the enlarged
collection \red{$\Ua_N$},  first for those of \cite{DS21} (when $\qs>0$),  and then 
for the smaller collection of functions from \cite{BDG2} (when $\qs=0$).  To this end,  note that for $\qs>0$,  the law of
$(\bxo,\BJ,\BB)$ in \eqref{313n-DS21} is precisely $\P_\star$ of \cite[Pg.  481]{DS21},  with $H_{\BJ}^\star$
having} the covariance $\widetilde{k}(\bx,\by)$ of \cite[Lemma 3.7]{DS21}
(see \cite[(3.34)]{DS21},  \red{where} $\cptq={\sf CP}_\star$ of \cite[Pg.  481]{DS21}).
Thus,  \red{at $\qs>0$,  the only difference between \eqref{313n-DS21} and  \cite[(3.13)]{DS21}}
is the additional right-most-term \red{in} \eqref{313n-DS21},   \red{and we denote} 
by $\Ua_N$ the collection of random,  pre-limit
functions of \cite{DS21},  which consist of $(C_N,\chi_N)$ and the fifteen other functions of \cite[(3.2)-(3.8)]{DS21},  
apart from changing in \cite[(3.4)-(3.6)]{DS21},  to  
\begin{align}\label{34n-DS21}
\HN(s) =& \widehat{H}_N(s) + \bv(q_N(s),C_N(s,0)) \,,\\
\label{35n-DS21}
Q_N(s) =& -f'(K_N(s))q_N(s)+V_N(s) + \qs^2 \bv_x(q_N(s), C_N(s,0)) + \qo \bv_y (q_N(s),C_N(s,0)) \, , \\
\bar D_N(s,t) =& - f'(K_N(t)) \bar C_N(t,s) + \bar A_N(t,s)+\bar C_N(s,0) \bv_y(q_N(t),C_N(t,0)) \,.
\label{36n-DS21}
\end{align} 
Indeed,  \eqref{34n-DS21} registers our different conditional expectation field $\bar H_{\V}(\cdot)$,  with
\eqref{35n-DS21}-\eqref{36n-DS21} immediate consequences of \eqref{313n-DS21},  since by definition
\[
Q_N(s) := \frac{1}{N} \frac{d}{ds} \big[ \langle \bx_s-\BB_s,\bxs \rangle \big] \,, \qquad 
\bar D_N(s,t) := \frac{1}{N} \frac{d}{dt} \big[ \sum_{i=2}^N x^i_s (x^i_t-B^i_t) \big] \,.
\]
The case $\qs=0$ has the unconditional covariance $k(\cdot,\cdot)$ for $\BG(\cdot)$, 
as in \cite[Lemma 3.2]{BDG2},  and we retain in the second line of \eqref{313n-DS21} only the term 
$[\int_0^s \bv'(C_N(u,0)) du] \bxo$ for $\bv(r):=\Ep\nu(r)/\nu(1)$.  It then suffices,  as in the proof of
\cite[Thm. 1.2]{BDG2},  to consider in $\Ua_N$ only the functions $(C_N,\chi_N,\HN)$ and 
the auxiliary $(D_N,E_N,A_N,F_N)$ of \cite[(1.15)-(1.16)]{BDG2},  where \red{in our setting one only has} 
to add $C_N(s,0) \bv'(C_N(t,0))$ to $D_N(s,t)$ of \cite[(1.16)]{BDG2}.  

With these modifications in place,  we establish step (a) of the preceding program,  by the following 
analogs of  \cite[Prop.  2.3-2.4]{BDG2} and \cite[Prop.  3.1 \& 3.3]{DS21},  respectively.  We note in passing 
that the exponential in $N$ concentration rate of \eqref{310n-DS21},  guarantees the corresponding 
exponential in $N$ decay rate in \eqref{eq:exp-err-no-sup}.
\begin{prop}\label{prop:33-DS21}
Fix $T<\infty$,  $\qs \in [0,1]$ and an allowed $\V$ \red{(as in Definition \ref{ft:cons}).} Consider our modified
collection 
%of functions 
$\Ua_N$  in terms of 
% the strong solution of 
\eqref{313n-DS21},  under the law $\P_\star$ 
%of $(\bxo,\BJ,\BB)$
on $\Ea_N := \R^N \times \R^{d(N,m)} \times \C([0,T],\R^N) $,  equipped
%as in  \cite[(3.1)]{DS21},  
with 
%the norm
%$\|(\bxo,\BJ,\BB)\|$ of  Namely, 
\begin{equation}\label{eq:norm}
\red{\|(\bx_0, \BJ, \BB)\|^2= \| \bx_0\|^2 
+ \sum_{p=2}^m N^{(p-1)} \!\!\!\!\!\!\!\! \sum_{1 \leq i_1 \leq \ldots \leq i_{p} \leq N} \!\!\!\!\!\! (J_{(i_1\cdots i_{p})})^2
+ \sup_{0 \leq t \leq T}\| \BB_t \|^2 \,}
\end{equation}
\red{and the corresponding collection} $\Ua^a_N :=\{ \E_\star [U_N(s,t)] : U_N \in \Ua_N\}$ \red{of mean functions}.\\
(a).  Fix $k<\infty$,  $U_N \in \Ua_N$.  Then,  $\sup_N \{ \E_\star [ \|U_N\|_T^k] \}< \infty$ 
and $\{ \red{U^a_N}(s,t)\}_N$ is pre-compact in $(\Ca_b,\|\cdot\|_T)$. \\
(b).  For any $\rho>0$,
\begin{equation}\label{310n-DS21}
\limsup_{N \to \infty} \frac{1}{N} \log \P_\star(\|U_N - \red{U^a_N}\|_T \ge \rho) < 0 \,.
\end{equation}
\end{prop}
\begin{proof}[Proof of Proposition \ref{prop:33-DS21}]$~$\\
(a).  The Euclidean 
norm of the drift terms on the second line of \eqref{313n-DS21} is  
at most
\begin{equation}\label{eq:drift-ext-bd}
\kappa \sqrt{N} \int_0^s |K_N(u)|^{\frac{(m-1)}{2}} du,
\end{equation}
for some $\kappa=\kappa(\qs,\V)<\infty$.  When following the proofs of \cite[(2.2)]{BDG2} and \cite[(3.37)]{DS21}, 
the extra term $\kappa \int_0^{s \wedge \tau_M} C_N(t)^{\frac{m}{2}} dt$,  which the second line of \eqref{313n-DS21}
thus induces on the \abbr{rhs} of \cite[(2.5)]{BDG2} does not affect the above mentioned conclusions. 

For $\qs=0$,  the validity here of the analog of \cite[(2.2)]{BDG2} 
(restricted to time interval $[0,T]$ and with $q=k$ from \cite[(1.6)]{DS21}),  
combined with the tail condition \eqref{342n-DS21} and the bound \eqref{eq:drift-ext-bd} on
the norm of the term we added to $D_N$ in \cite[(1.16)]{BDG2},  suffice for   
the uniform boundedness of $\E_\star [ \|U_N\|_T^k]$ (see \cite[after (2.17)]{BDG2}).
Given such uniform boundedness,  pre-compactness is a direct consequence of the equi-continuity of
$\{\E_\star U_N(s,t)\}_N$ on $[0,T]^2$,  which boils down after a few applications of Cauchy-Shwarz,
to the equi-continuity of $N^{-1} \E_\star \big[  \| \bx_t-\bx_{t'}\|^2 \big]$ near the diagonal $t'=t$.  The latter
% equi-continuity 
follows as in \cite[(2.18)]{BDG2},  now utilizing also the differentiability in $s$
of the second line in \eqref{313n-DS21}.  The proof of \cite[Prop.  2.3]{BDG2} similarly adapts to show that 
also here,  for $U_N \in \Ua_N$ and all $\epsilon>0$,
\begin{align}\label{349n-DS21}
\lim_{\delta \to 0} 
\limsup_{N \to \infty} \frac{1}{N} \log \P_\star (\sup_{\|(s',t')-(s,t)\| \le \delta} \{|U_N(s,t)-U_N(s',t')|\} \ge \epsilon) < 0 \,.
\end{align}

Similar reasoning applies for $\qs>0$.  First,  as explained in \cite{DS21},  the tail control \cite[(3.43)]{DS21} 
on $\|\BJ\|^N_\infty$ of \cite[(3.36)]{DS21} holds under $\E_\star$,  our bound \eqref{342n-DS21} 
replaces \cite[(3.42)]{DS21},  and as explained above,  we also have the
bound \cite[(3.37)]{DS21} (restricted to $[0,T]$).  Now,  combining these facts leads to uniform in $N$ 
bounds on the moments $\E_\star [(\|K_N\|_T)^k]$ and $\E_\star [(\|\BJ\|_\infty^N)^k]$ at any $k<\infty$. 
Utilizing these bounds,  one then follows the derivation leading to \cite[(3.49)]{DS21},  to arrive at the
uniform in $N$ moment bounds for $\|U_N\|_T$ under $\E_\star$ and at the equi-continuity,  hence pre-compactness,
of $(s,t) \mapsto \E_\star [U_N(s,t)]$.  Indeed,  see \cite[Proof of Prop.  3.3]{DS21} on why \cite[(3.49)]{DS21} applies 
under the law $\P_\star$ of \cite{DS21},  namely,  when changing only the centered field,  from 
$H_{\BJ}^{\qo}$ to $H_{\BJ}^\star$.  As for the change here due to the new terms
involving $\bv(\cdot)$ and its derivatives,  both on the \abbr{rhs} of \eqref{34n-DS21}-\eqref{36n-DS21} 
and in the second line of \eqref{313n-DS21},  having
the bound $\kappa (1+\|K_N\|_T^m)$ on the norms of the former, 
and  \eqref{eq:drift-ext-bd} on those of the latter,  allows us to mitigate their  
effect on $\E_\star [ \|U_N\|_T^k]$ by merely increasing certain universal constants.  
Similarly,  thanks to its differentiability in $s$,  the second line of \eqref{313n-DS21} does not 
change the equi-continuity of $N^{-1} \E_\star \big[  \| \bx_t-\bx_{t'}\|^2 \big]$ near the diagonal 
$t'=t$,  nor the $e^{-\Omega(N)}$-probability decay of $\sup_{|t-t'| \le \delta} \{ N^{-1} \|\bx_t-\bx_{t'}\|^2\}$
(which are shown by adapting \cite[(2.18)]{BDG2} to our setting,  just as was implicitly done  
in \cite{DS21} en-route to \cite[(3.49]{DS21}).  In addition,  with $\bv \in \Ca^2$,  the modulus of continuity of 
the new terms on the \abbr{rhs} 
of \eqref{34n-DS21}-\eqref{36n-DS21} 
and of their expectation under $\E_\star$ are
controlled by the corresponding quantities for $(q_N(\cdot),C_N(\cdot,0))$.  Upon utilizing this,  
both \eqref{349n-DS21} and the equi-continuity of $\E_\star U_N(s,t)$ for all 
$U_N \in \Ua_N$ follow by the corresponding arguments in \cite{DS21}.

\noindent
(b).  As in \cite{BDG2,DS21},  by utilizing \eqref{349n-DS21} and the union bound,  it suffices to show that 
for any $(s,t) \in [0,T]^2$,  $\rho>0$,
\begin{align}
\label{221n-BDG2}
\limsup_{N \to \infty} \frac{1}{N} \log \P_\star(|U_N(s,t) - \E_\star [U_N(s,t)]| \ge \rho) &< 0 \,.
\end{align} 
To this end,  recall that $\P_\star$ has the Lipschitz sub-Gaussian tails of
\cite[Hyp.  1.1,  $\alpha=2$]{BDG2} and $\P_\star(\La_{N,M}^c) = e^{-\Omega(N)}$
for $\La_{N,M} \subset \Ea_N$ of \cite[(3.45)]{DS21} and some $M=M(T,f,\nu,\qs,\qo)<\infty$
(cf.  \cite[Prop.  3.10,  (3.46),  and Sec.  3.3]{DS21}).  Fixing $(s,t) \in [0,T]^2$,  by part (a)
the function $U_N(s,t)=V_N(\bxo,\BJ,\BB)$ is $L^2$-bounded under $\E_\star$.  In addition,  with  
the second line of \eqref{313n-DS21} and the new terms in \eqref{34n-DS21}-\eqref{36n-DS21}
all bounded by $\kappa (1 + \|K_N\|^m_T)$,  it is not  hard to verify that 
$\sup\{ |V_N(\bxo,\BJ,\BB)|  : N \ge 1,  (\bxo,\BJ,\BB) \in \La_{N,M} \} < \infty$. Thus,  
as in \cite{BDG2,DS21},  we get \eqref{221n-BDG2} upon applying \cite[Lemma 2.5]{BDG2} 
to $V_N$,  once we verify that it has on $\La_{N,M}$ the Lipschitz property of \cite[(3.50)]{DS21}.
The latter property is established similarly to its proof in \cite[Lemma 3.9]{DS21},  which builds
on \cite[Lemma 2.7]{BDG2} (to take care of the subset $\Ua_N^\dagger$ from \cite[(3.2)-(3.3)]{DS21}).
Specifically,  we first adapt the Gronwall argument from the proof of \cite[Lemma 2.6]{BDG2} to the current setting,  
where having $\bv \in \Ca^2$ allows us to bound the contribution of the second line of \eqref{313n-DS21} to 
$e_N(t):=N^{-1} \|\bx_t-\widetilde{\bx}_t\|^2$,  similarly to the treatment in \cite{BDG2} 
of the terms $I_1$,  $I_3$ and $I_4$ in \cite[(2.23)]{BDG2}.  Moreover,  one is to show  
\cite[(3.50)]{DS21} for the functions on the \abbr{lhs} of \eqref{34n-DS21}-\eqref{36n-DS21},  
only after this has already been established for 
%$\Ua_N^\dagger$,  so in particular for 
$C_N$ and $q_N$.  Hence,  
as $\bv \in \Ca^2$,  even with the new terms on the \abbr{rhs} of  \eqref{34n-DS21}-\eqref{36n-DS21},
we still arrive at the desired conclusion (of \cite[(3.50)]{DS21}).
\end{proof}

\red{We have} the following immediate consequence of Prop.  \ref{prop:33-DS21} 
(cf.  the proofs of \cite[Cor. 2.8]{BDG2} and \cite[(3.57)]{DS21}).
\begin{cor}\label{cor:32-DS21}
If $\Psi:\R^\ell \to \R$ is locally Lipschitz with $|\Psi(\bz)| \le M' \|\bz\|^k$ for finite $M',\ell,k$, 
and each coordinate of the vector $\BZ_N \in \R^\ell$ is of the form $\E_\star [U_N(s_j,t_j) |\Fa_{\tau_j}]$ 
for some $U_N \in \Ua_N$ and $(s_j,t_j,\tau_j) \in [0,T]^3$,  then
\begin{equation}\label{312n-DS21}
\lim_{N \to \infty} \sup_{s_j,t_j,\tau_j} \{ |\E_\star [ \Psi(\BZ_N) -\Psi(\E_\star \BZ_N)] | \} = 0 \,.
\end{equation}
\end{cor}
\red{In case $\qs=0$,  step (b) of our program is then the following analog of \cite[Prop. 1.3]{BDG2}.}
\begin{prop}\label{prop:13-BDG2}
Consider Proposition \ref{prop:33-DS21} at $\qs=0$.  Namely,  the modified collection 
$\Ua_N=(C_N,\chi_N,D_N,E_N)$ in terms of \eqref{313n-DS21} with $q_N \equiv 0$.
\red{Upon adding the term $C(s,0)\bv'(C(t,0))$ on the \abbr{rhs} of \cite[(1.19)]{BDG2},}
any $\|\cdot\|_T$-limit point of $\Ua^a_N$ satisfies \cite[(1.17)-(1.20)]{BDG2} subject to the
symmetry and boundary conditions of \cite[Lemma 4.1]{BDG2}.
% and our modification of \cite[(1.19)]{BDG2}.
\end{prop}

\red{For $\qs>0$ we similarly complete step (b) of our program by the following analog of \cite[Prop.  3.5]{DS21}.}
\begin{prop}\label{prop:35-DS21n}
For $\qs>0$,  consider the collection $\Ua_N$ consisting of $(C_N,\chi_N)$ and the
% fifteen 
functions of \cite[(3.2)-(3.8)]{DS21},  modifying \cite[(3.4)-(3.6)]{DS21} as in \eqref{34n-DS21}-\eqref{36n-DS21}. 
\red{Setting} 
\begin{align}\label{320n-DS21}
Q(s) :=&-f'(K(s))q(s)+V(s) + \qs^2 \bv_x(q(s), C(s,0)) + \qo \bv_y (q(s),C(s,0)) \, , \\
\bar D(s,t) :=& - f'(K(t)) \bar C(t,s) + \bar A(t,s)+\bar C(s,0) \bv_y(q(t),C(t,0)) \,,
\label{321n-DS21}
\end{align}
any $\|\cdot\|_T$-limit point of $\Ua^a_N$ 
% :=\{ \E_\star [U_N] : U_N \in \Ua_N\}
satisfies \cite[(3.19)-(3.31)]{DS21},  subject to the symmetry and boundary conditions of \cite[Prop.  3.5]{DS21} 
and the modification \eqref{320n-DS21}-\eqref{321n-DS21} of \cite[(3.20)-(3.21)]{DS21}. 
\end{prop}
By definition $C_N(s,t)=C_N(t,s)$,  $K_N(s)=C_N(s,s)$,  the boundary conditions
$K_N(0)=1$,   $q_N(0)=\qo$ and $E_N(s,0)=0$ hold,
% (as $\BB_0={\bf 0}$),  
with  
$E_N^a(s,t)=E_N^a(s,s)$ for $t \ge s$,  by the independence of $\bx_s$ and $\BB_t-\BB_s$.
These 
% symmetry and boundary 
relations are retained by the expectation and as
% the limit 
$N \to \infty$,  so it suffices to show that up to an $o_N(1)$ error,  uniformly over $[0,T]^2$,  
when $\qs=0$ the functions $(C_N^a,\chi_N^a,D_N^a,E_N^a)$ satisfy the modified \cite[(1.17)-(1.20)]{BDG2},  
and for $\qs>0$ the larger collection $\Ua_N^a$ satisfies the modified 
\cite[(3.19)-(3.31)]{DS21}.  \red{We proceed to do so} by adapting \cite[Sect.  3]{BDG2} to our setting of $\qs=0$ and \cite[Sect. 4.1]{DS21} to that of $\qs>0$,  where for brevity,  we detail only the changes one must make in the various formulas. 
\begin{proof}[Proof of Proposition \ref{prop:13-BDG2}]
Adding
$x_0^i \int_0^s \bv'(C_N(u,0)) du$ to the \abbr{rhs} of \cite[(3.1)]{BDG2}
due to \eqref{313n-DS21},  translates
to $\int_0^s \E_\star [C_N(0,t) \bv'(C_N(u,0))] du$ 
and  $\int_0^s \E_\star [\chi_N(0,t) \bv'(C_N(u,0))] du$ added
on the \abbr{rhs} of \cite[(3.2)]{BDG2} and \cite[(3.3)]{BDG2},  respectively.  Then,  
Corollary \ref{cor:32-DS21} allows us to replace the latter terms,  as $N \to \infty$,   
by $C^a_N(0,t) \int_0^s \bv'(C^a_N(u,0)) du$ 
and  $\chi^a_N(0,t)\int_0^s  \bv'(C^a_N(u,0)) du$,  respectively.   With $\chi_N^a(0,t)=0$ 
(as $(\BB_t,  t\ge 0)$ is independent of $\bxo$),  we deduce from our update of $D_N(s,t)$ 
of \cite[(1.16)]{BDG2} and the preceding modification of \cite[(3.2)-(3.3)]{BDG2},  that 
limit points of $(C^a_N,\chi_N^a,D_N^a,E_N^a)$ must satisfy \cite[(1.17)-(1.18)]{BDG2}.
The proof is thus concluded by verifying that \cite[Prop.  3.1]{BDG2} extends to our setting.

For the latter task,  \cite[Lemmas 3.2-3.4]{BDG2} only change by subtracting 
the drift term $[\int_0^s \bv'(C_N(u,0)) du]  x_0^i$ we have added in \eqref{313n-DS21} 
from the \abbr{rhs} of \cite[(3.11)]{BDG2}.  This change yields the 
subtraction of $C_N(\cdot,u) \bv'(C_N(u,0))du$ from $d_u C_N(u,\cdot)$ on the last 
line of \cite[(3.18)]{BDG2},  the effects of which on the display above \cite[(3.20)]{BDG2} 
are consistent,  when expressing there $\widehat{A}^a_N(\cdot)$ in terms of 
$D^a_N(\cdot)$,  with subtracting $C^a_N(s,0) \bv'(C_N^a(t,0))$ 
from the \abbr{rhs} of \cite[(3.20)]{BDG2}.  By Corollary \ref{cor:32-DS21},  such 
a change in \cite[(3.20)]{BDG2} is consistent with our modification of \cite[(1.16)]{BDG2},  
and thereby we conclude that \cite[(3.4)]{BDG2} still hold in our current setting.  

When proving \cite[(3.5)]{BDG2} we similarly must subtract from $dx_u^i$ and $d_u C_N(s,u)$ 
on the \abbr{rhs} of \cite[(3.22)]{BDG2} the terms $x_0^i \bv'(C_N(u,0)) du$ and 
$C_N(s,0) \bv'(C_N(u,0)) du$,  respectively.  These are translated,  after some algebra,  to having on 
\cite[(3.24)]{BDG2} the additional quantities
\[
- \chi_N(0,t) \int_0^s \nu'(C_N(s,u))\bv'(C_N(u,0)) du  - \int_0^s \nu''(C_N(s,u)) \chi_N(u,t) C_N(s,0) \bv'(C_N(u,0)) du \
:= {\rm I} + {\rm II} \,.
\]
When taking the expected value of \cite[(3.24)]{BDG2},  and in particular,  of the preceding quantities,  one utilizes 
Corollary \ref{cor:32-DS21}.  In particular,  as $\chi_N^a(0,t)=0$,  clearly \red{$\| \E_\star [{\rm I}] \|_T \to 0$ as $N \to \infty$.}
Moreover,  combining $\E_\star [{\rm II}]$ with the expectation of the other two  
terms of \cite[(3.24)]{BDG2} that contain $\nu''(C_N(s,u)) \chi_N(u,t)$,  yields 
by our modification of \cite[(1.16)]{BDG2},  \red{that uniformly over $0 \le t \le s \le T$,   as $N \to \infty$,}
\begin{align*}
& \int_0^s  \nu''(C^a_N(s,u)) \chi^a_N(u,t)
\Big[  f'(K_N^a(u)) C_N^a(s,u) - \widehat{A}_N^a(u,s)  
- C^a_N(s,0) \bv'(C^a_N(u,0)) 
\Big]   du \\
+ & \int_0^s  \nu''(C^a_N(s,u)) \chi^a_N(u,t) D_N^a(s,u) du \red{\;\;  \to 0 \,.}
\end{align*}
This matches the right-most term of \cite[(3.5)]{BDG2},  with the remainder of the proof of \cite[(3.5)]{BDG2} unchanged
\footnote{note the typo in the display above \cite[Lemma 3.4]{BDG2},  where $C^a_N(s,u)$ should be replaced by 
$\chi_N^a(u,t)$}.
\end{proof}
\begin{proof}[Proof of Proposition \ref{prop:35-DS21n}] \red{We adapt \cite[Sec. 4.1]{DS21} to the current setting,  where the right-most term in the stochastic differential system of \cite[(3.13)]{DS21} has been replaced by 
the second line of \eqref{313n-DS21}.  Specifically},  thanks to Corollary \ref{cor:32-DS21},  the limit
identities \cite[(3.19)-(3.26)]{DS21} with \eqref{320n-DS21}-\eqref{321n-DS21} \red{instead of \cite[(3.20)-(3.21)]{DS21}},  follow by considering the $N \to \infty$ limit of the expected values in \cite[(3.2),  (3.5)-(3.8)]{DS21},  
with our modification \eqref{35n-DS21}-\eqref{36n-DS21}
(for the \abbr{rhs} of \cite[(3.19)]{DS21} recall that $\sup_N \E_\star [\|q_N\|_T^2]<\infty$).  
\red{Next,  to show that \cite[(3.27)]{DS21} also holds,  taking
the inner product of \eqref{313n-DS21} at $s=t$ with $N^{-1} (0,x_s^2,\ldots,x_s^N)$,  leads 
to\footnote{\purple{fixing the display above \cite[(4.1)]{DS21},  which erroneously has the integrand $\bar D_N(u,s)$.}}
\begin{equation}\label{41a-DS21}
\bar C_N(s,t) = \bar C_N(s,0) + \bar \chi_N(s,t)  
+ \int_0^t \bar D_N(s,u) d u 
\end{equation}
(by definition $\bar C_N$ of \cite[(3.2)]{DS21}} is unaffected by a drift in direction $\be_1$ 
%in \eqref{313n-DS21}
and \eqref{36n-DS21} incorporated within $\bar D_N(s,u)$ the extra 
term  $\bar C_N(s,0)  \bv_y(q_N(u),C_N(u,0))$ whose integral over 
$[0,t]$ matches the contribution to $\bar C_N(s,t)$ from the drift in direction $\bxo$ in \eqref{313n-DS21}).  
\red{Similarly,  the inner product of \eqref{313n-DS21} and $N^{-1} (0, B_t^2,\ldots,B_t^N)$ yields that   
\begin{equation}\label{41b-DS21}
\bar \chi_N(s,t) = \bar \chi_N(0,t) + \frac{1}{N} \sum_{i=2}^N B_s^i B_t^i  + 
\int_0^s \bar E_N(u,t) d u + \int_0^s \bar \chi_N(0,t) \bv_y(q_N(u),C_N(u,0)) du \,,
\end{equation}
with $\bar E_N$ as in \cite[(3.6)]{DS21}.  We now get the \abbr{lhs} of \cite[(3.27)]{DS21} upon taking 
the $\E_\star$ expectation of \eqref{41a-DS21},  followed by $N \to \infty$.  Further,  operating similarly on 
\eqref{41b-DS21},   yields the \abbr{rhs} of \cite[3.27]{DS21},  since 
$\bar \chi_N^a(0,t)=0$
%(the zero-mean $\BB_t$ is independent of $\bx_0$),  
(so by Corollary \ref{cor:32-DS21} the contribution of the new,  right-most term in \eqref{41b-DS21} is negligible).}
Having confirmed that \cite[(3.27)]{DS21} holds,  and with \cite[(3.28)-(3.31)]{DS21} merely rephrasing 
\cite[(4.3)-(4.6)]{DS21},  we complete the proof by verifying the latter four relations.   
\red{Moreover,  since} $H_{\BJ}^\star$ shares the covariance kernel of \cite[(3.34)]{DS21},  
we can do so with minimal changes to \red{the proof of  \cite[(3.28)-(3.31)]{DS21}}.
Specifically,  \red{replacing} the right-most term of 
\cite[(3.13)]{DS21} \red{by} those of \eqref{313n-DS21} translates to the corresponding 
differences in the formulas for $Q^i_{s;\tau}$ and $U_s^i$ at 
\cite[\abbr{rhs} of (4.9)]{DS21} and \cite[top line of (4.14)]{DS21},  respectively.  Nevertheless,  
these changes are consistent with our modifications \eqref{35n-DS21}-\eqref{36n-DS21} 
of $Q_N$ and $\bar D_N$,  so the relations of \cite[(4.17)]{DS21} are retained here.  \red{Thereafter, } 
beyond \cite[Lemma 4.2]{DS21} the proof of \cite[Prop.  4.1]{DS21} proceeds with no further changes.
\end{proof}

\red{Step (c) of our program}
%whose proof we defer to Section \ref{sec:cont-lim}},  
embeds $\bUqf(\V)$ within the enlarged limit dynamics of Prop. \ref{prop:13-BDG2}-\ref{prop:35-DS21n}.\\
\red{Specifically,  for $\qs=0$ we have the following replacement of \cite[Lemma 4.1]{BDG2}}.
 \begin{prop}\label{lem:41-BDG}
If $(C,\chi,D,E) \in \Ca_b([0,T]^2,\R^4)$ satisfies \cite[(1.17)-(1.20)]{BDG2} subject to the
symmetry and boundary conditions of \cite[Lemma 4.1]{BDG2} and our modification of \cite[(1.19)]{BDG2}, 
then $(C,\chi,H)=\bUof(E)$ of Definition \ref{def:limit-dyn-high-temp} at $\b=1$.
\end{prop}
\begin{proof}
%[Proof of Proposition \ref{lem:41-BDG}] 
Up to \cite[(4.2)]{BDG2},  the proof of \cite[Lemma 4.1]{BDG2}
applies as is.  On the \abbr{rhs} of that formula we now have the extra term $C(s,0)\bv'(C(t,0))$,
resulting with having here \eqref{eqCs-new} and \eqref{def:ACs} at $\b=1$ and $\mu(s)=f'(K(s))$.  
Continuing with that proof,  but 
with $C(s,0)\bv'(C(s,0))$ now being added to $D(s,s)$ (due to \cite[(1.19)]{BDG2}),  results
with \eqref{eqfZs} holding (at $\b=1$).  The remainder of the proof of 
\cite[Lemma 4.1]{BDG2},  led to \cite[(1.12)]{BDG2},  which being the same as \eqref{eqRs-new},
requires no change.
\end{proof}
\red{Similarly,  for $\qs>0$ we have,  as our step (c),  the following analog of \cite[Prop.  3.6]{DS21}}.
\begin{prop}\label{prop:36-DS21n}
Suppose $\qs>0$ and $\bU=(C,\chi,q,\hat{H}) \in \Ca_b([0,T]^2,\R^4)$ satisfies \cite[(3.19)-(3.31)]{DS21},  
subject to the symmetry and boundary conditions of \cite[Prop.  3.5]{DS21} and the  
modification \eqref{320n-DS21}-\eqref{321n-DS21} of \cite[(3.20)-(3.21)]{DS21}.  Then,  
$\bU=\bUqf(\V)$ is the $f$-dynamic of Definition \ref{def:limit-dyn-low-temp} at $\b=1$. 
\end{prop}
\begin{proof}
%[Proof of Proposition \ref{prop:36-DS21n}]
Apart from \eqref{34n-DS21} converting \cite[(1.21)]{DS21} into \eqref{eqH-new}, 
the proof of \cite[Prop.  3.6]{DS21} applies as is,  up to the \abbr{rhs} of \cite[(4.37)]{DS21} 
where \eqref{321n-DS21} results with the extra term 
$\bar{C}(s,0) \bv_y(q(t),C(t,0))$.  Continuing as in \cite{DS21},  we get the 
same explicit expression for $V(s)$ of \cite[(3.28)]{DS21},  which together 
with \eqref{320n-DS21} leads to \eqref{eqqs} with $\mu(s)=f'(K(s))$ and $\widehat{{\sf A}}_q$ of \eqref{def:Aq-new},
at $\b=1$.  Likewise,  substituting  
\eqref{320n-DS21}-\eqref{321n-DS21} in \cite[(4.35)]{DS21} 
and combining it with the revised \abbr{rhs} of \cite[(4.37)]{DS21},  converts
the right-most term on \cite[(4.38)]{DS21} to $q(s)\bv_x(q(t),C(t,0))+C(s,0) \bv_y(q(t),C(t,0))$,
which after exchanging $s$ with $t$ results in \eqref{eqCs-new} for  $\widehat{{\sf A}}_C$
of \eqref{def:AC-new} (also at $\b=1$).  Similarly,  with 
$D(s,s) = q(s) Q(s)/\qs^2+\bar D(s,s)$,  
substituting \eqref{320n-DS21}-\eqref{321n-DS21} on the \abbr{rhs} of \cite[(4.39)]{DS21} verifies
that \eqref{eqfZs} holds with that same $\widehat{{\sf A}}_C$
(and $\b=1$).  \red{As with our proof of Prop.  \ref{lem:41-BDG}},  the last part of the proof 
of \cite[Prop.  3.6]{DS21},  leads,  as is,  to \eqref{eqRs-new} (also at $\b=1$).
\end{proof}

\red{Step (d) of our program,  namely} the uniqueness of the $f$-dynamics of Definitions 
\ref{def:limit-dyn-low-temp} and \ref{def:limit-dyn-high-temp},  is established 
in Proposition \ref{uniqueness}(a).  \red{Recall Prop.  \ref{prop:13-BDG2}-\ref{prop:36-DS21n} that}
any $\|\cdot\|_T$-limit 
point of $\E_\star(C_N,\chi_N,q_N,\HN)$ is of the form of our $f$-dynamics,
\red{which by Prop.  \ref{prop:33-DS21} completes our} 
proof of Proposition \ref{prop-macro-easy} (cf.  \cite[Proof of Thm.  1.1]{DS21}).

\section{Limiting dynamics: uniqueness and continuity}\label{sec:cont-lim}

We first establish the uniqueness of the dynamics $\bUqf(\V)$ and $\bUsp(\V)$
from Definitions \ref{def:limit-dyn-low-temp} and \ref{def:limit-dyn-high-temp}.
%(which a minor extension of \cite[Prop.  3.4]{DS21}), 
We then prove Prop.  \ref{prop:nu-cont} and  \ref{prop:ell-lim} about the 
continuity of $\bUqf(\V;\nu)$ \abbr{wrt} $\nu(\cdot)$ and $f(\cdot)$,  respectively,
and conclude by proving  \red{Propositions \ref{lem:41-BDG} and \ref{prop:36-DS21n}} 
(that embed $\bUqf(\V)$ in the limit dynamics \red{for the relevant collection $\Ua_N^a$}).

\begin{prop}\label{uniqueness}
Let $T<\infty$,  $\D_T=\{s,t\in (\R^+)^2 : 0\le t\le s\le T\}$ and $1 < r < \bar r$.\\
(a).  For $\qs \in (0,1]$ and $\sfv(\cdot,\cdot)$,  $f(\cdot)$ 
of uniformly Lipschitz derivatives on $([-r^2,r^2])^2$ and $[0,r^2]$,  respectively,  
there exists at most one solution $(R,C,q,K) \in \Ca_b^1( \D_T)^2 \ts \Ca_b^1([0,T])^2$ 
to \eqref{eqRs-new}--\eqref{eqqs},  \eqref{def:AC-new}-\eqref{eqL},  \eqref{eqfZs} 
with $C(s,t)=C(t,s)$,  stopped at $\tau_\star:=\inf\{s \ge 0: K(s) \ge r^2\}$,  and having the boundary conditions
\begin{align}\label{bc-uniq}
q(0)=\qo,  \qquad R(s,s)\equiv 1,  \qquad C(s,s)=K(s)\, ,  \quad\forall s\ge 0 \,.
\end{align}
The same applies for $(R,C,K)$ of \eqref{eqRs-new}-\eqref{eqCs-new},  \eqref{eqfZs}-\eqref{def:ACs} 
with the boundary conditions
of \eqref{bc-uniq}.\\
(b).  Similarly,  for $\qs \in (0,1]$ there is at most one solution $(R,C,q) \in \Ca_b^1(\Delta_T)^2 \times \Ca_b^1([0,T])$ 
to \eqref{eqRs-new}-\eqref{eqqs},  \eqref{eqZs-new}-\eqref{eqL},   
with $C(s,t)=C(t,s) \in [-1,1]$ and boundary conditions \eqref{bc-uniq} with $K(s) \equiv 1$.
This also applies for $(R,C)$ of \eqref{eqRs-new}-\eqref{eqCs-new},  \eqref{eqZs-new} and \eqref{def:ACs},
with $C(s,t)=C(t,s) \in [-1,1]$ and $R(s,s) \equiv C(s,s) \equiv 1$.
\end{prop}
\begin{proof}$~$\\
(a).  The case $\qs>0$ has the system of equations and boundary conditions of
 \cite[Prop. 3.4]{DS21} apart from having now $\nu''(\cdot)$ and $f'(\cdot)$ Lipschitz only on 
$[-r^2,r^2]$ and $[0,r^2]$,  while replacing the functions $q(t) \bv_\star'(q(s))$ and $\qs^2 \bv_\star'(q(s))$
for $\bv(\cdot)$ of \cite[(1.22)]{DS21},  by
$q(t) \bv_x(\cdot) + C(t,0)\bv_y(\cdot)$ and $\qs^2 \bv_x(\cdot)+\qo \bv_y(\cdot)$,  respectively, 
at $(q(s),C(s,0))$.  Nevertheless,  the same Gronwall's lemma reasoning as in \cite[Sec.  3.3]{DS21}
applies here as well,  provided the arguments $\{q(u), C(u,v)$,  for $ u,v \le s\}$ of 
$\nu(\cdot)$,  $\nu'(\cdot)$, $\nu''(\cdot)$,  $\bv_x(\cdot)$ and $\bv_y(\cdot)$ are all 
in $[-r^2,r^2]$ with the arguments of $f'(\cdot)$ in $[0,r^2]$.
It is easy to see that this must hold up to $\tau_\star$ (thanks to the non-negative 
definiteness of $C(s,t)-q(s) q(t)/\qs^2$,  see Remark \ref{rem:non-neg}),  \red{which thereby yields the 
uniqueness of the solution  up to $\tau_\star$.  Since any solution $(R,C,K)$ of} \eqref{eqRs-new}-\eqref{eqCs-new} and \eqref{eqfZs}-\eqref{def:ACs},  \red{together with $q(s) \equiv 0$,  also satisfies  
\eqref{eqqs} and  \eqref{def:AC-new}-\eqref{eqL}
%for $\qs>0$,  
in case $\bv(x,y)=\bv(0,y)$ and $\qo=0$,  
we have thus established 
the uniqueness of such $(R,C,K)$}.
%The same applies for $\qs=0$,  where we have the system of equations and boundary 
%conditions of \cite[Prop.  4.2]{BDG2},  apart from having now also the extra term $\b C(t,0) \bv'(C(s,0))$.
\\
(b).  Having $K(s) \equiv 1$ confines,  in view of Remark \ref{rem:non-neg},  all 
arguments of $\nu(\cdot)$,  $\nu'(\cdot)$, $\nu''(\cdot)$,  $\bv_x(\cdot)$ and $\bv_y(\cdot)$ to be in 
$[-1,1]$,  and those of $f'(\cdot)$ to $[0,1]$,  where these functions are globally Lipschitz.  In case $\qs>0$
our equations for $(R,C,q)$ are those of \cite[(1.30)-(1.33)]{DS21} apart from modifying
their functions $q(t) \bv_\star'(q(s))$ and $\qs^2 \bv_\star'(q(s))$,  precisely as in part (a).
%for $\bv(\cdot)$ of \cite[(1.22)]{DS21},  
%by $q(t) \bv_x(\cdot) + C(t,0)\bv_y(\cdot)$ and $\qs^2 \bv_x(\cdot)+\qo \bv_y(\cdot)$,  respectively, 
%at $(q(s),C(s,0))$. 
The proof of uniqueness of their solution,  on \cite[Page 509]{DS21},  requires only
that the latter functions be Lipschitz on $[-1,1]$,  and it applies here verbatim,  even after our modification as in part (a).
Next,  note that for $\bv(x,y)=\bv(0,y)$ and $\qo=0$,  combining any solution $(R,C)$ 
of \eqref{eqRs-new}-\eqref{eqCs-new},  \eqref{eqZs-new} and \eqref{def:ACs},   with $q(s) \equiv 0$,  yields a solution of
\eqref{eqRs-new}-\eqref{eqqs} and \eqref{eqZs-new}-\eqref{eqL}.  Uniqueness of the former solution 
is thus a direct consequence of the preceding proof (for $\qs>0$).
\end{proof}

\begin{proof}[Proof of Proposition \ref{prop:nu-cont}] \purple{Equip $\Ca^k(([-r_1^2,r_1^2])^d)$ for $d=1,2$}, 
with a norm $\|\cdot\|^{(k)}_\star$ of uniform convergence of the function and its first $k$ partial derivatives.  
With $\nu(\cdot)$ real analytic on $(-\bar r^2,\bar r^2)$,  by Definition \ref{def:H-m-coupling}
\begin{equation}\label{eq:nu-m-cnv}
\|\nu^{[m]}\|^{(3)}_\star \uparrow \|\nu\|^{(3)}_\star < \infty \quad \hbox{and} \quad  \| \nu^{[m]} - \nu \|^{(3)}_\star \to 0 \,.
\end{equation}
We assume \abbr{wlog} that $\nu(\cdot)$ is an infinite mixture
and that $\nu^{[m]}(\cdot)$ which are not pure $p$-spins,  have the same even/odd symmetry as $\nu(\cdot)$.  
In case $\qs>0$ and $|\qo|<1$,  set $\bv^{[m]}(x,y)$ by \eqref{def:vt-new}-\eqref{def:w-s} 
for the mixture $\nu^{[m]}$ and the prescribed $\V$.  It then follows 
from Lemma \ref{lem:app}(a) that $\Sigma^{-1}_{\nu^{[m]}}(\qo) \to \Sigma^{-1}_{\nu}(\qo)$,  which implies 
in turn that $\| \bv^{[m]} - \bv\|_\star^{(2)} \to 0$.  The latter applies also for $\qs=0$,  where
$\bv^{[m]}(y):=E \nu^{[m]}(y)/\nu^{[m]}(1)$.  Finally,  if $|\qo|=\qs=1$ then  
$\bxs=\qo \bxo$,  $q(\cdot) = \qo C(\cdot,0)$ and we accommodated the possible rank one drop in \eqref{def:w-s} 
%when $\qo=1$ or when $\qo=-1$ with $\nu(\cdot)$ even or odd, 
% it is full rank when \qo=-1 and \nu(.) of no even/odd symmetry
by requiring that $\Ep=\Eg$ if $\qo=1$ and $\Ep=\pm \Eg$ if
$\qo=-1$ with $\nu(\cdot)$ even or odd (see Footnote \ref{ft:degen}).  In such cases 
$w_4=0$ (see Remark \ref{rem:w-unique}),  and \abbr{wlog} $w_2 = \pm w_1$.  Thus,   
solving the reduced system in \eqref{def:w-s} for $(w_1,w_3)$ and replacing \eqref{def:vt-new} by
$\bv(y)=w_1 \nu(y) \pm w_3 y \nu'(y)$,  extends to all $|\qo| \le \qs$ the convergence 
\begin{equation}\label{eq:bv-m-cnv}
\|\bv^{[m]}-\bv\|^{(2)}_\star \to 0\,,   \quad \hbox{with} \quad \limsup_m \|\bv^{[m]}\|^{(2)}_\star < \infty  \,.
\end{equation} 
Next,  integrating \eqref{eqRs-new}-\eqref{eqqs} and \eqref{eqfZs} over $s \ge t$,  we get
from Definition \ref{def:limit-dyn-low-temp},  similarly to \cite[Pg.  508]{DS21},  
that $(R,C,q)$ of $\bUqf(\V;\nu)$ are the unique solution 
in $\Ca(\Delta_T,\R^3)$ of the equations 
\begin{align}\label{eqRs-m}
R(s,t) &= 1 + \int_t^s {\sf A}_R [R,C,q;\bv,\nu] (\theta,t) d\theta\,,  \qquad \qquad \qquad  \qquad \qquad \qquad \qquad 
\qquad \qquad \forall s \ge t \ge 0 \,,\\
\label{eqCs-m}
C(s,t) &= 1 + \int_0^t \big\{1 + 2 {\sf A}_C [R,C,q;\bv,\nu] (\theta,\theta) \big\} d\theta
 + \int_t^s {\sf A}_C [R,C,q;\bv,\nu] (\theta,t) d\theta\,,  \qquad \forall s \ge t \ge 0  \,, \\
\label{eqqs-m} 
q (s) &=  \qo + \int_0^s {\sf A}_q [R,C,q;\bv,\nu] (\theta) d\theta\,,  \qquad \qquad \qquad \qquad \qquad \qquad \qquad \qquad\qquad \qquad  \forall s \ge 0 \,,
\end{align}
where setting $\mu(\theta):=f'(C(\theta,\theta))$,  we have in view of \eqref{eqRs-new} and \eqref{def:AC-new}-\eqref{eqL},
that at any $\theta \ge t \ge 0$, 
\begin{align}
{\sf A}_R [R,C,q;\bv, \nu](\theta,t) 
&= - \mu(\theta) R(\theta,t) + \b^2 \int_t^\theta R(u,t) R(\theta,u) \nu''(C(\theta,u)) du ,
 \label{def:AR}
\\
{\sf A}_C [R,C,q;\bv,\nu](\theta,t) 
&= - \mu(\theta) C(\theta,t) + \b^2 \int_0^\theta C(u , t) R(\theta,u) \nu''(C(\theta,u)) \, du + \b^2 \int_0^t R(t,u) \nu'(C(\theta,u)) \, du  \nonumber \\
-  \b^2 q(t) & \nu''(q(\theta)) L(\theta) - \b^2 \nu'(q(\theta)) L(t) + \b q(t) \bv_x (q(\theta),C(\theta,0)) + \b C(t,0) \bv_y (q(\theta),C(\theta,0)) \,,
\label{def:AC}
\\
{\sf A}_q [R,C,q;\bv,\nu](\theta) 
&= - \mu(\theta) q(\theta) + \b^2 \int_0^\theta R(\theta,u) q(u) \nu''(C(\theta,u)) du 
\nonumber \\ &- \b^2 \qs^2 \nu''(q(\theta)) L(\theta)
                + \b [\qs^2 \bv_x (q(\theta),C(\theta,0)) + \qo \bv_y(q(\theta),C(\theta,0))] \,,
\label{def:Aq} 
\end{align}
% as in Definition \ref{def:limit-dyn-low-temp}, 
and $L(\theta) =  \frac{1}{\nu'(\qs^2)} \int_0^\theta R(\theta,u) \nu'(q(u)) dv$ if $\qs>0$,  while 
% of \eqref{eqL} 
$L \equiv q \equiv 0$ if $\qs=0$ (see Definition \ref{def:limit-dyn-high-temp}).  Further,  
\begin{equation}\label{eqH-F4}
H(s) = {\sf A}_{H}[R,C,q;\bv,\nu] (s) := \b \int_0^s R(s,u) \nu'(C(s,u)) \, du  -  \b \nu'(q(s)) L(s) + \bv (q(s),C(s,0)) \,.
\end{equation}
In particular,  $(R^{[m]},C^{[m]},q^{[m]})$ of $\bUqf(\V;\nu^{[m]})$ satisfy \eqref{eqRs-m}-\eqref{eqqs-m}
at $(\bv,\nu)=(\bv^{[m]}, \nu^{[m]})$.  Setting $\mu^{[m]}:=f'(K^{[m]})$ for 
$K^{[m]}(s):=C^{[m]}(s,s)$, 
by Remark \ref{rem:non-neg} 
% $|C^{[m]}(s,t)| \le \sqrt{K^{[m]}(s) K^{[m]}(t)}$ and $|q^{[m]}(s)| \le \qs \sqrt{K^{[m]}(s)}$. 
and the \abbr{lhs} of \eqref{eq:m-cont-Uf},  for some $m_o$ finite and all $m \ge m_o$, 
\begin{equation}\label{eq:C-q-m-ubd}
\|q^{[m]}\|_T \le \qs \| K^{[m]} \|_T^{\frac{1}{2}} \le r_1\,,  \quad 
\|C^{[m]}\|_T = \|K^{[m]}\|_T \le r_1^2 \quad \hbox{and}  \quad \|\mu^{[m]}\|_T \le \|f'\|_\star \,.
\end{equation}
With $\|\nu''^{[m]}\|_\star \le \|\nu''\|_\star < \infty$,  we further get from 
\eqref{eq:C-q-m-ubd} and \eqref{eqRs-new} at $\mu(\cdot)=\mu^{[m]}(\cdot)$ that 
\begin{equation}\label{eq:R-m-ubd}
\sup_{m \ge m_o} \{ \|R^{[m]}\|_T \} \le \kappa
\end{equation}
for $\kappa=\kappa(\|f'\|_\star,\b^2 \|\nu''\|_\star,T)$ finite
(see \cite[(2.2)]{DGM}).  Namely,  $\{ (\mu^{[m]},R^{[m]},C^{[m]},q^{[m]}),  m \ge m_o\}$ is contained in 
%the $\|\cdot\|_T$-closed set
\begin{equation}\label{def:La-kap}
\La_{\kappa,r_1} := \{(\mu,R,C,q) \in \Ca(\Delta_T,\R^4)  : \|\mu\|_T \le \kappa,   \|R\|_T \le \kappa,  \|C\|_T \le r_1^2,  \|q\|_T \le r_1\} \,.
\end{equation}
For any $(\mu,R,C,q) \in \La_{\kappa,r_1}$ 
% In view of \eqref{eq:C-q-m-ubd},  when considering $F^{[m]}_i$,  $i \le 4$,  
all arguments of the functions $(\bv, \bv_x, \bv_y, \nu',\nu'')$ that appear in $\underline{{\sf A}}:=({\sf A}_R,{\sf A}_C,{\sf A}_q,
{\sf A}_H)$ 
are within the interval $[-r_1^2,r_1^2]$ on which these functions are globally Lipschitz.  Let $\underline{{\sf A}}^{[m]}$
denote the functional $\underline{{\sf A}} \in \Ca(\Delta_T,\R^4)$ evaluated at $(\mu^{[m]}, R^{[m]},C^{[m]},q^{[m]})$ 
and $(\bv^{[m]},\nu^{[m]})$.  With $\|\bv^{[m]}\|^{(1)}_\star$ and $\|\nu^{[m]}\|_\star^{(2)}$ bounded,  we thus 
deduce from \eqref{eq:C-q-m-ubd}-\eqref{eq:R-m-ubd} that 
$\{  \|\underline{\sf A}^{[m]}\|_T,  m \ge m_o\}$ is bounded.  As $f'(\cdot)$ is Lipschitz on $[0,r_1^2]$,  in view of
\eqref{eqRs-m}-\eqref{eqqs-m},   the collection 
$\{(\mu^{[m]},R^{[m]},C^{[m]},q^{[m]}),  m \ge m_o \} 
%\subset \La_{\kappa,r_1}
$
is equi-continuous and uniformly bounded.  By Arzel\`a-Ascoli,   $\{(\mu^{[m]},R^{[m]},C^{[m]},q^{[m]})\}$ is 
pre-compact in $\La_{\kappa,r_1}$.  Fixing a $\|\cdot\|_T$-limit point of this sequence,  
say $(\mu,R,C,q) \in \La_{\kappa,r_1}$,  clearly $\mu=f'(K)$ and  
% we further have that 
for some $M=M(\kappa,r_1,T,\|\nu\|^{(3)}_\star, \|\bv\|_\star^{(2)})$ finite
% any $(R,C,q) \in \La_{\kappa,r_1}$ 
and all $m \ge m_o$,
\[
\|\underline{\sf A}^{[m]}-  \underline{\sf A}[\mu,R,C,q;\bv,\nu] \|_T \le M \Big[ \|
(\mu^{[m]},R^{[m]},C^{[m]},q^{[m]})-(\mu,R,C,q)\|_T 
+ \| \bv^{[m]} - \bv\|^{(1)}_\star + \|\nu^{[m]}-\nu\|^{(2)}_\star \Big] \,.
\]
Fixing a sub-sequence $\{m_\ell\}$ along which the first term on the \abbr{rhs} goes to zero, 
we deduce from \eqref{eq:nu-m-cnv}-\eqref{eq:bv-m-cnv} that 
$\underline{\sf A}^{[m_\ell]} \to \underline{\sf A}[\mu,R,C,q;\bv,\nu]$ in $(\Ca(\Delta_T,\R^4),\|\cdot\|_T)$.  
Consequently,  $(R,C,q) \in \Ca(\Delta_T,\R^3)$ satisfies \eqref{eqRs-m}-\eqref{eqqs-m} with $\mu(s)=f'(C(s,s))$.
By Prop. \ref{uniqueness}(a) it coincides with $(R,C,q)$ of $\bUqf(\V;\nu)$.
From \eqref{eqH-F4} and the preceding,  also 
${\sf A}_H^{[m]}=H^{[m]}$ of $\bUqf(\V;\nu^{[m]})$ 
converges when $m \to \infty$ to ${\sf A}_H[R,C,q; \bv,\nu]=H$ of $\bUqf(\V;\nu)$,  as claimed.
\end{proof}

\begin{proof}[Proof of Proposition \ref{prop:ell-lim}] We adapt the proof of \cite[Prop. 1.6]{DS21} to \purple{our} setting,
with $(C^{(\ell)},R^{(\ell)},q^{(\ell)},H^{(\ell)})$ denoting the corresponding components of $\bU_{\qs}^{f_\ell}(\V)$.
Indeed,  recall Prop. \ref{prop:nu-cont} and Remark \ref{rmk:Uf-bd} that 
%the limiting dynamic 
$\bU_{\qs}^{f_\ell}(\V)$ exists and is the $\|\cdot\|_T$-limit
of $\bU_{\qs}^{f_\ell}(\V;\nu^{[m]})$,  as $m \to \infty$.
  In view of Prop.  \ref{prop-macro-easy} and 
our reduction of \eqref{eq:m-apx2} to \eqref{eq:to-H-o},  
the dynamic $\bU_{\qs}^{f_\ell}(\V)$ is thus the $\|\cdot\|_T$-limit 
in probability,  as $N \to \infty$ and then $m \to \infty$,  of the functions 
$(C^{[m]}_N,\chi^{[m]}_N,q^{[m]}_N,H^{[m]}_N)$ induced by the solution $\bx_t^{[m]}$ of 
\eqref{diffusion-phi} for $f=f_\ell$ and $\varphi_N=H^{\qo,[m]}_{\BJ}+\bar H^{[m]}_{\V}$.  
Further,  as \eqref{eq:exp-conf} holds for $\tau_{c/9}^{[m]}$,  some $c<\infty$ and 
all large $m,\ell$,  it follows that  
$\P( \| K^{[m]}_N -1\|_T \ge \frac{c}{3\ell}) \to 0$ as $N \to \infty$ for any $T<\infty$,  
out of which we deduce that for some $\ell_o<\infty$,  analogously to \cite[Lemma 6.1]{DS21}, 
\begin{equation}\label{62n-DS21}
\sup_{s \ge 0} \{ \big|K^{(\ell)}(s) -1 \big| \} \le \frac{c}{2\ell} \le r_1 - 1 \,,  \qquad \forall \ell \ge \ell_o \,.
\end{equation} 
Recall as in our proof of Prop.  \ref{prop:nu-cont},  that 
$(C^{(\ell)},R^{(\ell)},q^{(\ell)},H^{(\ell)})$ satisfy \eqref{eqRs-m}-\eqref{eqH-F4} 
with $\mu^{(\ell)}(s)=f_\ell'(K^{(\ell)}(s))$,  whereas $\bUsp(\V)$ 
corresponds to the same equations,  for $\mu(s)$ which is determined by the identity 
\begin{equation}
{\sf A}_K[\mu,R,C,q](s):= 
1 + 2 {\sf A}_C [\mu,R,C,q](s,s) = 0 \,, \qquad \forall s \ge 0 
\label{def:DS6.2}
\end{equation}
(where to simplify our notations,  hereafter the dependence of various quantities on $\bv$ and $\nu$ is implicit).
Armed with the bound \eqref{62n-DS21},  we proceed,  analogously to \cite[Lemma 6.2]{DS21},  to show the equi-continuity
of $(\mu^{(\ell)},R^{(\ell)},C^{(\ell)},q^{(\ell)})$ (and thereby the existence of $\|\cdot\|_T$-limit points of this collection).
To this end,  note that in view of \eqref{eqRs-m}-\eqref{eqqs-m},  the
functions $\mu_o(s):=\frac{1}{2} + \b \, \widehat{{\sf A}}_C(s,s)$ with 
$\widehat{{\sf A}}_C(s,s)$ from \eqref{def:AC-new} or \eqref{def:ACs},  are
of the form 
\begin{equation}\label{def:Ao}
\mu_o(s)=\mu_o(0)+\int_0^s {\sf A}_o[\mu,R,C,q](\theta) d\theta \,,
\end{equation}
for some continuous function ${\sf A}_o(s)$ with $\|{\sf A}_o[\mu,R,C,q]\|_T$ bounded,  uniformly over
$
%(\mu,R,C,q) \in 
\La_{\kappa,r_1}$ of \eqref{def:La-kap}.
\begin{comment}
If $\qs>0$,  
\begin{align}
 {\sf A}_o& [\mu,R,C,q] (\theta) := \b^2 \psi(C(\theta,\theta)) + \b^2 \int_0^\theta {\sf A}_R(\theta,u) \psi(C(\theta,u)) du 
+ \b^2 \int_0^\theta R(\theta,u) \psi'(C(\theta,u)) {\sf A}_C(\theta,u)) du \nonumber \\
 & - \b^2 \psi'(q(\theta)) L(\theta) {\sf A}_q (\theta)  - \b^2 \psi(q(\theta)) {\sf A}_L(\theta)
 + \b \bv_1(q(\theta),C(\theta,0)) {\sf A}_q (\theta) 
+ \b \bv_2(q(\theta),C(\theta,0)) {\sf A}_C(\theta,0) \,,
%\\
%\mu_o(s) := &
%\frac{1}{2} + \b^2 \int_0^s R(s,u) \psi(C(s,u))  \, du 
%\nonumber \\& 
%- \b^2 \psi(q(s)) L(s) + \b q(s) \bv_x (q(s),C_o(s)) + \b C_o(s) \bv_y(q(s),C_o(s)) \,, 
\label{eqZs-der} 
\end{align}
for $\bv_1 := \partial_x (x \bv_x+y \bv_y)$,  $\bv_2:=\partial_y (x \bv_x + y \bv_y)$,  
%$\bv_1=\bv_x + x \bv_{xx} + y \bv_{xy}$ and 
%$\bv_2=\bv_y + x \bv_{xy} + y \bv_{yy}$
$L(\cdot)$ of \eqref{eqL} and
\begin{align*}
{\sf A}_L(\theta) := \frac{\nu'(q(\theta))}{\nu'(\qs^2)}  + \frac{1}{\nu'(\qs^2)} \int_0^\theta {\sf A}_R(\theta,u) \nu'(q(u)) \, du \,,
%\label{eqL-der}
\end{align*}
while for $\qs=0$,  set
$q \equiv 0$,  ${\sf A}_q \equiv 0$ and ${\sf A}_L \equiv 0$ in \eqref{eqZs-der}.  
\end{comment}
Next,  as in our derivation of \eqref{eq:C-q-m-ubd}-\eqref{eq:R-m-ubd},  we get from \eqref{eq:fdef} 
and \eqref{62n-DS21} that for $\kappa=\kappa(c+\|f'_0\|_\star,\b^2\|\nu''\|_\star,T)$ finite 
\begin{equation}\label{eq:RCq-ell-ubd}
\|q^{(\ell)}\|_T  \le r_1\,,  \quad 
\|C^{(\ell)}\|_T \le r_1^2\,, \quad \|\mu^{(\ell)}\|_T \le c + \|f'_0\|_\star 
\quad \hbox{and} \quad \|R^{(\ell)}\|_T  \le \kappa \,,  \qquad \qquad \forall \ell \ge \ell_o \,.
\end{equation}
Namely,  $\{ (\mu^{(\ell)},R^{(\ell)},C^{(\ell)},q^{(\ell)}),  \ell \ge \ell_o\}$ is contained in 
the $\|\cdot\|_T$-closed set $\La_{\kappa,r_1}$.   
% of \eqref{def:La-kap}.  
Evaluating the functional 
$\underline{\widehat{\sf A}} =({\sf A}_R, {\sf A}_C, {\sf A}_q,{\sf A}_H,{\sf A}_K, {\sf A}_o) \in \Ca(\Delta_T,\R^6)$ 
at $(\mu^{(\ell)}, R^{(\ell)},C^{(\ell)},q^{(\ell)})$ 
thus results with $\{ \underline{\widehat{\sf A}}^{(\ell)},  \ell \ge \ell_o\}$ of uniformly bounded $\|\cdot\|_T$ 
and thereby with equi-continuous and uniformly bounded $\{(R^{(\ell)},C^{(\ell)},q^{(\ell)}),  \ell \ge \ell_o \}$.
Turning to the equi-continuity of $\{\mu^{(\ell)}\}$,  we set $g_\ell(x):=2 \ell(x-1)^2+ x f'_0(x)$,  plug  
\eqref{eq:fdef} into \eqref{eqfZs},  then 
utilize \eqref{def:Ao},  the identity 
$K'(s) = {\sf A}_K(s)$ (see \eqref{eqCs-m}),   
% \eqref{def:DS6.2}),  
and 
having $g_\ell(1)=f'_0(1)=\mu_o(0)$ (see \eqref{eq:f0-cnd}),  to deduce that $K^{(\ell)}$ satisfies the \abbr{ode}
\begin{align}
 K'(s) + 4 \ell (K(s)-1) &= 2 [\mu_o(s) -  g_\ell(K(s))] 
= 
%2 [\mu_0(0) -g_\ell(1)] + 
2 \int_0^s \big[ {\sf A}_o(\theta) - g'_\ell(K(\theta)) {\sf A}_K(\theta) \big] d\theta \,,
 \label{eqfZs-der} 
\end{align}
\red{with $({\sf A}_o,{\sf A}_K)$ at $(\mu^{(\ell)},R^{(\ell)},C^{(\ell)},q^{(\ell)})$.  The
\abbr{rhs} of \eqref{eqfZs-der} is thus $\int_0^s \Gamma^{(\ell)}(\theta) d\theta$ for
$\Gamma^{(\ell)}:= 2 [ {\sf A}^{(\ell)}_o - g'_\ell(K^{(\ell)}) {\sf A}^{(\ell)}_K]$ and with}
$K^{(\ell)}(0)=1$ we arrive at 
\begin{equation}\label{eq:bd-Kder}
{\sf A}_K^{(\ell)} (s) = \partial_s K^{(\ell)}(s) = \int_0^s e^{-4 \ell(s-\theta)} \Gamma^{(\ell)}(\theta) d\theta \quad 
\Longrightarrow
%\hbox{and} \quad  \hbox{thus} 
\quad 
\|{\sf A}_K^{(\ell)}\|_T \le \frac{1}{4\ell} \|\Gamma^{(\ell)}\|_T \,.
\end{equation}
Now,  from \eqref{62n-DS21} we have that 
\begin{equation}\label{eq:h-def}
\sup_{\ell \ge \ell_o} \{ \|g'_\ell(K^{(\ell)})\|_T \} < \infty\, .
\end{equation}
Combining \eqref{eqfZs-der}-\eqref{eq:h-def} with our uniform bounds on 
$\|{\sf A}^{(\ell)}_o\|_T$,  $\|{\sf A}^{(\ell)}_K\|_T$,  we conclude that 
\begin{equation}\label{eq:K-der-ell}
\| {\sf A}_K^{(\ell)} \|_T \le \frac{\alpha}{4\ell} \,,  \qquad \qquad \qquad \forall \ell \ge \ell_o \,,
\end{equation}
for some 
$\alpha=\alpha(c,\kappa,\b,T,r_1,\|f_0\|^{(2)}_\star,\|\nu\|^{(3)}_\star,\|\bv\|^{(2)}_\star)$ finite.  Since
\[
\partial_s \mu^{(\ell)} = \partial_s f'_\ell(K^{(\ell)})=[2\ell+ f''_0(K^{(\ell)})] {\sf A}_K^{(\ell)} \,,  
\]
we deduce from \eqref{eq:K-der-ell} that $\{\| \partial_s \mu^{(\ell)}\|_T,  \ell \ge \ell_o\}$ are uniformly 
bounded,  hence $\{\mu^{(\ell)}, \ell \ge \ell_o\}$ is equi-continuous and uniformly bounded.  

To complete the proof,  fix a $\|\cdot\|_T$-limit point
\red{$(\tilde{\mu},\tilde{R},\tilde{C},\tilde{q}) \in \Ca(\Delta_T,\R^4)$}
of the pre-compact sequence
$(\mu^{(\ell)}, R^{(\ell)},C^{(\ell)},q^{(\ell)})$.  While proving
Prop. \ref{prop:nu-cont} we saw that 
the functional \red{$\underline{{\sf A}}=({\sf A}_R,{\sf A}_C,{\sf A}_q,
{\sf A}_H)$} is Lipschitz continuous \abbr{wrt} 
$(\mu,R,C,q) \in \La_{\kappa,r_1}$ equipped with the $\|\cdot\|_T$ norm \red{(which in turn then holds 
for ${\sf A}_K(s)=1+2 {\sf A}C(s,s)$)}.
% for any $\|\cdot\|_T$-limit point $(\mu,R,C,q) \in \La_{\kappa,r_1}$,  
The corresponding sub-sequence of $\underline{\sf A}^{(\ell)}$ converges 
to \red{$\underline{\sf A}[\tilde{\mu},\tilde{R},\tilde{C},\tilde{q}]$} in $(\Ca(\Delta_T,\R^4),\|\cdot\|_T)$ and
such $\red{(\tilde{\mu},\tilde{R},\tilde{C},\tilde{q})}$ \red{must} also satisfy \eqref{eqRs-m}-\eqref{eqqs-m}.
\red{Since $\| {\sf A}_K^{(\ell)} \|_T \to 0$ (see \eqref{eq:K-der-ell}),  necessarily ${\sf A}_K[\tilde{\mu},\tilde{R},\tilde{C},\tilde{q}] \equiv 0$ which translates to $\tilde{K} \equiv 1$ and} by the uniqueness result 
of Prop.  \ref{uniqueness}(b) all those limit points coincide.
\end{proof}

\appendix
\section{\label{sec:Appendix} Conditional mean and Hessian norm:
% given $\cpteq(\Ve)$: 
Proof of Lemma \ref{lem:G-conc}}

We start by collecting a few exact computations which we then utilize to prove Lemma \ref{lem:G-conc}.

\begin{lem}\label{lem:app}
Assume \eqref{eq:nudef}-\eqref{eq:r-star},  excluding $\nu$ which is pure $p$-spins,  using the
coupling of Definition \ref{def:H-m-coupling} with $\nu^{\Delta}(r):= \nu(r)-\nu^{[m]}(r)$.  For 
$\qs \in (0,1]$,  $\bxs  \in \qs \SN$ and $\bxo \in \SN$ such that
$\qo = \frac{1}{N} \langle \bxo,\bxs \rangle :=\alpha \qs$,  recall the events $\cpteq(\Ve)$ 
of Definition \ref{def:ext-cond}.  Setting $\gamma:=\nu'(\qs^2)$,  we then have the following:

\noindent
(a).  The matrix $\Sigma_{\nu}(\qo)$ of \eqref{def:w-s} is positive definite for  any
$\qo \in (-1,1)$,  with $\Sigma_{\nu}^{-1}(\qo)$ bounded and continuous in $\qo$
(\abbr{wrt} the spectral norm $\|\cdot\|_2$),  uniformly away from $|\qo|=1$.

\noindent 
(b).   Set for any $\bx \in \B^N_\star$ 
\begin{equation}\label{def:z-alt}
x := \frac{1}{N} \langle \bx, \bxs \rangle,  \quad y := \frac{1}{N} \langle \bx, \bxo \rangle \,,  \quad 
z :=  \frac{1}{\|\bxs\|} \langle \bx,\bz \rangle \,,
\end{equation}
with $\alpha x + \sqrt{1-\alpha^2}\, \qs^2 z = \qs y$.  For $\underline{\hat{\bv}}_\nu(x,y,z)$ of \eqref{def:vt-new}
and $\Ve=(\hat{V},\bu)$,  $m$,  setting $\underline{w}= \Sigma^{-1}_\nu(\qo) \hat{V}$,   we get
\begin{align}
\bar{H}_{\Ve} (\bx) :=& \E[H_{\bJ}(\bx)\,|\,  \cpteq(\Ve)] = 
\label{eq:EH-cond}
 - N \langle \underline{w}, \underline{\hat{\bv}}_\nu(x,y,z) \rangle 
 - \frac{1}{\gamma} \nu'(x)  \langle\bu M,  \bx \rangle\,, \\
(\bar{H}_{\Ve} - \widehat{H}^{[m]}_{\Ve}) (\bx) =&  \E[H^\Delta_{\bJ}(\bx)\,|\,  \cpteq(\Ve)] =
- N \langle \underline{w},\underline{\hat{\bv}}_{\nu^\Delta}(x,y,z) \rangle
 - \frac{1}{\gamma} (\nu^\Delta)'(x) \langle\bu M,  \bx \rangle\,.
 \label{eq:E-del-cond}
\end{align}
In particular,  when $\Ve = (\Ep,\Eg,\Gg,0,{\bf 0})$,  we get in \eqref{eq:EH-cond} 
the formula $\bar{H}_{\V}(\bx)=-N \bv(x,y)$ of \eqref{eq:bH-V}.  

\noindent
(c).  Represent $\bx=(x^1,\ldots,x^N)$ in the orthonormal basis $\hbxs,\bz,M$ of Definition \ref{def:ext-cond} and
form $\bar \bu \in \R^N$ by amending $\bu =(\bar u^i,  i \ge 3)$ with $\bar u^1:=\sqrt{N}\frac{\gamma}{\qs} w_3$,  
$\bar u^2:=\sqrt{N} \frac{\gamma}{\qs} w_4$.  
Setting ${\bf c}(\alpha) = \alpha [\alpha,\sqrt{1-\alpha^2}]$,  
$\rho(\bx):=\frac{\qs}{\sqrt{N}} x^1$,  $\rho_\alpha(\bx):=\frac{1}{\sqrt{N}}(\alpha x^1 + \sqrt{1-\alpha^2} x^2)$, 
in this basis, 
$\{H_{ij}(\bx)\}_{ij} := - {\rm Hess}(\bar{H}_{\Ve}) (\bx)$ is zero except for
\begin{equation}\label{def:Hess-Hbar}
\begin{aligned}
%H_{ij} (\bx) =& w_1 c_{ij}(\alpha) \nu''(\rho^\alpha) + w_{4} \nu''(\rho) \indic_{i \ne j} \\
%& + \indic_{i=j=1} 
%\Big[ (w_2 \qs^2 + w_3) \nu''(\rho) + \frac{\qs^2 \nu'''(\rho)}{\gamma N} \langle \bar \bu,\bx \rangle \Big] \,, 
%\qquad \quad 1 \le i \le j \le 2 \,,\\
%H_{1j}(\bx) =& H_{j1} (\bx)=
%\frac{\qs}{\gamma} \nu''(\rho) \frac{u^j}{\sqrt{N}},   \qquad \qquad  \qquad \qquad  \qquad \;\;
%\qquad   3 \le j \le N.\\
H_{1j}(\bx) = H_{j1} (\bx) =& w_1 c_{j}(\alpha) \nu''(\rho_\alpha)  \indic_{j \le 2} +
\frac{\qs}{\gamma} \nu''(\rho) \frac{\bar u^j}{\sqrt{N}} + 
\Big[ w_2 \qs^2 \nu''(\rho) + \frac{\qs^2 \nu'''(\rho)}{\gamma N} \langle \bar \bu,\bx \rangle \Big] \indic_{j=1},\\
H_{22} (\bx) =& w_1 (1-\alpha^2) \nu''(\rho_\alpha)\,.
\end{aligned}
\end{equation}
The matrix $\{H^\Delta_{ij}(\bx)\}_{ij} := {\rm Hess}(\hat{H}^{[m]}_{\Ve} - \bar{H}_{\Ve})(\bx)$ follows 
 \eqref{def:Hess-Hbar} with $\nu''$ and $\nu'''$ replaced by $(\nu^\Delta)''$ and $(\nu^\Delta)'''$. 

In case $\qs=0$,  the formulas \eqref{eq:EH-cond} and \eqref{eq:E-del-cond} simplify to
\begin{equation}\label{eq:EH-cond-0}
\bar{H}_{\Ep} (\bx) = -N \Ep \frac{\nu(y)}{\nu(1)},   \qquad \qquad (\bar{H}_{\Ep}-\widehat{H}^{[m]}_{\Ep})(\bx)= -N \Ep \frac{\nu^\Delta(y)}{\nu(1)}\,,
\end{equation}
and having $\frac{1}{\sqrt{N}} \bxo$ as second element of our basis,  yields that 
\begin{equation}\label{def:Hess-Hbar-0}
H_{ij}(\bx)=\frac{E}{\nu(1)} \nu''(\rho_0) \indic_{i=j=2}\,,  \qquad \qquad H^\Delta_{ij}(\bx)=\frac{E}{\nu(1)} (\nu^\Delta)''(\rho_0) \indic_{i=j=2}\,,
\end{equation}
or equivalently,  to set  $\alpha=0$,  $w_1=\frac{E}{\nu(1)}$ and $w_2=0$,  $\bar \bu={\bf 0}$,
in \eqref{def:Hess-Hbar}.
\end{lem}
\begin{proof} Since $\bz \perp \bxs$ with $\sp\{\bxo,\bxs\} \subseteq \sp\{\bz,\bxs\}$ and $\langle \bxo,\bz \rangle \ge 0$,  
by elementary geometry  
\begin{equation}\label{def:v-formula}
\alpha \hbxs + \sqrt{1-\alpha^2} \bz =   \frac{1}{\sqrt{N}} \bxo  \,,
\end{equation}
which yields the stated identity $\alpha x + \sqrt{1-\alpha^2}\,  \qs^2 z = \qs y$ between the terms of \eqref{def:z-alt}.
When $|\alpha|<1$,  inverting this identity yields $z=z(x,y)$ as in \eqref{def:vt-new},  which when substituted 
in \eqref{eq:EH-cond} together with $\bu={\bf 0}$ and $G=0$ results with formula \eqref{eq:bH-V} 
for $\bar{H}_{\V}(\bx)$ (for $|\alpha|=1$ and $G=0$ we get $w_4=0$ in \eqref{def:w-s},  hence $z(x,y)$ irrelevant).
Next recall that $\cpteq(\Ve)
%:=\{ \Ha=\Ve\} :
= \{\hat{\Ha}=\hat{V},   \Ha_\perp = \bu \}$ for
$\Ha_{\perp} = - M \nabla H_{\BJ}(\bxs) \in \R^{N-2}$
and $\hat{\Ha} \in \R^4$ of \eqref{def:calH-c}.  With $M$ forming an orthonormal basis of $\sp\{\bz,\bxs\}^\perp$,  
it follows from \cite[(3.35)]{DS21} that $\hat{\Ha}$ and $\Ha_\perp$ are independent,  
%with $\frac{1}{\sqrt{\gamma}} \Ha_\perp$ a standard
%$(N-2)$-dimensional Gaussian vector 
and that for any $\bx \in \B^N_\star$
\begin{align}\label{eq:Hat-H-cov}
N  \E\big[ \hat{\Ha} \hat{\Ha}^{\top}\big] = \Sigma_\nu(\qo) \, ,\qquad \qquad \quad
\E\big[H_{\bJ}(\bx) \hat{\Ha}\big] &= - \underline{\hat{\bv}}_\nu (x,y,z)^\top,\\
\E \big[\Ha_\perp \Ha_\perp^\top \big] = \nu'(\qs^2) {\bf I}_{N-2} \,,  \qquad\;\;  \E\big[H_{\BJ}(\bx) \Ha_\perp \big] &= - \nu'(x) M \bx \,.
\label{eq:Grad-0-cov}
\end{align}
(a).  From \eqref{eq:Hat-H-cov} we deduce that the
% non-negative definite
matrix $\Sigma_\nu(\qo)$ is positive definite,  hence of bounded inverse,  whenever the four random 
variables $\{H_{\BJ}(\bxo), H_{\BJ}(\bxs), \partial_{\hbxs} H_{\BJ}(\bxs),\,\partial_{\bz} H_{\BJ}(\bxs)\}$ 
are linearly independent.  Excluding pure $p$-spins,  such linear independence applies whenever 
$\{\bxo,\pm \bxs\}$ are distinct points,  namely,  for any $\qo \in (-1,1)$.  With $\nu(\cdot)$ 
real analytic up to $\bar r^2>1$ all entries of $\Sigma_\nu(\qo)$ are uniformly bounded and continuous in $\qo$.
Further,  for $|\qo|<1$ the eigenvalues of $\Sigma^{-1}_\nu(\qo)$ are positive,  with $\|\Sigma^{-1}_\nu(\qo)\|_2<\infty$ 
and all entries of $\Sigma_\nu(\qo)^{-1}$ continuous in $\qo$.  Clearly,  then
$\| \Sigma^{-1}_\nu(\qop)- \Sigma^{-1}_\nu(\qo)\|_2 \to 0$ as $\qop \to \qo$,  
uniformly on compacts (i.e.  away from $|\qo|=1$).\\
(b).  By the well-known formula for conditional Gaussian distributions (see \cite[Pages 10-11]{AT}),
\begin{align*}
\E[H_{\bJ}(\bx)\,|\,  \cpteq(\Ve)] &= \E[H_{\BJ}(\bx) \,|\,  \hat{\Ha}=\hat{V}] + \E[H_{\BJ}(\bx)\,|\,  
\Ha_\perp={\bf u}\,]
\nonumber \\
&= \hat{V} \,  \E\big[ \hat{\Ha} \hat{\Ha}^{\top}\big]^{-1}  \E\big[H_{\bJ}(\bx) \hat{\Ha}\big] 
+  {\bf u} \,  \E \big[\Ha_\perp \Ha_\perp^\top \big]^{-1} \E\big[ H_{\BJ}(\bx) \Ha_\perp \big] \,.
\end{align*}
Plugging the expressions of \eqref{eq:Hat-H-cov}-\eqref{eq:Grad-0-cov} in the preceding formula 
results with \eqref{eq:EH-cond}.  Turning to \eqref{eq:E-del-cond} recall that the random process 
$H_{\BJ}^{\Delta}=H^{[\infty]}_{\BJ}-H_{\BJ}^{[m]}$ which corresponds to the mixture $\nu^\Delta$,  
is independent of $H_{\BJ}^{[m]}$.  With $\hat{\Ha}=\hat{\Ha}^{[m]}+\hat{\Ha}^\Delta$
and $\Ha_\perp=\Ha^{[m]}_\perp+\Ha^\Delta_\perp$,  we deduce that similarly to the \abbr{rhs} of 
\eqref{eq:Hat-H-cov}-\eqref{eq:Grad-0-cov},  
\begin{align}\label{eq:Hat-Delta-cov}
\E\big[H^\Delta_{\bJ}(\bx) \hat{\Ha}\big] &= \E\big[H^\Delta_{\bJ}(\bx) \hat{\Ha}^\Delta \big] =
- \underline{\hat{\bv}}_{\nu^\Delta} (x,y,z)^\top,\\
\E\big[H^\Delta_{\BJ}(\bx) \Ha_\perp \big] &= \E\big[H^\Delta_{\bJ}(\bx) \Ha_\perp^\Delta \big] =
- (\nu^\Delta)'(x) M \bx \,.
\label{eq:Grad-Delta-cov}
\end{align}
We thus arrive at  \eqref{eq:E-del-cond} upon plugging \eqref{eq:Hat-Delta-cov}-\eqref{eq:Grad-Delta-cov} 
in the formula
\[
E[H^\Delta_{\bJ}(\bx)\,|\,   \hat{\Ha}=\hat{V},  \Ha_\perp={\bf u} \,] 
%&= \E[H_{\BJ}(\bx) \,|\,  \hat{\Ha}=\hat{V}] + \E[H_{\BJ}(\bx)\,|\,  \Ha_\perp={\bf u}\,]\nonumber \\ &
= \hat{V} \,  \E\big[ \hat{\Ha} \hat{\Ha}^{\top}\big]^{-1}  \E\big[H^\Delta_{\bJ}(\bx) \hat{\Ha}\big] 
+  {\bf u} \,  \E \big[\Ha_\perp \Ha_\perp^\top \big]^{-1} \,  \E\big[ H^\Delta_{\BJ}(\bx) \Ha_\perp \big]\,.
\]
(c).  We express $\bx \in \B^N_\star$ in the basis $(\hbxs,\bz,M)$ of Definition \ref{def:ext-cond} to get by   
\eqref{def:v-formula},  that $(x,y,z)$ of \eqref{def:z-alt} are then
\begin{equation}\label{eq:xyz-conv-base}
x= \rho(\bx) = \frac{\qs}{\sqrt{N}} x^1 \,,  \qquad y = \rho_\alpha(\bx) = \frac{1}{\sqrt{N}}  \big(\alpha x^1 + \sqrt{1-\alpha^2}\, x^2\big) \,,\qquad z = \frac{1}{\qs \sqrt{N}} x^2 \,.
\end{equation} 
Likewise,  for $\bu=(u^3, \ldots,u^N) \in \R^{N-2}$ we have in this basis that $\langle \bu M,\bx \rangle = \sum_{i \ge 3} u^i x^i$.  Substituting \eqref{eq:xyz-conv-base} as the arguments for $\underline{\hat{\bv}}_\nu(\cdot)$ of \eqref{def:vt-new} 
converts \eqref{eq:EH-cond} into 
\begin{equation}\label{eq:EH-alt}
\bar H_{\Ve}(\bx) = - \Big[ w_1 N \nu (\rho_\alpha)+  w_2  N \nu (\rho) +
\frac{1}{\gamma} \nu'(\rho) \langle \bar \bu, \bx \rangle \Big]\,,
\end{equation}
for $\bar \bu$ as stated.  The function in \eqref{eq:EH-alt} is linear in $(x^3,\ldots,x^N)$ 
with no interaction between those coordinates and $x^2$,  hence the only non-zero terms 
of its Hessian be at $i=1$ or $j=1$ or $i=j=2$ and it is not hard to verify the expressions in 
\eqref{def:Hess-Hbar}.  The formula in this basis for the \abbr{rhs} of \eqref{eq:E-del-cond} be
the same as \eqref{eq:EH-alt} except with $\nu^\Delta$ instead of $\nu$,  with the same 
modification then between $(H_{ij}(\bx))$ and $(H_{ij}^\Delta(\bx))$.

When $\qs=0$ we condition only on $\hat{\Ha}_1$ of variance $N^{-1} \nu(1)$
and with $\cpteo(\Ve)=\cpto(E)$,  the derivation of the \abbr{lhs} of \eqref{eq:EH-cond-0} is a special case
of that of \eqref{eq:EH-cond}.  Its \abbr{rhs} 
%of \eqref{eq:EH-cond-0} 
follows similarly,  using (only) the first coordinate of \eqref{eq:Hat-Delta-cov}.
Having $\frac{1}{\sqrt{N}} \bxo$ as the second vector of the basis,  converts \eqref{eq:EH-cond-0} 
to a special case of \eqref{eq:EH-alt} corresponding to $\alpha=0$,  
$w_1=\frac{E}{\nu(1)}$,  $w_2=0$ and $\bar \bu = {\bf 0}$,
and we read its Hessian formula \eqref{def:Hess-Hbar-0},  as a special case of  \eqref{def:Hess-Hbar}.
\end{proof}

\begin{proof}[Proof of Lemma \ref{lem:G-conc}] 
(a).  Applying  \cite[(C.8) at $k=2$]{BSZ} for the field $r_\star^{-2} H_{\BJ}(r_\star \bx)$,  we see that 
if $\nu(\cdot)$ satisfies \eqref{eq:r-star} with $\bar r>r_\star$,  then for some $\kappa=\kappa(\nu)<\infty$ and $c>0$,   
\begin{equation}\label{eq:BSZ-C8}
 \P\big(\|{\rm Hess}(H_{\BJ}) \|_\infty \ge \kappa \, \big) \le e^{-c N} 
\end{equation}
(see also  \cite[(7.17)]{BSZ} for a relevant definition).  In case $\qs>0$,  using the notations of Definition \ref{def:ext-cond},  
the field $H_{\BJ}^{\qo}$ corresponds to conditioning on $\Ha := (\hat{\Ha},\Ha_\perp)={\bf 0}$,
where $\hat{\Ha} \in \R^4$ is independent of $\Ha_\perp \in \R^{N-2}$.  Hence, 
\[
\bar H_{\Ha} := \E[H_{\BJ} | \Ha ] = \E[H_{\BJ} | \hat{\Ha}] + \E[H_{\BJ}|\Ha_\perp] 
 := \bar H_{\hat{\Ha}} + \bar H_{\Ha_\perp}
\]
(see Lemma \ref{lem:app}(b)),  and we have the identity in law 
\begin{equation}\label{eq:HJ-ident}
H_{\BJ} = H_\BJ^{\qo} + \bar{H}_{\hat{\Ha}} + \bar{H}_{\Ha_\perp}
\end{equation}
where the fields on the \abbr{rhs} are mutually independent.  
Such decomposition as sum of independent fields applies for the 
% centered 
Gaussian vectors $\{{\rm Hess} (H_{\BJ})(\bx_i),  i \le \ell\}$.  As the 
set of all $\ell$-tuples of matrices 
%$\{{\bf H}(\bx_i), i \le \ell\}$
with $\max_{i} \{ \|{\bf H}(\bx_i)\|_2 \} \le \kappa$ is convex and symmetric,  we get from \eqref{eq:BSZ-C8}
by Anderson's inequality (see \cite[Corollary 3]{Anderson}),  that
% this decomposition implies that
\begin{equation}\label{eq:finite-And}
\P\big(\max_{i \le \ell} \{ \|{\rm Hess}(H^{\qo}_{\BJ})(\bx_i)\|_2 \} \ge \kappa \, \big) \le 
\P\big(\|{\rm Hess}(H_{\BJ}) \|_\infty \ge \kappa \, \big) \le e^{-c N} 
\end{equation}
and the \abbr{lhs} of \eqref{Grad-Lip-conc} follows by monotone convergence (with $H_{\BJ}^{\qo} \in C^2_0(\B^N)$, 
taking centers $\{\bx_i,  i \le \ell_k\} \subset \B^N_\star$ of nested dyadic partitions 
of $[-r_\star,r_\star]^N$,  the variable on the \abbr{lhs} of \eqref{eq:finite-And} converges as $k \uparrow \infty$
to $\|{\rm Hess}(H^{\qo}_{\BJ})\|_\infty$).  In case $\qs=0$ the field $H_{\BJ}^{\qo}$ stands for $H_{\BJ}$
conditioned on $\hat{\Ha}_1=0$ and we thus replace $\Ha$ in \eqref{eq:HJ-ident} by $\hat{\Ha}_1$.  Once this 
has been done,  the same argument via Anderson's inequality yields again the \abbr{lhs} of \eqref{Grad-Lip-conc}.
In both cases,  precisely the same reasoning applies for ${\rm Hess}(H^{\qo}_{\BJ}-\hat{H}^{{\qo},[m]}_{\BJ})$,  and it is easy 
to verify that the value of $\kappa_m$ in \cite[(C.8)]{BSZ} for the field on $\B^N(\sqrt{N})$ 
with mixture $r_\star^{-4} \nu^{\Delta}(r_\star^2 \cdot)$,
% in \eqref{eq:nudef},  
indeed goes to zero as $m \to \infty$.

Turning to the \abbr{rhs} of \eqref{Grad-Lip-conc},  we are free to choose 
the basis of $\R^N$ as in Lemma \ref{lem:app}(c) (the value of 
$\|{\rm Hess}(\cdot)\|_\infty$ is invariant to the choice of basis
since Hessian matrices for different 
% orthogonal 
bases are conjugate,  hence having the same spectral norm,  while 
$\B^N_\star$ is invariant to the choice of basis).  Having done so,  as $r_\star<\bar r$, 
\begin{equation}\label{def:S}
S_\nu :=\sup_{r \le r_\star} \{ |\nu''(r)| \vee |\nu'''(r)| \} < \infty 
\end{equation}
and with $|\rho| \vee |\rho_\alpha| \le \frac{\|\bx\|}{\sqrt{N}}$,  it follows from \eqref{def:Hess-Hbar}, 
by Cauchy-Schwarz and the bounds $\|\bx\| \le r_\star \sqrt{N}$,  $\qs \le 1$,  $\|{\bf c}(\alpha)\| = |\alpha| \le 1$,
that for $\qs \in (0,1]$,  uniformly over $\bx \in \B^N_\star$
\begin{align}
\|{\bf H}(\bx)\|^2_{2} \le \|{\bf H}(\bx)\|^2_{\rm F} = \sum_{i,j=1}^N  H_{i,j}(\bx)^2 
& \le 4 (1+r_\star^2) \Big[  \frac{1}{\gamma^2 N} \|\bu\|^2  +  \|\underline{w}\|^2 \Big] S_\nu^2 \,.
\label{eq:S-bd}
\end{align}
As we saw in deriving \eqref{def:Hess-Hbar-0},  this also applies for $\qs=0$ upon setting 
$w_2=\frac{E}{\nu(1)}$ and all other terms being zero.
We also similarly arrive at the bound
\begin{equation}\label{eq:S-del-bd}
\|{\bf H}^\Delta\|_\infty^2 \le 4 (1+r_\star^2) \Big[  \frac{1}{\gamma^2 N} \|\bu\|^2  +  \|\underline{w}\|^2 \Big] (S_{\nu^\Delta})^2 \,.
\end{equation}
Next,  recall from Lemma \ref{lem:app}(a) that for $\qs>0$ and any $\delta>0$,  
\begin{equation}\label{def:C-delta}
C_{\nu,\delta} := \sup\{\|\Sigma^{-1}_{\nu}(\qo) \|_2 : |\qo| \le 1-\delta\} < \infty \,.
\end{equation}
From \eqref{eq:bar-B} we know that $\bu={\bf 0}$ throughout $\vB_\delta(\V)$,
while in case $\qs>0$,  if $1-|\qo| \ge 2 \delta$,  then by \eqref{def:C-delta} also, 
\begin{equation}\label{eq:bd-w-ball}
\| \underline{w} \| = \| \Sigma^{-1}_\nu(\qop) \hat{V} \|
\le \|\Sigma^{-1}_{\nu}(\qop) \|_2 \|\hat{V}\| \le C_{\nu,\delta} (\|(\Ep,\Eg,\Gg)\| + \sqrt{3}\,  \delta) 
\end{equation}
(for $\qs=0$,  merely set $C=\frac{1}{\nu(1)}$).
Combining \eqref{eq:S-bd} and \eqref{eq:bd-w-ball} establishes the \abbr{rhs} of \eqref{Grad-Lip-conc} 
(say,  for $\kappa = 2 (1+r_\star) C_{\nu,\delta} (\|(\Ep,\Eg,\Gg)\| + \sqrt{3} \,  \delta) S_\nu$).
Moreover,  considering \eqref{eq:S-del-bd} instead of \eqref{eq:S-bd} and noting that 
$S_{\nu^\Delta} \to 0$ as $m \to \infty$,  we conclude that for any $\delta \le (1 - |\qo|)/2$,  and some 
$\kappa_m=\kappa_m(\delta) \to 0$,  as claimed
\[
\sup_N \sup_{\vB_\delta(\V)} \{ \|{\rm Hess} (\bar H_{\V'}-\hat{H}^{[m]}_{\V'}) \|_\infty \} \le \kappa_m \,.
\]
(b).  Turning to \eqref{eq:Grad-CE-cont},  assume \abbr{wlog} that $\epsilon \le \delta := \frac{1}{2}(1-|\qo|)$
and consider first ${\bf H}_{\qop}:={\rm Hess}(\bar H_{\Ve'} - \bar H_{\V'})$ with 
$\V'=(\Ep,\Eg,\Gg,\qop)$.  It follows by the linearity of the Hessian and \eqref{eq:S-bd} that 
\begin{align*}
\|{\bf H}_{\qop}\|_\infty^2 &\le 4 (1+r_\star^2) \Big[  \frac{1}{\gamma^2 N} \|\bu'\|^2  +  
\|\underline{\Delta w}\|^2 \Big] S_\nu^2 \,,
\end{align*}
while by \eqref{def:vBp} we have that $\|\bu'\| \le \epsilon \sqrt{N}$ and if $\qs>0$,  then similarly 
to \eqref{eq:bd-w-ball},  
\begin{align*}
\| \underline{\Delta w} \| = \|  \Sigma^{-1}_{\nu}(\qop) [\Ep'-\Ep,\Eg'-\Eg,\Gg'-\Gg,G']^{\top} \| 
\le 2 \epsilon C_{\nu,\delta} 
\end{align*}
(with the same conclusion for $\qs=0$,  see after \eqref{eq:bd-w-ball}).
Combining the last three inequalities,  we conclude that 
$\|{\bf H}_{\qop}\|_\infty \le 2 (1+r_\star) (\gamma^{-1} + 2 C_{\nu,\delta}) \epsilon S_\nu$.  
If $\qs=0$ we are done,  while otherwise,  set $\epsilon'=\frac{\epsilon}{\qs}$ and 
consider $\widetilde{{\bf H}}:={\rm Hess}(\bar H_{\V'} - \bar H_{\V})$.  The only 
difference between the fields is that $\qop \in [\qo-\epsilon,\qo+\epsilon] \cap [-\qs,\qs]$
in $\V'$,  hence $\alpha'=\frac{\qop}{\qs} \in [\alpha-\epsilon',\alpha+\epsilon'] \cap [-1,1]$ 
in the corresponding Hessian entries of \eqref{def:Hess-Hbar}.  Writing $\rho=\qs \rho^1$ and 
$\rho_\alpha:=\alpha \rho^1 + \sqrt{1-\alpha^2} \, \rho^2$ with 
$\rho^i (\bx) = \frac{x^i}{\sqrt{N}} \in [-r_\star,r_\star]$,  we utilize again the linearity of the 
Hessian,  to deduce from \eqref{def:Hess-Hbar} that $\widetilde{{\bf H}}$ is a
\emph{two dimensional} matrix (as ${\bf u}={\bf 0}$ in both fields),  with entries of the form
\[
H_{ij}(\rho^1,\rho^2,\alpha',\underline{w}(\qop))-H_{ij}(\rho^1,\rho^2,\alpha,\underline{w}(\qo))\,,   \qquad 
1 \le i \le j \le 2 \,,
\]
for some continuous $H_{ij} : [-r_\star,r_\star]^2 \times [-1,1] \times \R^4 \to \R$,
independent of $N$.  Moreover,  by Lemma \ref{lem:app}(a),  
\[
\| \underline{w}(\qop)-\underline{w}(\qo) \| \le 
\| \Sigma^{-1}_\nu(\qop) - \Sigma^{-1}_\nu(\qo) \|_2\,  \| (\Ep,\Eg,\Gg) \| \to 0,  \quad \hbox{as} \quad \epsilon \to 0\,, 
\]
uniformly over $|\qop-\qo| \le \epsilon$,  $|\qop| \le \qs$.  It thus follows that
$\sup\{ \|\widetilde{\bf H}\|_\infty : |\qop-\qo| \le \epsilon,  |\qop| \le \qs  \} \to 0$ as $\epsilon \to 0$
(uniformly in $N$).  To get \eqref{eq:Grad-CE-cont},  simply note that by the triangle inequality,  
\[
 \| {\rm Hess}(\bar H_{\Ve'} - \bar H_{\V})\|_\infty \le \|{\bf H}_{\qop}\|_\infty + \|\widetilde{\bf H}\|_\infty \,.
\]
(c).  For $\qs>0$,  condition both sides of \eqref{eq:HJ-ident} on $\{\Ha_\perp={\bf 0}\}$ to cancel the right most term and
get the coupling
\begin{equation}\label{eq:HJ-perp}
H_{\BJ}^\perp = H_\BJ^{\qo} + \bar{H}_{\hat{\Ha}} \,,
\end{equation}
where the fields on the \abbr{rhs} are independent of each other,  and $H_{\BJ}^\perp$ stands
for $H_{\BJ}$ conditioned only on $\Ha_\perp={\bf 0}$.  With $\HJs$ denoting
the field $H_{\BJ}$ conditioned on $(\hat{H}_4,\Ha_\perp)={\bf 0}$,  
taking $\hat{\Ha}_4$ and $H_{\BJ}^\star$ instead of $\hat{\Ha}$ and $H_{\BJ}^{\qo}$,  respectively,  
leads by the same reasoning to the coupling
\begin{equation}\label{eq:HJ-perp2}
H_{\BJ}^\perp = \HJs + \bar{H}_{\hat{\Ha}_4} 
\end{equation}
(with $\HJs$ independent of $\hat{\Ha}_4$).  From \eqref{eq:HJ-perp}-\eqref{eq:HJ-perp2}
we deduce by the triangle inequality for $\|{\rm Hess} (\cdot)\|_\infty$ and the union bound,  the 
existence of a coupling of $(H_{\BJ}^{\qo},H_{\BJ}^\star)$ such that 
\begin{equation}\label{eq:HO-to-Ho}
\P\big( \|{\rm Hess}(H^{\qo}_{\BJ}-\HJs)\|_\infty \ge 2\delta\big) \le 
%\P\big( \|{\rm Hess}(H^{\qo}_{\BJ}-H_{\BJ}^\perp)\|_\infty \ge \delta\big) +
%\P\big( \|{\rm Hess}(H_{\BJ}^\perp-H_{\BJ}^\star)\|_\infty \ge \delta\big) = 
\P\big( \|{\rm Hess}(\bar H_{\hat{\Ha}})\|_\infty \ge \delta\big) +
\P\big( \|{\rm Hess}(\bar H_{\hat{\Ha}_4})\|_\infty \ge \delta\big) \,.
\end{equation}
It thus suffices for \eqref{Grad-Lip-o} to establish the $e^{-\Omega(N)}$-decay of each of the terms on the \abbr{rhs}
of \eqref{eq:HO-to-Ho},  uniformly away from $|\qo|=1$.  To this end,  following the derivation of \eqref{eq:EH-cond},  
we find that 
\begin{align*}
\bar{H}_{\hat{\Ha}} = - N \langle \underline{w},  \underline{\hat{\bv}}_\nu(x,y,z) \rangle \quad \hbox{and} \quad
\bar{H}_{\hat{\Ha}_4} = - N \tilde{w}_4 z \nu'(x) \,,
\end{align*}
for the random variables $\underline{w} := \Sigma_\nu^{-1}(\qo) \hat{\Ha}$
and $\tilde{w}_4 := \frac{\qs^2}{\nu'(\qs^2)} \hat{\Ha}_4$.  Combining \eqref{eq:S-bd} (at $\bu={\bf 0}$),
with the left inequality of \eqref{eq:bd-w-ball},  we deduce that 
$ \|{\rm Hess}(\bar H_{\hat{\Ha}})\|_\infty \le c\,  \|\hat{\Ha}\|$ for some $c=c(\delta,\nu)<\infty$ and  
% $c = 2(1+r_\star) C_{\nu,\delta} S_\nu$ 
by the same reasoning also $ \|{\rm Hess}(\bar H_{\hat{\Ha}_4})\|_\infty \le c \,  |\hat{\Ha}_4|$.
The \abbr{rhs} of \eqref{eq:HO-to-Ho} is thus further bounded by
\begin{equation}\label{eq:temp}
\P\Big(  \|\hat{\Ha}\| \ge \frac{\delta}{c} \Big) +
\P\Big( |\hat{\Ha}_4| \ge \frac{\delta}{c} \Big) \le 2 \sum_{i=1}^4 \P\Big( |\hat{\Ha}_i| \ge \frac{\delta}{2c} \Big) \,.
\end{equation}
Recall from the left-side of \eqref{eq:Hat-H-cov},  that the centered Gaussian variables $\{ \sqrt{N} \,  \hat{\Ha}_i\}$ have 
uniformly bounded variances.
% (which are the diagonal entries of $\Sigma_\nu(\qo)$).  
Consequently,  the terms on the \abbr{rhs} of \eqref{eq:temp} decay 
exponentially in $N$,  establishing \eqref{Grad-Lip-o}.   

Finally,  if $\qs=0$ then $\bxs={\bf 0}$,  so $H_{\BJ}^\star=H_{\BJ}$.  As mentioned before,  in this case  
$H_{\BJ}^{\qo}$ is merely $H_{\BJ}$ given $\{\hat{\Ha}_1=0\}$,  yielding the 
explicit coupling $H_{\BJ}^\star=H_{\BJ}^{\qo} + \bar H_{\hat{\Ha}_1}$.  Analogously to \eqref{eq:HO-to-Ho} 
we now control \eqref{Grad-Lip-o} by 
\[
\P\big(\|{\rm Hess}(\bar H_{\hat{\Ha}_1})\|_\infty \ge 2 \delta\big) \le \P\Big(|\hat{\Ha}_1| \ge \frac{2\delta}{c}\Big)\,,
\]
where by \eqref{def:Hess-Hbar-0} we set $c=\frac{S_\nu}{\nu(1)}<\infty$,  and conclude as before.
\end{proof}

\end{document}